\documentclass[10pt]{amsart}

\usepackage{hyperref}
\usepackage{amsfonts}
\usepackage{amssymb}
\usepackage{amsmath}
\usepackage{amsthm}
\usepackage{enumerate}
\usepackage{mathrsfs}
\usepackage{tikz}

\newcommand{\mathscripty}{\mathscr}

\newcommand{\NN}{\mathbb{N}}

\newcommand{\er}{\mathbb{R}}

\newcommand{\Cstar}{\mathrm{C}^*}

\newcommand{\SI}{\mathscripty{I}}
\newcommand{\SJ}{\mathscripty{J}}

\newcommand{\SP}{\mathscripty{P}}

\newcommand{\ZFC}{\mathrm{ZFC}}

\newcommand{\CH}{\mathrm{CH}}
\newcommand{\MA}{\mathrm{MA}}

\newcommand{\OCA}{\mathrm{OCA}}
\newcommand{\PFA}{\mathrm{PFA}}

\newtheorem{theorem}{Theorem}[section]
\newtheorem*{theorem*}{Theorem}
\newtheorem{proposition}[theorem]{Proposition}
\newtheorem*{proposition*}{Proposition}
\newtheorem{lemma}[theorem]{Lemma}
\newtheorem*{lemma*}{Lemma}

\newtheorem*{corollary*}{Corollary}

\newtheorem*{fact*}{Fact}
\theoremstyle{definition}
\newtheorem{definition}[theorem]{Definition}
\newtheorem*{definition*}{Definition}
\newtheorem{claim}[theorem]{Claim}

\newtheorem*{claim*}{Claim}
\newtheorem{conjecture}[theorem]{Conjecture}
\newtheorem*{conjecture*}{Conjecture}
\newtheorem{notation}[theorem]{Notation}
\newtheorem{theoremi}{Theorem}

\newtheorem{questioni}[theoremi]{Question}
\newtheorem{conjecturei}{Conjecture}

\newtheorem*{OCAinfty}{$\mathbf{OCA_\infty}$}

\theoremstyle{remark}

\newtheorem*{example*}{Example}
\newtheorem{remark}[theorem]{Remark}
\newtheorem*{remark*}{Remark}

\newtheorem*{note*}{Note}
\newtheorem{question}[theorem]{Question}
\newtheorem*{question*}{Question}

\newcommand{\set}[2]{\left\{#1\mathrel{}\middle|\mathrel{}#2\right\}}

\newcommand{\norm}[1]{\left\lVert #1 \right\rVert}

\DeclareMathOperator{\intr}{int}

\DeclareMathOperator{\supp}{supp}

\DeclareMathOperator{\Fin}{Fin}

\DeclareMathOperator{\Aut}{Aut}

\DeclareMathOperator{\Ad}{Ad}

\usepackage{enumitem}

\title{Rigidity conjectures for continuous quotients}%

\author[A. Vignati]{Alessandro Vignati}
\address[A. Vignati]{Institut de Math\'ematiques de Jussieu - Paris Rive Gauche (IMJ-PRG)\\
UP7D - Campus des Grands Moulins\\
B\^atiment Sophie Germain\\
8 Place Aur\'elie Nemours\\ Paris, 75013, France.
}
\email{ale.vignati@gmail.com}
\urladdr{http://www.automorph.net/avignati}

\thanks{The author was partially supported by a Prestige co-fund program while in Paris and was supported by a FWO scholarship. This research is  partially funded by ANR  Project AGRUME (ANR-17-CE40-0026)}

\date{\today}%
\begin{document}

\begin{abstract}
We prove several rigidity results for corona $\Cstar$-algebras and \v{C}ech-Stone remainders under the assumption of Forcing Axioms. In particular, we prove that a strong version of Todor\v{c}evi\'c's $\OCA$ and Martin's Axiom at level $\aleph_1$ imply: (i) that if $X$ and $Y$ are locally compact second countable topological spaces, then all homeomorphisms between $\beta X\setminus X$ and $\beta Y\setminus Y$ are induced by homeomorphisms between cocompact subspaces of $X$ and $Y$; (ii) that all automorphisms of the corona algebra of a separable $\Cstar$-algebra are trivial in a topological sense; (iii) that if $A$ is a unital separable infinite-dimensional $\Cstar$-algebra, the corona algebra of $A\otimes \mathcal K(H)$ does not embed into the Calkin algebra. All these results do not hold under the Continuum Hypothesis.

\vspace{10pt}

 Nous prouvons plusieurs r\'esultats de rigidit\'e pour les $\Cstar$-alg\`ebres de couronnes et les restes de \v{C}ech-Stone sous l'hypoth\`ese des axiomes du forcing. En particulier, nous prouvons qu'une version forte de l'axiome du coloration ouvert de Todor\v{c}evi\'c et l'axiome de Martin au niveau $\aleph_1$ impliquent~: (i) que si $X$ et $Y$ sont des espaces topologiques localement compacts et \`a base d\'enombrable, alors tous les hom\'eomorphismes entre $\beta X \setminus X$ et $\beta Y \setminus Y$ sont induits par des hom\'eomorphismes entre sous-espaces cocompacts de $X$ et $Y$~; (ii) que tous les automorphismes des $\Cstar$-alg\`ebres de couronne s\'eparables sont triviaux en un sens topologique~; (iii) que si $A$ est une $\Cstar$-alg\`ebre de dimension infinie unitaire et s\'eparable, l'alg\`ebre couronne de $A\otimes \mathcal K(H)$ ne se plonge pas dans l'alg\`ebre de Calkin. Tous ces r\'esultats sont en d\'efaut sous l'hypoth\`ese du continu.
 \end{abstract}
\maketitle

\setcounter{tocdepth}{1}
\tableofcontents
\section{Introduction}

This paper is squarely placed in the interface between set theory and operator algebras, with feedback into general topology. Its main goal is to study automorphisms of corona $\Cstar$-algebras.

We begin by introducing the general framework. Let $X$ and $Y$ be two Polish spaces carrying an algebraic structure compatible with the topology. Let $I\subseteq X$ and $J\subseteq Y$ be two definable (e.g., Borel, or analytic) substructures inducing quotients $X/I$ and $Y/J$.\footnote{An operator algebraist should think at $X=\mathcal M(A)$ and $Y=\mathcal M(B)$, the multiplier algebras of two nonunital $\Cstar$-algebras $I=A$ and $J=B$.} Can we characterise the isomorphisms between the quotients $X/I$ and $Y/J$ in terms of the topological, or algebraic, structure of $X$, $Y$, $I$, and $J$? If $X/I$ and $Y/J$ are isomorphic,  what are the relations between $X$ and $Y$, and $I$ and $J$? 

Given an isomorphism $\phi\colon X/I\to Y/J$, one searches for a map $\Phi\colon X\to Y$ preserving some topological or algebraic structure, while making the following diagram commute:
\begin{center}
\begin{tikzpicture}
\node  (A1) at (2,5) {$X$};
\node (A2) at (2,3) {$X/I$};
\node  (B1) at (5,5){$Y$};
\node (B2) at (5,3) {$Y/J$}; 
\draw (A1) edge[->] node [below] {$\Phi$} (B1) ;
\draw (A1) edge[->] node [left] {} (A2) ;
\draw (A2) edge[->] node [above] {$\phi$} (B2) ;
\draw (B1) edge[->] node [left] {} (B2) ;
\end{tikzpicture}
\end{center}
An isomorphism $\phi\colon X/I\to Y/J$ is
\begin{itemize}
\item topologically trivial if there is a lifting map $\Phi\colon X\to Y$ preserving (some of) the topological structure;
\item algebraically trivial if there is a lifting map $\Phi\colon X\to Y$ preserving (some of) the algebraic structure.
\end{itemize} 
The quotients $X/I$ and $Y/J$ are topologically (algebraically) isomorphic if there is a topologically (algebraically) trivial isomorphism $X/I\to Y/J$. 
While the technical definition of (topologically or algebraically) trivial depends on the category one is interested in (e.g. Boolean algebras, groups, rings, $\Cstar$-algebras,...), there are two questions common throughout different categories.
\begin{questioni}\label{ques:main}
\begin{enumerate}[label=(\alph*)]
\item Are all isomorphisms $X/I\to Y/J$ topologically trivial?
\item Are all topologically trivial isomorphisms algebraically trivial?
\end{enumerate}
\end{questioni}
Shoenfield's Absoluteness Theorem (\cite[Theorem 13.15]{Kana})  implies that answers to Question~\ref{ques:main}(b) cannot be changed by forcing. On the contrary, the existence of isomorphisms that are not topologically trivial is in many cases subject to the axioms in play; more than that, thanks to Woodin's $\Sigma_1^2$-absoluteness (\cite{Woodin.Abs}), the Continuum Hypothesis\footnote{$\CH$ is the statement $2^{\aleph_0}=\aleph_1$, the latter being the first uncountable cardinal.} $\CH$ provides the optimal set-theoretic assumption for obtaining isomorphisms that are not topologically trivial. On the other hand, for large classes of quotients it is consistent with $\ZFC$ that all isomorphisms between the objects in consideration are topologically trivial. Often, these `rigidity' phenomena happen in presence of Forcing Axioms\footnote{Forcing Axioms can be viewed as higher versions of Baire's category theorem, see~\S\ref{subsec:ST}.}. One of the most important among Forcing Axioms is Shelah's Proper Forcing Axiom $\PFA$ (\cite{Shelah.PF}), which historically provides the environment for rigidity results, although there is no metatheorem showing that $\PFA$ (or some of its consequences) is the optimal environment for the statement `all isomorphisms are topologically trivial' to hold. We will prove our rigidity results using two Forcing Axioms which are consequences of $\PFA$, the axioms $\OCA$ and $\MA_{\aleph_1}$. Both axioms are in contradiction with $\CH$; differently from $\PFA$ whose consistency needs large cardinals, the consistency of $\OCA+\MA_{\aleph_1}$ can be proved within $\ZFC$ (\cite[\S2]{Velickovic.OCA1}).

The first concrete case of study was the one of quotients of the Boolean algebra $\mathcal P(\NN)$, considered with its natural Polish topology. An isomorphism between quotients $\mathcal P(\NN)/\SI$ and $\mathcal P(\NN)/\SJ$, for definable ideals $\SI,\SJ\subseteq\mathcal P(\NN)$, is algebraically trivial if it is induced by a Boolean algebra endomorphism of $\mathcal P(\NN)$, and topologically trivial if it admits a Baire-measurable lifting, without the requirement of preserving any algebraic operation. Thanks to Ulam stability phenomena for finite Boolean algebras, see~\cite[\S1.8]{Farah.AQ}, if $\SI$ is the ideal of finite sets $\Fin$, the notions of topologically and algebraically trivial automorphisms coincide, and we will simply refer to trivial automorphisms. (This is not yet proved to be the case for all quotients of $\mathcal P(\NN)$ by a Borel ideal. For more on this, see \cite{Farah.AH}.)

The question on whether all automorphisms of $\mathcal P(\NN)/\Fin$ are trivial arose from the study of the homogeneity properties of the space $\beta\NN\setminus\NN$. Rudin (\cite{Rudin}) showed that $\CH$ implies the existence of nontrivial automorphisms of $\mathcal P(\NN)/\Fin$. On the other hand, Shelah proved in \cite{Shelah.PF} that the assertion `all automorphisms of $\mathcal P(\NN)/\Fin$ are trivial' is consistent with $\ZFC$. This was later shown to be implied by $\PFA$ in \cite{Shelah-Steprans.PFA}. This assumption was weakened to that of $\OCA$ and $\MA_{\aleph_1}$ in \cite{Velickovic.OCA}. Extending the focus  to quotients by more general Borel ideals, the eye-opening work of Farah \cite{Farah.AQ} showed that in presence again of $\OCA$ and $\MA_{\aleph_1}$, in many cases\footnote{Assumptions on $\SI$ or $\SJ$ should be given here, see~\cite[\S3.4]{Farah.AQ}}, all isomorphisms between quotients $\mathcal P(\NN)/\SI$ and $\mathcal P(\NN)/\SJ$ are algebraically trivial. In this case, the quotients are isomorphic if and only if $\SI$ and $\SJ$ are isomorphic themselves. This fails in many situations in presence of $\CH$: in \cite{JustK} it was showed that all quotients over $F_\sigma$ ideals are countably saturated. This, together with a Cantor's back-and-forth argument, shows that under $\CH$ all quotients over $F_\sigma$ ideals are pairwise isomorphic, and additionally each such quotient has $2^{2^{\aleph_0}}$ distinct automorphisms that are not topologically trivial. (For more information on the situation in the Boolean algebra setting, see \cite{Farah-Shelah.TA, Farah.RC})

Such a general approach was undertaken in the study of other quotient structures such as groups and semilattices, sometimes with limited success (\cite{Farah.LHQS}). The next category where this way of thinking proved to be successful is the one of $\Cstar$-algebras, the focus being particular quotients known as  corona algebras, the forefather being the Calkin algebra $\mathcal Q(H)$. Since abelian $\Cstar$-algebras correspond, by a controvariant equivalence of categories, to locally compact topological spaces, the study of $\Cstar$-algebras is often viewed as `noncommutative topology'. Following this philosophy, one sees $\mathcal B(H)$, the algebra of bounded linear operators on a separable complex Hilbert space $H$, as the noncommutative analog of the algebra $\ell_\infty(\NN)$, and $\mathcal K(H)$, the ideal of compact operators on $H$, as the noncommutative $c_0(\NN)$. The Calkin algebra $\mathcal Q(H)=\mathcal B(H)/\mathcal K(H)$ is then thought as the noncommutative analog of $\ell_\infty(\NN)/c_0(\NN)$. Note that the Boolean algebras of projections in $\ell_\infty(\NN), c_0(\NN)$, and $\ell_\infty(\NN)/c_0(\NN)$ are, respectively, isomorphic to $\mathcal P(\NN)$, $\Fin$, and $\mathcal P(\NN)/\Fin$. (For the many parallels between $\mathcal Q(H)$ and $\mathcal P(\NN)/\Fin$ see \cite{FHV.Calkin, FKV.Calkin,Wofsey,Zamora}.)

The interest in the study of automorphisms of $\mathcal Q(H)$ has operator algebraic roots: while developing the connections between algebraic topology, index theory, and operator algebras, in introducing extensions and $K$-homology for $\Cstar$-algebras, Brown, Douglas, and Fillmore asked in \cite{BDF.Ext}  whether there can exist an outer automorphism of $\mathcal Q(H)$. Their main question, whether a $K$-theory reverting (therefore necessarily outer) automorphism of $\mathcal Q(H)$  can exist, has yet to be given a full answer, but it was shown that the existence of an outer automorphism depends on the set theoretical axioms in use. In \cite{Phillips-Weaver} Phillips and Weaver showed that outer automorphisms exist if one assumes $\CH$, while Farah in \cite{Farah.C} proved that under $\OCA$ all automorphisms of $\mathcal Q(H)$ are inner. 
This intuition was later generalised to the setting of more general corona $\Cstar$-algebras, leading to our notion of triviality. If $A$ is a nonunital $\Cstar$-algebra one constructs the multiplier algebra of $A$, $\mathcal M(A)$, and its corona algebra $\mathcal Q(A)=\mathcal M(A)/A$\footnote{If $A=\mathcal K(H)$ then $\mathcal Q(A)=\mathcal Q(H)$.}. These correspond to the \v{C}ech-Stone compactification and remainder of a locally compact space. $\mathcal M(A)$ is the largest unital algebra in which $A$ sits as a `dense' ideal; it carries a natural topology known as the strict topology, which makes the unit ball of $\mathcal M(A)$ a Polish space in case $A$ is separable (see~\S\ref{subsec:cstartrivial}).

\begin{definition}[\cite{Coskey-Farah}] Let $A$ and $B$ be separable $\Cstar$-algebras. An isomorphism $\Lambda\colon\mathcal Q(A)\to\mathcal Q(B)$ is topologically trivial\footnote{Coskey and Farah call these isomorphisms trivial, and do not make any distinction between topologically and algebraically trivial.} if its graph 
\[
\Gamma_\Lambda=\{(a,b)\in\mathcal M(A)\times\mathcal M(B)\colon \Lambda(\pi_A(a))=\pi_B(b)\}
\]
is Borel in strict topology when restricted to pairs of contractions, $\pi_A$ and $\pi_B$ being the quotient maps from the multiplier algebras to the coronas. 
\end{definition}
Since a Polish space has only $2^{\aleph_0}$ Borel subsets, there can be only $2^{\aleph_0}$ topologically trivial isomorphisms between $\mathcal Q(A)$ and $\mathcal Q(B)$. 
\begin{conjecturei}[\cite{Coskey-Farah}]\label{conj1}
Let $A$ and $B$ be a separable nonunital $\Cstar$-algebras. Then
\begin{itemize}
\item $\CH$ implies that if $\mathcal Q(A)$ and $\mathcal Q(B)$ are isomorphic, then there are isomorphisms between $\mathcal Q(A)$ and $\mathcal Q(B)$ that are not topologically trivial;
\item $\PFA$ implies that all isomorphisms between $\mathcal Q(A)$ and $\mathcal Q(B)$ are topologically trivial.
\end{itemize}
\end{conjecturei}

The conjecture was confirmed for larger and larger classes of algebras, as summarised by the following table. ($A$ here is always a separable nonunital $\Cstar$-algebra). $\PFA$ can be weakened to $\OCA+\MA_{\aleph_1}$.

\vspace{7pt}

\begin{center}
\begin{tabular}{|c|c|c|}\hline
 & $\CH\Rightarrow \exists$ automorphisms &$\OCA+\MA_{\aleph_1}\Rightarrow $ all\\ 
 &of $\mathcal Q(A)$ that are not&automorphisms of $\mathcal Q(A)$\\
& topologically trivial &are topologically trivial\\ \hline\hline
 $A=\mathcal K(H)$&\cite{Phillips-Weaver}&\cite{Farah.C}\\\hline
 $A=B\otimes\mathcal K$, $B$ unital &\cite{Coskey-Farah}&partially \cite{MKAV.FA}\\\hline
 $A$ simple&\cite{Coskey-Farah}& Theorem~\ref{thmi:alltrivialnoncomm}\\\hline
 $A=\bigoplus A_n$, $A_n$ unital &\cite{Coskey-Farah} &partially \cite{MKAV.FA}\\
 \hline
 full generality & Unknown&Theorem~\ref{thmi:alltrivialnoncomm}\\\hline 
\end{tabular}
\end{center}

\vspace{7pt}

 More results were obtained in \cite{McKenney.UHF, V.Nontrivial}, while \cite{Ghasemi.FDD} contains consistency results. The results in \cite{MKAV.FA}, in addition to the hypotheses listed above, require the algebra $A$ to satisfy a weak form of finite-approximation property, known as the Metric Approximation Property of Grothendieck (\cite{Groth}). Such hypotheses are not needed in these notes, as we fully solve the `rigidity' part of the conjecture.

\begin{theoremi}\label{thmi:alltrivialnoncomm}
Assume $\OCA$ and $\MA_{\aleph_1}$, and let $A$ and $B$ be separable $\Cstar$-algebras. Then, all isomorphisms between $\mathcal Q(A)$ and $\mathcal Q(B)$ are topologically trivial.
\end{theoremi}


Although important from a set-theoretical point of view, the study of when automorphisms of coronas are topologically trivial is not satisfactory from an algebraist' point of view. Since there are coronas admitting in $\ZFC$ outer automorphisms (any outer automorphism of the Cuntz algebra $\mathcal O_2$ induces an outer automorphism of $\mathcal Q(c_0(\mathcal O_2))$), and automorphisms of $\mathcal Q(H)$ that do not admit a lift $\mathcal B(H)\to\mathcal B(H)$ that is a $^*$-homomorphism (consider an inner automorphism implement by a unitary of nonzero Fredholm index), the proper notion of algebraically trivial isomorphism between coronas has to be looser than `being inner' and than `admitting a lifting that is a $^*$-homomorphism'. An algebraic notion of triviality which works for all coronas of separable $\Cstar$-algebras was not given before. We get inspiration from the abelian setting.

If $A=C_0(X)$ for a locally compact space $X$, then $\mathcal M(A)=C_b(X)=C(\beta X)$ and $\mathcal Q(A)=C(X^*)$ where $X^*=\beta X\setminus X$. ($\beta X$ is the \v{C}ech-Stone compactification of $X$). By Gelfand's duality, automorphisms of abelian coronas correspond to homeomorphisms of \v{C}ech-Stone remainders.
Fix locally compact second countable spaces $X$ and $Y$. Following \cite{Farah-Shelah.RCQ}, a homeomorphism $\phi\colon X^*\to Y^*$ is trivial\footnote{In \cite{Farah-Shelah.RCQ} these were defined as, simply, trivial.} if there are compact sets  $K_X\subseteq X$ and $K_Y\subseteq Y$ and a homeomorphism $\Phi\colon X\setminus K_X\to Y\setminus K_Y$ such that $\beta\Phi=\phi$ on $X^*$. 

In case $X=\NN$, trivial homeomorphisms of $\NN^*$ correspond to almost permutations of $\NN$, and therefore to trivial automorphisms of $\mathcal P(\NN)/\Fin$. The following conjecture was made in~\cite{Farah-Shelah.RCQ}.
\begin{conjecturei}\label{conj2}
Let $X$ and $Y$ be second countable locally compact noncompact spaces. Then
\begin{itemize}
\item $\CH$ implies that if $X^*$ and $Y^*$ are homeomorphic, then there are nontrivial homeomorphisms between $X^*$ and $Y^*$;
\item $\PFA$ implies that all homeomorphisms between $X^*$ and $Y^*$ are trivial.
\end{itemize}
\end{conjecturei}
As for topologically trivial isomorphisms of general coronas, the following table summarizes the results obtained so far. ($X$ always denote a locally compact noncompact second countable topological space). Once again, $\PFA$ can be weakened to $\OCA+\MA_{\aleph_1}$.

\begin{center}

\begin{tabular}{|c|c|c|}\hline
 & $\CH\Rightarrow \exists$ nontrivial & $\OCA+\MA_{\aleph_1}\Rightarrow X^*$ has only\\
 &homeomorphisms of $X^*$&trivial homeomorphisms
 \\\hline\hline
 $X=\NN$&\cite{Rudin}&\cite{Shelah.PF,Shelah-Steprans.PFA,Velickovic.OCA}\\\hline
 $X$ $0$-dimensional&Parovicenko's Theorem&\cite{Farah.AQ,Farah-McKenney.ZD}\\\hline
  $X=\bigsqcup X_i$, $X_i$ compact&\cite{Coskey-Farah}&\cite{MKAV.FA}\\\hline
 $X=[0,1)$, $X=\er$ &Yu, see \cite{KP.STCC}& Theorem~\ref{thmi:alltrivial}\\\hline
 $X$ manifold &\cite{V.Nontrivial} &Theorem~\ref{thmi:alltrivial}\\\hline
all spaces &Unknown&Theorem~\ref{thmi:alltrivial}\\
 \hline
\end{tabular}
\end{center}

\vspace{7pt}

Partial results were also obtained in~\cite{Farah-Shelah.RCQ}, where a weaker version of triviality was considered. As noted, we prove the following: 
\begin{theoremi}\label{thmi:alltrivial}
Assume $\OCA+\MA_{\aleph_1}$ and let $X$ and $Y$ be second countable locally compact spaces. All homeomorphisms between $X^*$ and $Y^*$ are trivial.
\end{theoremi}

Let us turn a trivial homeomorphism $\phi\colon X^*\to Y^*$ into algebraic terms. The homeomorphism $\Phi\colon X\setminus K_X\to Y\setminus K_Y$ induces an isomorphism between the algebras $\overline{aC_0(X)a}$ and $\overline{bC_0(Y)b}$, where $a\in C_b(X)$ and $b\in C_b(Y)$ are positive elements (supported on $X\setminus K_X$ and $Y\setminus K_Y$ respectively) such that $1-a\in C_0(X)$ and $1-b\in C_0(Y)$. 

Viceversa, let $X$ and $Y$ be locally compact noncompact spaces. Suppose we have positive contractions $a\in C_b(X)$ and $b\in C_b(Y)$ with the property that $1-a\in C_0(X)$ and $1-b\in C_0(X)$, and suppose we are given an isomorphism between 
\[
\overline{aC_0(X)a}\text{ and }\overline{bC_0(Y)b}
\] which extends to a strictly continuous isomorphism between 
\[
\overline{a\mathcal M(C_0(X)a}\text{ and }\overline{b\mathcal M(C_0(Y))b}.
\]
 This isomorphism gives an isomorphism between $\mathcal Q(C_0(X))$ and $\mathcal Q(C_0(Y))$ whose dual maps is a trivial homeomorphism between $Y^*$ and $X^*$.

We generalise this to the noncommutative setting:
\begin{definition}
Let $A$ and $B$ be separable $\Cstar$-algebras and $\psi\colon \mathcal Q(A)\to\mathcal Q(B)$ be an isomorphism. $\psi$ is said algebraically trivial if there exist 
\begin{itemize}
\item positive $a\in\mathcal M(A)$ and $b\in\mathcal M(B)$ such that $1-a\in A$ and $1-b\in B$, and
\item an isomorphism $\phi\colon\overline{aAa}\to\overline{bBb}$ that extends to an isomorphism $\bar\phi\colon\overline{a\mathcal M(A)a}\to\overline{b\mathcal M(B)b}$
\end{itemize}
 such that
\[
\psi=\pi(\bar\phi).
\]
\end{definition}
Algebraically trivial isomorphisms of corona of separable algebras are topologically trivial, hence this notion strengthens the notion of triviality of Coskey and Farah. As an evidence that this is the correct notion of triviality for automorphisms of corona, we notice that, as sketched above (see also Proposition~\ref{prop:dual}), algebraically trivial isomorphisms of abelian coronas are dual to trivial homemorphisms of \v{C}ech-Stone remainders. Moreover, as unitaries in $\mathcal Q(H)$ can be lifted to partial isometries in $\mathcal B(H)$, and all automorphisms of $\mathcal B(H)$ are inner, algebraically trivial automorphisms of $\mathcal Q(H)$ correspond to inner ones. Theorems~\ref{thm:absolute} and ~\ref{thm:redprodFD} give further evidences that this is the correct notion of triviality.
 
 Having dealt with Question~\ref{ques:main}(a), and having given the appropriate notion of algebraically trivial isomorphism, we turn to study Question~\ref{ques:main}(b), by connecting it to Ulam stability and perturbation theory of $\Cstar$-algebras. Perturbation theory for $\Cstar$-algebra was developed after the seminal work of Kadison and Kastler \cite{KadKast}. The fundamental question on whether two $\Cstar$-algebras represented on the same Hilbert space that are close to each other (in the Hausdorff metric) are isomorphic can be rephrased as asking whether the relation of being isomorphic is stable under small perturbation for pairs of subalgebras of the same $\Cstar$-algebra.  Kadison and Kastler conjectured that this happens for all pairs of $\Cstar$-algebras\footnote{As every set theorist immediately guesses, separability is a needed assumption,  see \cite{ChoiCh.PertNonsep}.}. The conjecture was confirmed for separable nuclear\footnote{The class of nuclear is arguably the most important class of $\Cstar$-algebras. It consists of those algebras satisfying a certain internal approximation property and it coincides with the class of those algebras which are amenable Banach algebras (see \cite[\S II.9]{Blackadar.OA} for more information.)} algebras in \cite{CSSWW}; it has been the source of research in perturbation theory for decade, but a full solution is yet out of reach. In \cite{MKAV.UC}, the author and McKenney introduced a stronger notion of stability, named after Ulam's work on approximate homomorphisms (similarly to what was done for Boolean algebras and groups, see e.g.\cite{Farah.LHQS}). In short, a class $\mathcal C$ of $\Cstar$-algebras is said Ulam stable if approximate isomorphisms between elements of $\mathcal C$ can be perturbed to isomorphisms in a uniform way (see Definition~\ref{defin:ulam}. We link answers to Question~\ref{ques:main}(b) to whether certain classes of $\Cstar$-algebras are Ulam stable (see Theorem~\ref{thm:ulam} and~\ref{thm:ulam2}). We also show that a positive answer to Question~\ref{ques:main}(b) would confirm the conjecture of Kadison and Kastler (Theorem~\ref{thm:KK}).
 

After studying isomorphisms of coronas, we analyse their mutual embeddings. Again, this is inspired by the Boolean algebra case for Borel quotients of $\mathcal P(\NN)$: while under $\CH$ all quotients of $\mathcal P(\NN)$ embed into $\mathcal P(\NN)/\Fin$, Farah proved in \cite[\S3.5]{Farah.AQ} many non embedding results under Forcing Axioms. (Notably, if $\SI$ is a meager Borel dense ideal, under Forcing Axioms the Boolean algebra $\mathcal P(\NN)/\SI$ does not embed into $\mathcal P(\NN)/\Fin$.) We produce similar results for embeddings of corona $\Cstar$-algebras. Once again the role of $\mathcal P(\NN)/\Fin$ is played by $\mathcal Q(H)$: thanks to the main result of \cite{FHV.Calkin}, under $\CH$ all $\Cstar$-algebras of density continuum embed into $\mathcal Q(H)$, and in particular all coronas of separable $\Cstar$-algebras do. More than that, if $A$ and $B$ are unital separable $\Cstar$-algebras, under $\CH$ there is an embedding $\mathcal Q(A\otimes\mathcal K(H))\to\mathcal Q(B\otimes\mathcal K(H))$. On the other hand, assuming Forcing Axioms, one gets nonembedding results.

\begin{theoremi}\label{thmi:embedding}
Let $A$ and $B$ be unital separable infinite-dimensional $\Cstar$-algebras. Then
\begin{enumerate}[label=(\roman*)]
\item $\CH$ implies that $\mathcal Q(A\otimes\mathcal K(H))$ and $\ell_\infty(A)/c_0(A)$ embed in $\mathcal Q(H)$, and therefore into $\mathcal Q(B\otimes\mathcal K(H))$;
\item $\OCA$ and $\MA_{\aleph_1}$ imply that $\mathcal Q(A\otimes\mathcal K(H))$ and $\ell_\infty(A)/c_0(A)$ do not embed in $\mathcal Q(H)$. Moreover, if $B$ is stably finite and $A$ is not, then $\mathcal Q(A\otimes\mathcal K(H))$ does not embeds into $\mathcal Q(B\otimes K(H))$.
\end{enumerate}
\end{theoremi}

Focusing on the abelian case, using Gelfand's duality, we analyze continuous surjections between \v{C}ech-Stone remainders. In \cite{DowHart.Images} Dow and Hart showed that under $\PFA$ that if $X^*$ is a continuous image of $\NN^*$ then $X$ must be homeomorphic, modulo compact, to $\NN$. We extend significantly their result:

\begin{theoremi}\label{thmi:contimgs}
Assume $\OCA$ and $\MA_{\aleph_1}$. Let $X$ and $Y$ be locally compact second countable topological spaces. Assume that $Y^*$ surjects onto $X^*$. Then there are $Z\subseteq Y$ and a compact $K\subseteq X$ such that $Z$ is clopen modulo compact and $Z$ surjects onto $X\setminus K$.
\end{theoremi}

\vspace{5pt}

All our results heavily use the `noncommutative $\OCA$ lifting' Theorem~\ref{thm:lifting}, which is Theorem 4.5 in~\cite{MKAV.FA}.
Theorem~\ref{thmi:alltrivialnoncomm} is then obtained via a filtration for $\mathcal Q(A)$ exploited in \S\ref{section:general}, a generalisation of the construction in \cite[\S7]{MKAV.FA}. Other than for the proof of Theorem~\ref{thmi:alltrivialnoncomm}, the main technicalities  come in proving Theorem~\ref{thm:almostredprod}, from which one deduces Theorem~\ref{thmi:alltrivial}--\ref{thmi:contimgs}.


The paper is structured as follows: in \S\ref{section:prel} we introduce the main concepts involved and prove some of their basic properties. In \S\ref{section:general}, after some some heavy lifting, we prove Theorem~\ref{thmi:alltrivialnoncomm}. In \S\ref{section:rigidity} we use the work in \S\ref{section:general} to prove Theorem~\ref{thm:almostredprod}. We then show how to obtain Theorems~\ref{thmi:alltrivial}--\ref{thmi:contimgs} as its consequences. In \S\ref{section:algtriv} we deal with the notion of algebraically trivial isomorphisms, and we show how answers to Question~\ref{ques:main}(b) are connected to Ulam stability phenomena.

\subsection*{Acknowledgements}
The author is indebted with Ilijas Farah, Stefaan Vaes, and Paul McKenney for many long conversations concerning the content of this manuscript and related topics, and the anonymous referee for useful remarks. These notes were completed partially while the author held a postdoctoral position at IMJ-PRG in Paris, financed by the FSMP and a PRESTIGE-co-Fund grant, and partially while being a FWO postdoc at KU Leuven. The author would like to thank both institutions, as well as the funding bodies.
\section{Preliminaries}\label{section:prel}
We work in the interface between set theory and $\Cstar$-algebras. For a good introduction of these topics, see \cite{Kunen.ST} and \cite{Blackadar.OA}. Definitions of the main objects in use are given below.
\subsection{Set theory}\label{subsec:ST}
In this paper we show results which are consequences of Forcing Axioms. Forcing Axioms are strengthenings of the Baire category theorem to higher cardinals. The most famous of these is Shelah's Proper Forcing Axiom $\PFA$ (\cite{Shelah.PF}), of which we will use two consequences: Todor\v{c}evi\`c' Open Coloring Axiom $\OCA$ and Martin's Axiom $\MA_{\aleph_1}$. As we will not use these axioms expliticitely, we will not give a definition of $\PFA$ and of $\MA_{\aleph_1}$. Indeed, $\PFA$ is not used at all in these notes, and $\MA_{\aleph_1}$ is only used to access the lifting theorem of the author and McKenney, Theorem~\ref{thm:lifting}, (see \cite{MKAV.FA}).

An axiom we are going to frequently use (e.g, the proofs of Theorems~\ref{thmi:alltrivialnoncomm} and~\ref{thm:almostredprod}) is a modification of Todor\v{c}evi\`c' $\OCA$ known as $\OCA_\infty$. $\OCA$ and $\OCA_\infty$ were recently proved to be equivalent by Moore in \cite{Moore.OCA}. We will therefore assume $\OCA$, but often use $\OCA_\infty$.

If $\mathcal X$ is a set, $[\mathcal X]^2$ is the set of unordered pairs of elements of $\mathcal X$.
Subsets of $[\mathcal X]^2$ are identified with symmetric subsets of $\mathcal X^2\setminus \Delta_{\mathcal X}$, 
thus giving meaning to  the phrase `an open subset of $[\mathcal X]^2$'. 
The following axiom was introduced in~\cite{Farah.Cauchy}\footnote{The name $\OCA_\infty$  was used for a weaker axiom in \cite{Farah.Cauchy}. What is presently known as $\OCA_\infty$ is the dichotomy between (a) and (b') on \cite[p. 4]{Farah.Cauchy}.}, where it was also proved that it is relatively consistent with $\ZFC$.
\begin{OCAinfty}
Let $\mathcal X$ be a metric space and let $(K_0^n)_n$  be a sequence of open 
subsets of $[\mathcal X]^2$ with $K_0^{n+1}\subseteq    K_0^n$ for all $n\in\NN$. One of the following applies. 
\begin{enumerate}
  \item There is a sequence  $(\mathcal X_n)_n$ of subsets of $\mathcal X$  such that $\mathcal X=\bigcup_n \mathcal X_n$ and
  $[\mathcal X_n]^2\cap     K_0^n=\emptyset$ for every $n\in\NN$, or
  \item there is an uncountable $Z\subseteq    2^\NN$ and a continuous injection
  $f\colon  Z\to \mathcal X$ such that for all $x\neq y\in Z$ we have
  \[
    \{f(x),f(y)\}\in K_0^{\Delta(x,y)},
  \]
  where $\Delta(x,y)=\min\set{n}{x(n)\neq y(n)}$.
\end{enumerate}
\end{OCAinfty}

$\OCA$ is the statement $\OCA_\infty$ when $K_0^n=K_0^{n+1}$ for all $n\in\NN$. Both $\OCA_\infty$ and $\MA_{\aleph_1}$, are consequences of $\PFA$ (\cite{Shelah.PF}). Unlike $\PFA$,   their relative consistency with $\ZFC$ does not require any large cardinal assumptions (see \cite{Farah.Cauchy} and \cite{Kunen.ST}).   

Each of these axioms  contradicts the Continuum Hypothesis $\CH$. More than that, $\OCA$ implies that the bounding number $\mathfrak b$ is $\aleph_2$. If $f$ and $g$ are functions in $\NN^\NN$, we write $f\leq_*g$ if there is $n_0$ such that $f(n)\leq g(n)$ whenever $n\geq n_0$. The cardinal $\mathfrak b$ is the minimal cardinality of a subset of $\NN^\NN$ which is $\leq_*$-unbounded. A diagonalisation argument gives that $\omega_1\leq\mathfrak b$, and clearly $\mathfrak b\leq2^{\aleph_0}$. We will use that $\OCA$ implies that $\mathfrak b>\omega_1$ in Lemmas~\ref{lemma:OCAfirstalternativecomplicated} and~\ref{lem:liftingmore}.

\subsubsection*{Ideals in $\mathcal P(\NN)$}

A subset $\SI\subseteq\mathcal P(\NN)$ is an ideal on $\NN$ if it is hereditary
and closed under finite unions. An ideal $\SI$ is dense if for every infinite $S\subseteq\NN$ there is an infinite $T\subseteq S$ with $T\in\SI$. 

When $\mathcal P(\NN)$ is endowed with the product topology from $2^\NN$, we can study the topological properties of its ideals. The following combinatorial condition characterizes meagerness:  an ideal $\SJ\subseteq\SP(\NN)$ is meager if and only if there is a partition  $\NN = \bigcup\set{E_n}{n\in\NN}$ into finite  intervals such that for any  infinite set $L$, $\bigcup\set{E_n}{n\in L} \not\in\SJ$.
  
Recall that the intersection of countably many nonmeager ideals is nonmeager, if each ideal contains all finite sets (e.g. \cite[\S 3.10]{Farah.AQ}). In case one wants to intersect more than countably many nonmeager ideals, one has to look at the cardinal invariant $\mathfrak g$, the groupwise density cardinal. Although defined in different terms, $\mathfrak g$ can be viewed as the minimal size of a set of nonmeager ideals, each containing all finite sets, whose intersection is nonempty. By \cite[Corollary 7.9]{Blass.HB}, if $\MA$ holds then $\mathfrak g=2^{\aleph_0}$, therefore under $\MA_{\aleph_1}$ the intersection of $\aleph_1$ nonmeager ideals is nonmeager. This will be used in Lemma~\ref{lem:liftingmore}.

\subsection{$\Cstar$-algebras, their multipliers and coronas, and trivial isomorphisms}\label{subsec:cstartrivial}
$\Cstar$-algebras are Banach complex algebras together with an involutive $^*$ operation satisfying $\norm{aa^*}=\norm{a}^2$. Thanks to the GNS construction (\cite[II.6.4]{Blackadar.OA}), every $\Cstar$-algebra can be viewed as a $^*$-closed Banach subalgebra of the algebra of bounded linear operators $\mathcal B(H)$ for some Hilbert space $H$. By Gelfand's duality all abelian $\Cstar$-algebras arise as algebras of continuous complex-valued functions vanishing at infinity on some locally compact $X$.

For a $\Cstar$-algebra $A$, especially a nonunital one, we  want to examine how $A$ can be embedded as an ideal into a unital $\Cstar$-algebra $B$, similarly as one studies the ways of embedding a locally compact space $X$ as an open subset of a compact space. The minimal such $B$ is the unitization of $A$, denoted $A^\dagger$, which corresponds to the one-point compactification. The `maximal' such unital algebra is the multiplier algebra $\mathcal M(A)$, which corresponds to the \v{C}ech-Stone compactification in the commutative case, and we introduce below.

If $A\subseteq B$ are $\Cstar$-algebras, $A$ is an essential ideal of $B$ if $A$ is a $^*$-closed norm closed ideal, and  $bA=Ab=0$  implies $b=0$ whenever $b\in B$. This is a notion of denseness, as it says that $A$ is `large' in $B$. The multiplier algebra of $A$, $\mathcal M(A)$, can be thought as `the largest unital $\Cstar$-algebra in which $A$ is an essential ideal'. More precisely, $\mathcal M(A)$ is defined as the unital algebra that has the property that whenever $A$ sits as an ideal into a $\Cstar$-algebra $B$ then the identity map on $A$ extends uniquely to a $^*$-homomorphism from $B$ into $\mathcal M(A)$ of kernel $\{b\in B\mid \forall a\in A(ab=ba=0)\}$. For every $\Cstar$-algebra $A$, its multiplier algebra exists, and it is unique up to isomorphism. The construction of $\mathcal M(A)$ can be made in different ways. The first construction can be traced back to the work of Busky (\cite{Busby}), but (perhaps more naturally) one can see $\mathcal M(A)$ as the closure of $A$ in its strict topology (see below), once $A$ is represented faithfully on some Hilbert space. Notably, if $A=C_0(X)$, $\mathcal M(A)$ corresponds to the \v{C}ech-Stone compactification of $X$: in this case $\mathcal M(A)\cong C_b(X)\cong C(\beta X)$. For more on multipliers, see e.g. \cite[II.7.3]{Blackadar.OA}.

\begin{definition}
The quotient $\mathcal Q(A)=\mathcal M(A)/A$ is the corona of $A$ and   $\pi_A\colon\mathcal M(A)\to\mathcal Q(A)$ is the quotient map.
\end{definition}
The main objects of interest of this paper are isomorphisms of corona algebras. To follow the pattern established in the introduction, we need to endow $\mathcal M(A)$ with a suitable Polish topology compatible with its algebraic structure.

If $A$ is nonunital, the multiplier algebra is nonseparable in norm, therefore we have to consider a different topology. The strict topology is the topology generated by the seminorms
\[
l_a(x)=\norm{ax}, \,\, r_a(x)=\norm{xa}, \, a\in A.
\]
If $A$ is a $\sigma$-unital, that is, if it has a countable approximate identity, the strict topology is Polish when restricted to norm bounded subsets of $\mathcal M(A)$. This turns $\mathcal M(A)$ into a standard Borel space.
If $X$ is locally compact, the strict topology of $C_b(X)\cong C(\beta X)$ generated by $C_0(X)$ coincides with the topology of uniform convergence on compact subsets of $X$. In case $A_n$ are unital $\Cstar$-algebras, the multiplier algebra of $\bigoplus A_n$ (the algebra of sequences whose norm converges to $0$) is the algebra of bounded sequence, $\prod A_n$. In this case, the strict topology on $\prod A_n$ generated by $\bigoplus A_n$ coincides with the product of the norm topologies on $A_n$. The corona algebra $\prod A_n/\bigoplus A_n$ is called the reduced product of the $A_n$'s.

\begin{definition}[\cite{Coskey-Farah}]\label{defin:trivialiso}
Let $A$ and $B$ be separable $\Cstar$-algebras. An isomorphism $\Lambda\colon\mathcal Q(A)\to\mathcal Q(B)$ is topologically trivial if its graph 
\[
\Gamma(\Lambda)=\{(a,b)\in\mathcal M(A)\times\mathcal M(B)\colon\norm{a},\norm{b}\leq 1,\,\,\, \Lambda(\pi_A(a))=\pi_B(b)\}
\]
is Borel in the product of the strict topologies.
\end{definition}
Since the unit balls of $\mathcal M(A)$ and $\mathcal M(B)$ are Polish spaces there can be only $2^{\aleph_0}$ Borel subsets of the unit ball of $\mathcal M(A)\times\mathcal M(B)$, therefore there exist only $2^{\aleph_0}$ topologically trivial isomorphisms.

White the notion of topologically trivial isomorphism is completely satisfactory from the point of view of set theory, it is not from the $\Cstar$-algebraist' view. Topologically trivial isomorphisms might not have liftings which are preserving even the more basics algebraic properties. For this reason, we introduce a stronger notion of triviality.

 Let $a\in \mathcal M(A)$ and $b\in\mathcal M(B)$ be positive contractions with $1-a\in A$ and $1-b\in B$. Suppose that there is an isomorphism 
\[
\phi\colon \overline{aAa}\to\overline{bBb}
\]
that extends to a strictly continuous isomorphism 
\[
\bar\phi\colon\overline{a\mathcal M(A)a}\to\overline{b\mathcal M(B)b}.
\]
 Since for all $x\in\mathcal M(A)$ we have that $x=axa+(1-a)xa+ax(1-a)+(1-a)x(1-a)$, we have that $x-axa\in A$. In particular, $\bar\phi$ induces a unique $^*$-isomorphism
\[
\tilde\phi=\pi(\bar\phi)\colon \mathcal Q(A)\to\mathcal Q(B),
\]
which we call the isomorphism induced by $\phi$.
\begin{definition}\label{def:algtriv}
Let $\psi\colon \mathcal Q(A)\to\mathcal Q(B)$ be an isomorphism. $\psi$ is algebraically trivial if there exist positive elements $a\in\mathcal M(A)$, $b\in\mathcal M(B)$ and an isomorphism $\phi\colon\overline{aAa}\to\overline{bBb}$ such that 
\begin{itemize}
\item $1-a\in A$, $1-b\in B$, and
\item $\phi$ extends to a strongly continuous isomorphism $\bar\phi\colon\overline{a\mathcal M(A)a}\to\overline{b\mathcal M(B)b}$ such that
\[
\psi=\pi(\bar\phi).
\]
\end{itemize}
We say that $\mathcal Q(A)$ and $\mathcal Q(B)$ are algebraically isomorphic if there is an algebraically trivial isomorphism between them.
\end{definition}
\begin{remark}\label{remark:notallisos}
If a net in $A$ strictly converges to an element of $A$, then it converges in norm. Therefore a norm continuous map $\overline{aAa}\to\overline{bBb}$ is automatically strictly continuous. On the other hand, even when $1-a\in A$ and $1-b\in B$, an isomorphism between $\overline{aAa}$ and $\overline{bBb}$ does not necessarily maps strictly convergent sequences to strictly convergent sequences, and therefore might not extend to an isomorphism at the level of the multiplier algebra.

For this, fix unital algebras $A_n$. Let $C=(\bigoplus A_{2n})^\dagger$, the unitisation of $(\bigoplus A_{2n})$, $B=\bigoplus A_{n}$ and $A=C\oplus \bigoplus A_{2n+1}$. Let 
\[
c=(\frac{1}{n})_n\in \bigoplus A_{2n},\,\,d=1\in\prod A_{2n+1},\,\,a=c\oplus d\text{ and }b=1\in \mathcal M(B).
\]
Then $a\in \mathcal M(A)$, and $1-a\in A$. The obvious isomorphism between $\overline{aAa}$ and $B$ does not extend to an isomorphism between $\overline{a\mathcal M(A)a}$ and $\mathcal M(B)$. In fact, if $p_{2n}\in A_{2n}$ is the unit of $A_{2n}$, the sequence $q_m=\sum_{n\leq m} p_{2n}$ is strictly convergent in $B$ but not in $A$.

A commutative example can by constructed by considering $A=B=C_0([0,1))$. Let $a\in C_b([0,1))$ be given by $a(t)=t$ for $t\in [0,1)$. The automorphism of $C_0((0,1))=\overline{aAa}$ dual to the homeomorphism given by $t\mapsto 1-t$ does not extend to an automorphism of $\overline{a\mathcal M(A)a}$.
\end{remark}

We relate this notion of triviality to other notions of triviality for isomorphisms of corona $\Cstar$-algebras.
\begin{theorem}\label{thm:absolute}
Let $A$ and $B$ be separable $\Cstar$-algebras. Then
\begin{enumerate}[label=(\arabic*)]
\item every algebraically trivial isomorphism $\phi\colon\mathcal Q(A)\to\mathcal Q(B)$ is topologically trivial, and
\item whether there is a topologically trivial isomorphisms $\mathcal Q(A)\to\mathcal Q(B)$ that is not algebraically trivial cannot be changed by forcing.
\end{enumerate}
\end{theorem}
\begin{proof}
For (1), let $\phi\colon \mathcal Q(A)\to\mathcal Q(B)$ be an algebraically trivial isomorphism, and let $a\in\mathcal M(A)$, $b\in \mathcal M(B)$, and $\phi\colon \overline{aAa}\to\overline{bBb}$ be as in Definition~\ref{def:algtriv}. Since $1-a\in A$, the map $\rho_a\colon \mathcal M(A)\to \mathcal M(A)$ given by $x\to axa$ is a strictly continuous map which induces the identity on $\mathcal Q(A)$. Similarly, the inclusion $\overline{b\mathcal M(B)b}\subseteq \mathcal M(B)$ is strictly continuous. Moreover, $\phi$ is strictly continuous, and so is $\bar\phi\colon \overline{a\mathcal M(A)a}\to\overline{b\mathcal M(B)b}$. Then $\bar\phi\circ\rho_a$ is a continuous lift of $\phi$, and so its graph is Borel.

For (2), note that the statement `$\phi$ is topologically trivial' and `$\phi$ is algebraically trivial' are, respectively, $\Pi_2^1$ and $\Sigma_2^1$. Shoenfield's Absoluteness Theorem (\cite[Theorem 13.17]{Kana}) affirms that their validity cannot be changed by forcing.
\end{proof}

In the abelian case we can reformulate, using Gelfand's duality, our algebraic notion of triviality in terms of maps between the spectra of the abelian algebras in consideration. In this case, isomorphisms of corona of abelian algebra correspond to homeomorphism between their respective spectra, that is, if $A=C_0(X)$ and $B=C_0(Y)$, an isomorphism between $\mathcal Q(A)$ and $\mathcal Q(B)$ gives a homeomorphism between the \v{C}ech-Stone remainders $Y^*=\beta Y\setminus Y$ and $X^*=\beta X\setminus X$. Viceversa, a homeomorphism between $X^*$ and $Y^*$ defines dually an isomorphism between $\mathcal Q(Y)$ and $\mathcal Q(X)$.

 The following notion of triviality was formulated by Farah and Shelah in \cite{Farah-Shelah.RCQ}. It extends naturally the notion of trivial homeomorphisms of $\beta\NN\setminus\NN$ originally formulated by Rudin (\cite{Rudin}).

\begin{definition}[\cite{Farah-Shelah.RCQ}]\label{defin:trivialhomeo}
Let $X$ and $Y$ be locally compact second countable topological spaces. A homeomorphism $\Lambda^*\colon X^*\to Y^*$ is trivial if there are compact sets $K_X\subseteq X$ and $K_Y\subseteq Y$ and a homeomorphism $\phi\colon X\setminus K_X\to Y\setminus K_Y$ such that $\beta\phi\restriction X^*=\Lambda^*$. 
\end{definition}

Our definition of algebraic triviality is dual to that of Farah and Shelah.
\begin{proposition}\label{prop:dual}
Let $X$ and $Y$ be locally compact second countable topological spaces. Then trivial homeomorphisms between $X^*$ and $Y^*$ are dual to algebraically trivial isomorphisms between $C(X^*)$ and $C(Y^*)$.
\end{proposition}
\begin{proof}
Consider a trivial homeomorphism $\Lambda^*\colon X^*\to Y^*$, and let $\phi\colon X\setminus K_X\to Y\setminus K_Y$ be as in Definition~\ref{defin:trivialhomeo}. Let $a\in C_b(X)$ be a positive element whose support equals $X\setminus K_X$ and such that $1-a\in C_0(X)$, and similarly define $b\in C_0(Y)$ to be an element of support equal to $Y\setminus X_Y$. Then $\phi$ dually induces an isomorphism 
\[
\psi\colon C_0(Y\setminus K_Y)\cong\overline{bC_0(Y)b}\to \overline{aC_0(X)a}=C_0(X\setminus K_X).
\]
We leave to the reader to verify that $\tilde\psi\colon \mathcal Q(C_0(Y))\to \mathcal Q(C_0(X))$ is dual to $\Lambda^*$.

For the opposite direction, let 
\[
\tilde\phi=\pi(\bar\phi)\colon \mathcal Q(C_0(Y))\to\mathcal Q(C_0(X))
\]
 be an isomorphism induced by 
 \[
 \phi\colon \overline{bC_0(Y)b}\to\overline{aC_0(X)a}
 \]
  as in Definition~\ref{def:algtriv}. Since $1-a\in C_0(X)$ and $1-b\in C_0(Y)$, the supports of $a$ and $b$ are open cocompact subsets of $X$ and $Y$ respectively. Notice that that 
  \[
  \overline {aC_0(X)a}\cong C_0(\supp(a))\text{ and } \overline {bC_0(Y)b}\cong C_0(\supp(b)).
  \]
  The dual map $\psi\colon \supp(a)\to\supp(b)$ provides the required homeomorphism ensuring triviality of the dual of $\tilde\phi$.
\end{proof}

In particular, our Theorem~\ref{thmi:alltrivial} below asserts that under $\OCA$ and $\MA_{\aleph_1}$ all isomorphisms of corona of separable abelian algebras are algebraically trivial.

The next corona $\Cstar$-algebra we consider is the Calkin algebra $\mathcal Q(H)$, the quotient of $\mathcal B(H)$, the algebra of bounded linear operator on a complex valued separable Hilbert space, by the ideal of compact operators. Let $\pi\colon\mathcal B(H)\to\mathcal Q(H)$ be the canonical quotient map.

\begin{proposition}
Inner automorphisms of the Calkin algebra correspond to algebraically trivial ones.
\end{proposition}
\begin{proof}
Suppose that $u\in\mathcal Q(H)$ induces the inner automorphism $\phi\in\Aut(\mathcal Q(H))$. Let $s\in\mathcal B(H)$ be a partial isometry between subspaces of $H$ of cofinite dimension such that $\pi(s)=u$. Let $p$ be the projection onto the domain the $s$. Then $s$ defines an inner automorphism of $\mathcal B(p\mathcal H)$ by $x\mapsto s^*xs$. In particular $\phi$ is algebraically trivial.  Viceversa, suppose that $\phi$ is an algebraically trivial automorphism of $\mathcal Q(H)$, and let $a,b\in\mathcal B(H)$ and $\psi\colon\overline{a\mathcal K(H)a}\to\overline{b\mathcal K(H)b}$ be as in Definition~\ref{def:algtriv}. Since $1-a$ is compact and positive, it is diagonalizable. Let $p$ be the projection cofinite range onto the support of $(1-a-\frac{1}{4})_+$. Then $\psi\restriction p\mathcal K(H)p$ is a $^*$-isomorphism onto its range, and therefore it is induced by a contractive isometry $u\colon pH\to\psi(p)H$. This is not a unitary in $\mathcal B(H)$, but $\pi(u)\in\mathcal Q(H)$ is, and so $\phi$ is inner. The above discussion shows
\end{proof}
Therefore Farah's result (\cite{Farah.C}) asserting that under $\OCA$ all automorphism of $\mathcal Q(H)$ are inner gives that all automorphisms of $\mathcal Q(H)$ are algebraically trivial.

One more evidence that this provides the correct notion of triviality is given by reduced products and their isomorphisms. We will investigate this case further in \S\ref{subsec:redprod} (see specifically Proposition~\ref{prop:redprodalgtriv}). We report what is known for the simplest reduced products: the ones of finite-dimensional $\Cstar$-algebras. Below \cite[Corollary~6.9]{MKAV.FA} stated in our terms. (The thesis of Theorem~\ref{thm:redprodFD} was proved to be consistent with $\ZFC$ by Ghasemi in \cite{Ghasemi.FDD}.)
\begin{theorem}\label{thm:redprodFD}
Assume $\OCA$ and $\MA_{\aleph_1}$. Let $A_n$ are finite-dimensional $\Cstar$-algebras. Then all automorphisms of $\prod A_n/\bigoplus A_n$ are algebraically trivial.\qedhere
\end{theorem}

\subsection{Asymptotically additive and skeletal maps}\label{ss:am}
\begin{definition}\label{defin:asadd}
Let $E_n$, $n\in\NN$ be Banach spaces and let $A$ be a separable $\Cstar$-algebra. Let $\{e_n\}$ be an approximate identity of positive contractions\footnote{The fact that we are using contractions here and later (\S\ref{sub:prep},\S\ref{sub:thmalltriv} and \S\ref{subsec:abelianproof}) is purely to reduce notation. Any bounded approximate identity with the lister properties would do, even though in this case the notation would exponentially increase.} for $A$ with $e_{n+1}e_n=e_n$ for all $n$. A map $\alpha\colon \prod E_n\to \mathcal M(A)$ is asymptotically additive if there are sequences $j_n<k_n$ and 
\[
\alpha_n\colon E_n\to (e_{k_n}-e_{j_n})A(e_{k_n}-e_{j_n})
\]
such that $\alpha_n(0)=0$, $j_n\to\infty$ as $n\to\infty$, and for all $a=(a_n)\in\prod E_n$ the sequence of partial sums $(\sum_{n\leq m}\alpha_n(a_n))_m$ converges in strict topology to $\alpha(a)$, that is
\[
\alpha(a)=\sum\alpha_n(a_n).
\]
\end{definition}
When writing $\alpha=\sum\alpha_n\colon\prod E_n\to\mathcal M(A)$ we always mean that $\alpha$ is an asymptotically additive map with $\alpha_n\colon E_n\to A$. 
We are interested in particularly well-behaved asymptotically additive maps: the ones determined by their values on a product of finite sets.
\begin{definition}
Let $\{E_n\}$ be finite-dimensional Banach spaces. We say that $\bar Y=(Y_{n},<_{n})$ is a dense system for $\{E_n\}$, if 
\begin{itemize}
\item $Y_n$ is a finite $2^{-n}$-dense subset of the unit ball of $E_n$, and $0\in Y_n$ and
\item $Y_n$ is linearly ordered by $<_n$, and $0$ is the $<_n$-minimal element.
\end{itemize}
If $E_n$ is a sequence of finite-dimensional Banach spaces with dense system $(Y_n,<_n)$, the map $\rho_{E_n,Y_n}\colon E_n\to E_n$ is the Borel map defined as
\begin{itemize}
\item $\rho_{E_n,Y_n}(x)=x$ if $x\in Y_n$,
\item if $x\notin Y_n$ is a contraction, $\rho_{E_n,Y_n}(x)$ is the $<_n$-minimal element of $Y_n$ which is $2^{-n}$-close to $x$,
\item if $\norm{x}>1$, then $\rho_{E_n,Y_n}(x)=\norm{x}\rho_{E_n,Y_n}(x/\norm{x})$.
\end{itemize}
Let $\rho_{\bar E,\bar Y}=\prod \rho_{E_n,Y_n}\colon\prod E_n\to\prod E_n$.
\end{definition}
\begin{definition}\label{defin:skeletal1}
Let $A$ be a separable $\Cstar$-algebra, $\{e_n\}$ be an approximate identity of positive contractions for $A$ with $e_{n+1}e_n=e_n$ for all $n$. Let $\{E_n\}$ be finite-dimensional Banach spaces and let $\bar Y=(Y_{n},<_{n})$ be a dense system for $\{E_n\}$. A map $\alpha_n\colon E_n\to A$ is skeletal w.r.t. $Y_n$ if 
 \[
 \alpha_n(x)=\alpha_n\circ\rho_{E_n,Y_n}(x), \,\, \forall x\in E_n,\, \norm{x}\leq 1.
 \] 
  If $\alpha=\sum\alpha_n\colon\prod E_n\to\mathcal M(A)$ is asymptotically additive and each $\alpha_n$ is skeletal w.r.t. $Y_n$ we say that $\alpha$ is skeletal w.r.t. $\bar Y$.
\end{definition}

We work with skeletal maps with different domains. This justifies the following.

\begin{definition}\label{defin:skeletal2}
Let $\{E_{n,m}\}$ be finite-dimensional Banach spaces. We say that $\bar Y=(Y_{n,m},<_{n,m})$ is a dense system for $\{E_{n,m}\}$ if
\begin{itemize}
\item $Y_{n,m}$ is a finite $2^{-n}$-dense subset of the unit ball of $E_{n,m}$, with $0\in Y_{n,m}$;
\item $Y_{n,m}$ is linearly ordered by $<_{n,m}$, and $0$ is the $<_{n,m}$-minimal element.
\end{itemize}
A map $\gamma\colon\prod_{n,m}E_{n,m}\to \mathcal M(A)$ is skeletal w.r.t. $\bar Y$ if there is $g\in\NN^\NN$ such that
\begin{itemize}
\item $\gamma$ is completely determined by its values on $\prod_n E_{n,g(n)}$, that is $\gamma(y)=\gamma(x)$ whenever $x_{n,g(n)}=y_{n,g(n)}$ for all $n$.
\item The map $\alpha\colon \prod E_{n,g(n)}\to \mathcal M(A)$ defined as 
\[
\alpha(x)=\gamma(y)\,\,\,\text{where } y_{n,g(n)}=x_{n}
\] 
is skeletal w.r.t. $(Y_{n,g(n)})$.
\end{itemize}
\end{definition}

\begin{proposition}\label{prop:skeletalpolish}
Let $\{E_{n,m}\}$ be finite dimensional Banach spaces and let $\bar Y=(Y_{n,m},<_{n,m})$ be a dense system for $\{E_{n,m}\}$. Let $A$ be a separable $\Cstar$-algebra. Then there is a compact metric (and therefore Polish) topology on the set of all maps $\alpha\colon\prod_{n,m} E_{n,m}\to\mathcal M(A)$ that are skeletal w.r.t. $\bar Y$. 
\end{proposition}
\begin{proof}
Let $g\in\NN^\NN$, and let $\alpha=\sum\alpha_n\colon \prod E_{n,g(n)}\to \mathcal M(A)$ with $\alpha_n\colon E_{n,g(n)}\to A$ be witnessing that $\alpha$ is skeletal w.r.t. $\bar Y$. Note that $\alpha_n$ is only determined by its values on the finite set $Y_{n,g(n)}$. Such values lie in $A$, so we can identify $\alpha$ as an element of the space $\prod (Y_{n,m}\times A)$ with the product topology, which is Polish.
\end{proof}
The reason we need to work with skeletal maps instead of just asymptotically additive ones resides in that the set of all asymptotically additive maps cannot be given a suitable separable topology. In particular cases, when passing to quotients, is enough to consider skeletal maps. The proof of the following is obvious.
\begin{proposition}\label{prop:skeletalareenough}
Let $\{E_{n}\}$ be finite dimensional Banach spaces and let $\bar Y=(Y_{n},<_{n})$ be a dense system for $\{E_{n}\}$. Suppose that $\alpha=\sum\alpha_n\colon\prod E_n\to\mathcal M(A)$ is an asymptotically additive map. Suppose that $\pi_A(\alpha(x))=\pi_A(\alpha(y))$ whenever $\pi_E(x)=\pi_E(y)$, where $\pi_E\colon\prod E_n\to\prod E_n/\bigoplus E_n$ is the quotient map. Define $\beta_n=\alpha_n\circ\rho_{E_n,Y_n}$ and $\beta=\sum\beta_n$. Then 
\[
\pi_A(\alpha(x))=\pi_A(\beta(x)),\,\, x\in(\prod E_n), \,\norm{x}\leq 1,
\]
 and $\beta$ is skeletal w.r.t. $\bar Y$. 
\end{proposition}
If $E_n$ are Banach spaces and $S\subseteq\NN$, we denote by $E[S]\subseteq\prod E_n$ the subspace of elements with support in $S$ (i.e., $x\in E[S]$ only when $n\notin S\Rightarrow x_n=0$). With this notation, $\prod E_n=E[\NN]$. If $x\in E[\NN]$, let $x\restriction S\in E[S]$ be defined as
\[
x\restriction S=\begin{cases}
x_n&n\in S\\
0&n\notin S.
\end{cases}
\]
The map $\pi_E\colon\prod E_n\to\prod E_n/\bigoplus E_n$ denotes the quotient map.
\begin{definition}\label{defin:preserving}
A function $\phi\colon\prod E_n/\bigoplus E_n\to\mathcal Q(A)$  preserves the coordinate algebra if there are positive contractions $p_S\in\mathcal M(A)$, for $S\subseteq \NN$, such that
    \begin{itemize}
    \item for all $S\subseteq\NN$ and $x\in E[\NN]$, 
    \begin{eqnarray*}
    \pi_A(p_S)\phi(\pi_E(x))&=&\pi_A(p_S)\phi(\pi_E(x\restriction S))=\phi(\pi_E(x\restriction S))\\&=&\phi(\pi_E(x\restriction S))\pi_A(p_S)=\phi(\pi_E(x))\pi_A(p_S),
    \end{eqnarray*}
    \item if $S$ is finite, $\pi_A(p_S)=0$,
    \item for disjoint $S,T\subseteq \NN$, $\pi_A(p_{S\cup T})=\pi_A(p_T)+\pi_A(p_S)$ and
    \item if $S\cap T$ is finite then $\pi_A(p_S)\pi_A(p_T)=0$.
    \end{itemize}
\end{definition}
Let $\phi\colon \prod E_n/\bigoplus E_n\to\mathcal M(A)/A$ be a map and let $Z\subseteq E[\NN]$. Suppose that $\Phi\colon E[\NN]\to\mathcal M(A)$ is such that the following diagram
\begin{center}
\begin{tikzpicture}
 \matrix[row sep=1cm,column sep=2cm] {
&\node  (A1) {$\prod E_n$};

& \node (A2) {$\mathcal M(A)$};
\\
&\node  (B1) {$\prod E_n/\bigoplus E_n$};
& \node (B2) {$\mathcal Q(A)$};
\\
}; 
\draw (A1) edge[->] node [below] {$\Phi$} (A2) ;
\draw (A1) edge[->]  node[left]{$\pi_E$} (B1) ;
\draw (A2) edge[->]   node[right] {$\pi_A$} (B2) ;
\draw (B1) edge[->] node [above] {$\phi$} (B2) ;
\end{tikzpicture}
\end{center}
commutes on $Z$.
We say that $\Phi$ is a lift of $\phi$ on $Z$. If $\SI\subseteq\mathcal P(\NN)$ is an ideal and for all $S\in\SI$, $\Phi$ is a lift of $\phi$ on $E[S]$, we say that $\Phi$ is a lift of $\phi$ on $\SI$. 

The following is a noncommutative version of the ``OCA lifting theorem" (\cite[Theorem~3.3.5]{Farah.AQ}). It also appeared in a weaker version (involving the existence of an approximate identity of projections) in \cite[\S 5]{V.PhDThesis}. It was proved as \cite[Theorem 4.5]{MKAV.FA}.
\begin{theorem}\label{thm:lifting}
Assume $\OCA$ and $\MA_{\aleph_1}$, let $\{E_n\}$ be finite dimensional Banach spaces and let $\bar Y$ be a dense system for $\{E_n\}$. Let $A$ be a separable $\Cstar$-algebra. If 
\[
\phi\colon\prod E_n/\bigoplus E_n\to \mathcal Q(A)
\] 
is a linear bounded map that preserves the coordinate algebra, then there are a nonmeager dense ideal $\SI\subseteq\mathcal P(\NN)$ and a skeletal (w.r.t. $\bar Y$) map
\[
\alpha\colon\prod E_n\to\mathcal M(A)
\]
such that $\alpha$ is a lift of $\phi$ on $\SI$.
\end{theorem}

\subsection{Approximate maps}\label{subsec:approxmaps}
Let $A$ and $B$ be Banach spaces, let $\epsilon>0$, and let $\phi\colon A\to B$ be a contractive map. We say that $\phi$ is $\epsilon$-linear if 
\[
\sup_{a,b\in A_1}\norm{\phi(a+b)-\phi(a)-\phi(b)}<\epsilon.
\]
$\phi$ is $\epsilon$-injective if 
\[
\sup_{a\in A_1}|\norm{a}-\norm{\phi(a)}|<\epsilon,
\]
and $\epsilon$-surjective if for every contraction $b\in B$ there is $a\in A$ with $\norm{\phi(a)-b}<\epsilon$. 

If $A$ and $B$ are $\Cstar$-algebras, we define similarly the notions of $\epsilon$-multiplicative and $\epsilon$-$^*$-preserving maps. An $\epsilon$-linear $\epsilon$-$^*$-preserving and $\epsilon$-multiplicative map $\phi\colon A\to B$ is an $\epsilon$-$^*$-homomorphism. An $\epsilon$-$^*$-homomorphism that is $\epsilon$-injective is an $\epsilon$-$^*$-monomorphism, and if in addition $\epsilon$-surjective, an $\epsilon$-$^*$-isomorphism.

The study of approximate $^*$-homomorphisms is motivated largely by perturbation theory and the study of isomorphisms of reduced products, see \cite[\S2]{MKAV.FA}, \cite{MKAV.UC} and \S\ref{section:algtriv}.
Here we record three facts: the first one is Proposition~5.1.5 in \cite{V.PhDThesis}.
\begin{proposition}\label{prop:liftsandalmost}
Suppose that $E_n$ are metric spaces, $A$ is a $\Cstar$-algebra and $\phi\colon\prod E_n/\bigoplus E_n\to\mathcal Q(A)$ is a function such that there is an asymptotically additive $\Phi=\sum\phi_n\colon\prod E_n\to\mathcal M(A)$ such that
\begin{itemize}
\item  $\phi_n\phi_m=0$ if $n\neq m$ (i.e., for every $x\in E_n$ and $y\in E_m$ we have that $\phi_n(x)\phi_m(y)=0$, and
\item for every $x\in \prod E_n$ there is a dense ideal $\SI=\SI_x$ such that $\pi_A(\Phi(x\restriction S))=\phi(\pi_E(x\restriction S))$ for all $S\in \SI$.
\end{itemize}
Then
\begin{itemize}
\item if $\phi$ is linear (injective) for all $\epsilon>0$ there is $n_0$ such that $\phi_n$ is $\epsilon$-linear (injective) whenever $n\geq n_0$;
\item if in addition every $E_n$ is a $^*$-algebra and $\phi$ is $^*$-preserving  (multiplicative) then for all $\epsilon>0$ there is $n_0$ such that $\phi_n$ is $\epsilon$-$^*$-preserving (multiplicative) whenever $n\geq n_0$;
\end{itemize}
\end{proposition}
\begin{proof}
We prove only the first statement and leave the rest to the reader. Suppose that there is $\epsilon>0$ and an infinite sequence $n_k$ such that no $\phi_{n_k}$ is an $\epsilon$-linear map. Then there are contractions $x_k,y_k\in E_{n_k}$ such that 
\[
\norm{\phi_{n_k}(x_{k}+y_k)-\phi_{n_k}(x_{k})-\phi_{n_k}(y_{k})}>\epsilon.
\]
Let $x=(x_k)$ and $y=(y_k)$. By density of $\SI_x\cap\SI_y\cap\SI_{x+y}$ we can assume that $\{n_k\}\in \SI_x\cap\SI_y\cap \SI_{x+y}$. Then 
\begin{eqnarray*}
0&=&\norm{\Phi(\pi_E(x+y))-\Phi(\pi_E(x))-\Phi(\pi_E(y))}\\&=&\limsup_k\norm{\phi_{n_k}(x_{k}+y_k)-\phi_{n_k}(x_{k})-\phi_{n_k}(y_{k})}\geq\epsilon,
\end{eqnarray*}
a contradiction.
\end{proof}
In certain cases approximate maps can be perturbed to actual morphisms uniformly over the class of objects one considers. This phenomenon is known as Ulam stability (see~\S\ref{section:algtriv}). For now, we record a result of S\v{e}rml which will be needed in the proof of Theorem~\ref{thmi:alltrivial}:
\begin{theorem}\label{thm:semrl} \cite[Theorem 5.1]{Semrl.USAbel}
There is $\epsilon_0>0$ such that whenever $A$ and $B$ are abelian $\Cstar$-algebras, $\epsilon<\epsilon_0$ and $\phi\colon A \to B$ is an $\epsilon$-$^*$-homomorphism then there is a $^*$-homomorphism $\psi\colon A\to B$ with $\norm{\psi(a)-\phi(a)}<10\sqrt{\epsilon}\norm{a}$ for all $a\in A$. 
\end{theorem}

Lastly, we relate approximate maps with topologically trivial isomorphisms. 
\begin{proposition}\label{prop:borelredprod}
Let $A_n,B_n$ be unital separable $\Cstar$-algebras. Suppose there are a sequence $\epsilon_n\to 0$ and maps $\phi_n\colon A_n\to B_n$ such that each $\phi_n$ is an $\epsilon_n$-$^*$-isomorphism. Then the isomorphism $\prod A_n/\bigoplus A_n\to \prod B_n/\bigoplus B_n$ given by quotienting $\prod \phi_n$ is topologically trivial.
\end{proposition}
\begin{proof}
For every $n$ there is a (norm-norm) Borel $\psi_n\colon A_n\to B_n$ such that $\|\phi_n-\psi_n\|<2^{-n}$. Therefore $\prod\psi_n$ and $\prod\phi_n$ induce the same isomorphisms between reduced products. Since the strict topology on $\prod A_n$ is the product of the norm topologies on the $A_n$'s (and the same happens for $B_n$), the map $\prod\psi_n$ is strictly-strictly Borel, and so its graph is Borel. Since $\bigoplus A_n$ is Borel in $\prod A_n$, we are done.
\end{proof}
\section{Filtrations of coronas and liftings}\label{section:general}
We fix separable nonunital $\Cstar$-algebras $A$ and $B$ and a $^*$-homomorphism \[
\Lambda\colon \mathcal Q(A)\to\mathcal Q(B).
\]
 $\pi_A\colon\mathcal M(A)\to\mathcal Q(A)$ and $\pi_B\colon\mathcal M(B)\to\mathcal Q(B)$ denote the quotient maps, with $\pi_0\colon\ell_\infty\to\ell_\infty/c_0$. If $S\subseteq\NN$, $\chi_S\in\ell_\infty$ is the projection corresponding to the characteristic function of $S$.

 In \S\ref{sub:prep} we study general properties of $\Lambda$ setting the stage for the proofs of our main results. In \S\ref{sub:thmalltriv} we prove Theorem~\ref{thmi:alltrivialnoncomm}.
The axioms $\OCA$ and $\MA_{\aleph_1}$ are assumed everywhere, mainly (but not only) to have access to Theorem~\ref{thm:lifting}.

\subsection{Structure of $^*$-homomorphisms between coronas}\label{sub:prep}
Let $\{\tilde e_n\}$ be an approximate identity of positive contractions for $A$ with the following properties:
\begin{itemize}
\item $\tilde e_{n+1}\tilde e_n=\tilde e_n$ for all $n$;
\item if $I\subseteq \NN$ is a finite interval, there is a positive contraction $h_I\leq (\tilde e_{\max I+1}-\tilde e_{\min I-2})$ such that $h_I(\tilde e_{\max I}-\tilde e_{\min I-1})=(\tilde e_{\max I}-\tilde e_{\min I-1})$ and with the property that $h_Ih_J=0$ whenever $\max I+1<\min J$.
\end{itemize}
(Such an approximate identity can be obtained from any approximate identity of positive contractions with $\tilde{\tilde {e}}_{n+1}\tilde{\tilde {e}}_{n}=\tilde{\tilde {e}}_{n}$, by letting $\tilde e_n=\tilde{\tilde {e}}_{3n}$ and $h_I=\tilde{\tilde {e}}_{3\max I+1}-\tilde{\tilde {e}}_{3(\min I-1)-1}$).
We also fix $\{e_n^B\}$, an approximate identity of contractions for $B$ such that $e_{n+1}^Be_n^B=e_n^B$.

Let $\tilde\rho_1,\tilde\rho_2\colon\ell_\infty/c_0\to\mathcal Q(B)$ be defined as
\begin{eqnarray*}
\tilde\rho_1(\pi_0(\chi_S))=&\Lambda(\pi_A(\sum_{n\in S}(\tilde {e}_{2n+1}-\tilde {e}_{2n}))),\\
\tilde\rho_2(\pi_0(\chi_S))=&\Lambda(\pi_A(\sum_{n\in S}(\tilde {e}_{2n}-\tilde {e}_{2n-1}))).
\end{eqnarray*}
\begin{lemma}
$\tilde\rho_1$ and $\tilde\rho_2$ preserve the coordinate algebra.
\end{lemma}
\begin{proof}
Let $q_S$, for $S\subseteq\NN$, be positive contractions such that 
\[
\pi_B(q_S)=\Lambda(\pi_A(\sum_{n\in S}h_{\{2n+1\}})).
\]
 Then $\{q_S\colon S\subseteq\NN\}$ satisfies Definition~\ref{defin:preserving} for $\tilde\rho_1$. The case of $\tilde\rho_2$ is analogous.
\end{proof}
By Theorem~\ref{thm:lifting} and thanks to $\OCA+\MA_{\aleph_1}$, there are asymptotically additive maps  (w.r.t. the approximate identity $e^B_n$)
\[
\rho_1=\sum\rho_{1,n}\colon \ell_\infty\to\mathcal M(B)\text{, and } \rho_2=\sum\rho_{2,n}\colon \ell_\infty\to\mathcal M(B)
\] which lift $\tilde\rho_1$ and $\tilde\rho_2$ on nonmeager dense ideals $\SI_1$ and $\SI_2$. Let 
\[
\rho'_{2n}=\rho_{2,n}\text{ and }\rho'_{2n+1}=\rho_{1,n},
\]
and
\[
\rho\colon\ell_\infty\to \mathcal M(B)\text{ defined by }\rho=\sum \rho'_n.
\]
Since $\rho_1$ and $\rho_2$ are asymptotically additive, so is $\rho$. Let
\[
p_n=\rho'_{n}(\chi_{\{n\}})
\]
and
\[
\SI_\rho=\{S\subseteq\NN\colon  S\cap\{n\colon n\text{ is even}\}\in\SI_2\text{ and }S\cap\{n\colon n\text{ is odd}\}\in\SI_1\}.
\]
$\SI_\rho$ is a nonmeager dense ideal and if $S\in\SI_\rho$, then 
\[
\Lambda(\pi_A(\sum_{n\in S}(\tilde{e}_n-\tilde{e}_{n-1})))=\pi_B(\sum_{n\in S}p_n).
\]
\begin{lemma}\label{lem:structural}
There is an increasing sequence $\{M_n\}$ such that if $M_n\leq \ell_1\leq M_{n+1}$ and $M_m\leq \ell_2\leq M_{m+1}$ for $|n-m|\geq 2$ then $p_{\ell_1}p_{\ell_2}=0$.
\end{lemma}
\begin{proof}
First note that, by definition of $\rho$, for every $n$ there are $j_n<k_n$ such that $p_{n}\in (e^B_{k_n}- e^B_{j_n})B(e^B_{k_n}-e^B_{j_n})$, where $j_n\to\infty$ as $n\to\infty$. We define the sequence $M_n$ inductively. Let $M_0=0$; if $M_n$ has been defined let 
\[
M_{n+1}=\min\{i'> M_n\colon \forall i\geq i' (j_i>\max\{k_\ell\colon \ell\leq M_n\})\}.
\]
As $j_n\to\infty$ as $n\to\infty$ the sequence $\{M_n\}$ is well defined. Suppose that $M_n\leq \ell_1<M_{n+1}<M_m\leq \ell_2$. Then $j_{\ell_1}<k_{\ell_1}<j_{\ell_2}<k_{\ell_2}$ by the choice of $M_{n+2}$, and therefore 
\[
p_{\ell_1}p_{\ell_2}\in (e^B_{k_{\ell_1}}-e^B_{j_{\ell_1}})B(e^B_{k_{\ell_1}}-e^B_{j_{\ell_1}})(e^B_{k_{\ell_2}}-e^B_{j_{\ell_2}})B(e^B_{k_{\ell_1}}-e^B_{j_{\ell_1}})=0.
\qedhere\]
\end{proof}

From now on we will use the following notation:
\begin{notation}\label{not:autos}
\begin{itemize}
\item $e_n=\tilde{e}_{M_n}$, $e_{-1}=0$;
\item if $I\subseteq\NN$ is finite, define $f_I=e_{\max I}-e_{\min I-1}$ and $A_I=\overline{f_IAf_I}$. If $I=[n,m]$ for $n>0$, let $I^+=[n-1,m+1]$;
\item if $I\subseteq\NN$ is an interval, let $r_I=\sum_{k\in [M_{\min I-1},M_{\max I+1})}p_k$;
\item if $G=\bigcup I_n$ where each $I_n$ is finite and $\max I_n+2<\min I_{n+1}$, let $ r_G=\sum r_{I_n}$; 
\item the ideal $\SI_\rho$ is fixed as before Lemma~\ref{lem:structural};
\item $\mathbb P$ is the set of all partition of $\NN$ into consecutive finite intervals $\mathbb I=\langle I_n\colon n\in\NN\rangle$ with $|I_n|\geq 3$ for all $n$. We order $\mathbb P$ by 
\[
\mathbb I\leq_1\mathbb J\iff\exists n_0\forall n\geq n_0\exists m (I_{n}\cup I_{n+1}\subseteq J_m\cup J_{m+1});
\]
\item If $\mathbb I=\langle I_n\colon n\in\NN\rangle\in\mathbb P$, let
\[
I^e_n=I_{2n}\cup I_{2n+1},\,\,\, I^o_n=I_{2n+1}\cup I_{2n+2}
\]
and
\[
\mathbb I^{e}=\langle I^e_{n}\colon n\in\NN\rangle \text{ and }\mathbb I^{o}=\langle I^o_{n}\colon n\in\NN\rangle.
\]
Here the letters $e$ and $o$ refer to even and odd partitions.
\end{itemize}
\end{notation}
\begin{proposition}\label{prop:properties1}
Using Notation~\ref{not:autos} we have:
\begin{enumerate}
\item $I^+\cap J^+=\emptyset$ implies $A_IA_J=0$ whenever $I,J\subseteq\NN$
\item if $x\in A_I$ then $xf_{I^+}=f_{I^+}x=x$
\item if $\{I_n\}$ is a sequence of finite intervals with $\max I_n+2<\min I_{n+1}$, $x\in\prod A_{I_n}$, and $\bigcup_n[M_{\min I_n}-1,M_{\max I_n}+1]\in\SI_\rho$ then $r_{\bigcup I_n}\Lambda(\pi_A(x))=\Lambda(\pi_A(x))$.
\item For every $x\in\mathcal M(A)$ there is $\mathbb I\in\mathbb P$, $x_e\in\prod A_{I^e_n}$ and $x_o\in \prod A_{I_n^o}$ such that $x-x_e-x_o\in A$.
\end{enumerate}
\end{proposition}
\begin{proof}
(1) and (2) follow from the definition. For (3), since 
\[
H=\bigcup_n[M_{\min I_n}-1,M_{\max I_n}+1]\in\SI_\rho
\] then 
\[
\pi_B(r_{\bigcup I_n})=\pi_B(\sum_n r_{I_n})=\pi_B(\sum_{k\in H}p_k)=\Lambda(\pi_A(y))
\]
where $y=\sum_{k\in H}(\tilde{e}_k-\tilde{e}_{k-1})$.
Since $yx=x$ the thesis follows. (4) is \cite[Lemma 2.6]{MKAV.FA} (see also \cite[Theorem~3.1]{Elliott.Der2}, \cite[Lemma 1.2]{Farah.C}).
\end{proof}
Fix finite sets of contractions $F_n\subseteq A$ with the following properties:
\begin{itemize}
\item $F_n\subseteq F_{n+1}$ for all $n$, and $0\in F_0$,
\item $\bigcup F_n$ is dense in the unit ball of $A$, and
\item each $F_n$ is linearly ordered by $<_n$, where $0$ is the $<_n$-minimal element.
\end{itemize} 
Denote by $E(F_n)$ the finite-dimensional Banach subspace of $A$ generated by $F_n$, and let 
\[
U_{n,m}=\{x\colon\exists y\in F_m, \norm{x-y}\leq2^{-n}\}.
\]
This is the open neighborhood of $F_m$ of radius $2^{-n}$. Note that $U_{n,m}\subseteq U_{n,m'}$ if $m\leq m'$, since $F_m\subseteq F_{m'}$. Let $\phi_{n,m}\colon U_{n,m}\to E(F_m)$ be obtained by sending $x$ to $<_m$-minimum element of $F_m$ minimizing $\norm{x-y}$, and let $\psi_n\colon E(F_n)\to A$ be the inclusion map.

(Let $E_{n,m}=E(F_m)$. Fix $\bar Y=(Y_{n,m},<_{n,m})$ a dense system for $\{E_{n,m}\}$ as in Definition~\ref{defin:skeletal2}. From now on, we will refer to skeletal maps w.r.t. this fixed dense system $\bar Y$, and omit the reference to $\bar Y$.)

If $g\in\NN^{\NN\uparrow}$, $\mathbb I\in\mathbb P$, $i\in \{e,o\}$ and $j\in\{0,1\}$, define
\[
B(g,\mathbb I,i,j)=\prod(U_{n,g(n)}\cap A_{I_{2n+j}^i}).
\]
Noticee that $B(g,\mathbb I,i,j)\cap A=\bigoplus (U_{n,g(n)}\cap A_{I_{2n+j}^i})$. Let
\[
\Phi^{i,j}_{g,\mathbb I}\colon\prod B(g,\mathbb I,i,j)\to\prod E_{n,g(n)}
\]
be defined as 
\[
\Phi^{i,j}_{g,\mathbb I}(x)=\prod\phi_{n,g(n)}(x),
\]
and let
\[
\Psi^{i,j}_{g,\mathbb I}\colon E_{n,g(n)}\to\prod A_{(I^i_{2n+j})^+}
\]
be defined as 
\[
\Psi^{i,j}_{g,\mathbb I}((x_n)_n)=(f_{(I^i_{2n+j})^+}\psi_{g(n)}(x_n)f_{(I^i_{2n+j})^+})_n.
\]
The map $\Psi^{i,j}_{g,\mathbb I}$ is linear and contractive, and it induces a well defined map 
\[
\tilde\Psi^{i,j}_{g,\mathbb I}\colon \prod E_{n,g(n)}/\bigoplus E_{n,g(n)}\to\prod A_{(I^i_{2n+j})^+}/\bigoplus A_{(I^i_{2n+j})^+}.
\]
(It is not necessary to appeal to $\mathbb P$ when considering abelian $\Cstar$-algebras, and that the use of the index $j$ is redundant in case $A$ has an approximate identity of projections.)
We leave the proof of the following to the reader.
\begin{proposition}\label{prop:properties2}
For every $g\in\NN^\NN$, $\mathbb I\in\mathbb P$, $i\in\{e,o\}$ and $j\in\{0,1\}$:
\begin{enumerate}
\item\label{prop:propBorel} each $\Phi^{i,j}_{g,\mathbb I}\circ\Psi^{i,j}_{g,\mathbb I}$ is strict-strict-Borel
\item\label{prop:propWelldefined} if $x,y\in B(g,\mathbb I,i,j)$ are such that $x-y\in A$ then 
\[
\Phi^{i,j}_{g,\mathbb I}(x)-\Phi^{i,j}_{g,\mathbb I}(y)\in\bigoplus E_{n,g(n)}.
\]
In particular $\Phi^{i,j}_{g,\mathbb I}[B(g,\mathbb I,i,j)\cap A]\subseteq\bigoplus E_{n,g(n)}$, so $\Phi^{i,j}_{g,\mathbb I}$ induces a well defined map 
\[
\tilde\Phi^{i,j}_{g,\mathbb I}\colon B(g,\mathbb I,i,j)/(B(g,\mathbb I,i,j)\cap A) \to \prod E_{n,g(n)}/\bigoplus E_{n,g(n)};
\]
\item the map $\tilde\Psi^{i,j}_{g,\mathbb I}$ preserves the coordinate function (as in Definition~\ref{defin:preserving});
\item For all $x\in B(g,\mathbb I,i,j)$, we have that 
\[
\Psi^{i,j}_{g,\mathbb I}\circ\Phi^{i,j}_{g,\mathbb I}(x)-x\in A.
\]
\end{enumerate}
\end{proposition}

The following is an extension of Proposition~\ref{prop:properties1}(4).
\begin{lemma}\label{lemma:filtration}
Let $x\in\mathcal M(A)$. Then there are $\mathbb I\in\mathbb P$, $g\in\NN^{\NN\uparrow}$ and $x_{i,j}\in B(g,\mathbb I,i,j)$, for $i\in\{e,o\}$ and $j\in\{0,1\}$, such that 
\[
x-\sum_{i,j}\Psi^{i,j}_{g,\mathbb I}\circ\Phi^{i,j}_{g,\mathbb I}(x_{i,j})\in A.
\]
If $x$ is a contraction, so is each $x_{i,j}$.
\end{lemma}
\begin{proof}
By Proposition~\ref{prop:properties1}(4), it  is enough to show that if $x=(x_n)\in \prod A_{I^i_{2n+j}}$ then there is $g$ such that $x_n\in U_{n,g(n)}$ for all $n$. Since $\bigcup F_n$ is dense in $A$, for all $n$ we can find $g(n)$ and $y\in F_{g(n)}$ such that $\norm{x_n-y}<2^{-n}$, so $x_n\in U_{n,g(n)}$. In particular 
\[
\prod A_{I^i_{2n+j}}=\bigcup_g B(g,\mathbb I,i,j).\qedhere
\]
\end{proof}
If $\mathbb I\in\mathbb P$ and $g\in\NN^{\NN\uparrow}$, for $i\in\{e,o\}$ and $j\in\{0,1\}$, define
\[
\Lambda^{i,j}_{g,\mathbb I}=\Lambda\circ\tilde\Psi^{i,j}_{g,\mathbb I}\colon \prod E_{n,g(n)}/\bigoplus E_{n,g(n)}\to\mathcal Q(B).
\]
Each $\Lambda^{i,j}_{g,\mathbb I}$ preserves the coordinate function (for $S\subseteq\NN$, let $q_S\in \mathcal M(B)$ be any positive contraction lifting $\Lambda(\pi_A(\sum_{n\in S}f_{(I_{2n+j}^i)^{++}}))$. Since $|I_n|\geq 3$, $\{q_S\colon S\subseteq\NN\}$ satisfies Definition~\ref{defin:preserving}).
By Theorem~\ref{thm:lifting}, there is a skeletal map $\alpha=\alpha(\mathbb I,g,i,j)$ with $\alpha\colon\prod E_{n,g(n)}\to\mathcal M(B)$ and a nonmeager dense ideal $\SI=\SI_{\alpha,g,\mathbb I,i,j}$ such that $\alpha$ is a lift of $\Lambda^{i,j}_{g,\mathbb I}$ on elements with support in $\SI$. Note that if $x=(x_n)\in B(g,\mathbb I,i,j)$ with $\supp(x)\in \SI$ then
\[
\pi_B(\alpha(\Phi^{i,j}_{g,\mathbb I}(x)))=\Lambda^{i,j}_{g,\mathbb I}(\tilde\Phi^{i,j}_{g,\mathbb I}(\pi_A(x)))=\Lambda(\tilde\Psi^{i,j}_{g,\mathbb I}(\tilde\Phi^{i,j}_{g,\mathbb I}(\pi_A(x))))=\Lambda(\pi_A(x)).
\]
If $\alpha=\alpha(\mathbb I,g,i,j)$ is as above and $F\subseteq\NN$, let $\alpha_F=\sum_{n\in F}\alpha_n$.

\begin{lemma}\label{lemma:cut1}
Let $g\in\NN^\NN$, $\mathbb I\in\mathbb P$, $i\in \{e,o\}$ and $j\in\{0,1\}$. Suppose that $\alpha\colon\prod E_{n,g(n)}\to\mathcal M(B)$ is an asymptotically additive lift of $\Lambda^{i,j}_{g,\mathbb I}$ on elements with support in a nonmeager dense ideal $\SI$. Then for every $\epsilon>0$ there is $n_0\in\NN$ such that, if $F\subseteq\NN$ is finite  with $\min F>n_0$ then, with $r=r_{\bigcup_{n\in F}I_{2n+j}^i}$, we have
\[
\norm{r\alpha_F(\Phi_{g,\mathbb I}^{i,j}(x))-\alpha_F(\Phi_{g,\mathbb I}^{i,j}(x))}<\epsilon
\]
and
\[
\norm{\alpha_F(\Phi_{g,\mathbb I}^{i,j}(x))r-\alpha_F(\Phi_{g,\mathbb I}^{i,j}(x))}<\epsilon
\]
for all $x\in\prod_{n\in F}(U_{n,g(n)}\cap A_{I_{2n+j}^i})$.
\end{lemma}
\begin{proof}
Let $\Phi=\Phi^{i,j}_{g,\mathbb I}$ and $J_n=I_{2n+j}^i$. To simplify the notation, for $S\subseteq\NN$, let 
\[
q_S=r_{\bigcup_{n\in S}J_n}.
\]
Suppose towards a contradiction that there are finite sets $F_k$ with $\max F_k< \min F_{k+1}$ and $x_k\in\prod_{F_k}(U_{n,g(n)}\cap A_{J_n})$ with 
\[
\norm{q_{F_k}\alpha_{F_k}(\Phi(x_k))-\alpha_{F_k}(\Phi(x_k))}>\epsilon
\]
for all $k$. Recall that an ideal $\SJ$ of $\NN$ which contains all finite sets is nonmeager if and only if for any given strictly increasing sequence of naturals$\{k_i\}$ there is an infinite $L$ with $\bigcup_{i\in L}[k_i,k_{i+1})\in \SJ$, see e.g. \cite[Proposition 2.4]{MKAV.FA}. Therefore, without loss of generality since both $\SI_\rho$ and $\SI$ are nonmeager and dense we can assume that 
\[
\bigcup F_k\in\SI\text{ and }\bigcup_k[M_{\min J_{\min F_k}-1},M_{\max J_{\max F_k}+1}]\in\SI_\rho.
\]
 By passing to a subsequence, since the support of each $\alpha_{F_k}(\Phi(x_k))$ and $q_{F_k}$ is finite, we can assume that 
\[
\alpha_{F_k}(\Phi(x_k))\alpha_{F_{k'}}(\Phi(x_{k'}))=\alpha_{F_k}(\Phi(x_k))q_{F_{k'}}=q_{F_k}q_{F_{k'}}=0,\]
for $k\neq k'$. Let $G=\bigcup F_k$, $q=\sum_kq_{F_k}$ and $x=(x_k)\in B(g,\mathbb I,i,j)$, where $\supp x\in\bigcup F_k$. 
Then 
\[
\norm{\pi_B(\alpha_G(\Phi_{2n+j}^j(x))-q\alpha_G(\Phi_{2n+j}^j(x)))}>\epsilon.
\]
Noticing that 
\[
q=\sum_k q_{F_k}=\sum_kr_{\bigcup_{n\in F_k}J_n}=r_{\bigcup_{n\in F_k}I_{2n+j}^i}.
\]
and by our choice of $G$, $q$, $\SI$, and $\SI_\rho$, this contradicts Proposition~\ref{prop:properties1}(3). This finishes the proof.
\end{proof}

\begin{proposition}\label{prop:cut2}
Let $g\in\NN^\NN$, $\mathbb I\in\mathbb P$, $i\in \{e,o\}$ and $j\in\{0,1\}$. Suppose that $\alpha\colon\prod E_{n,g(n)}\to\mathcal M(B)$ is an asymptotically additive lift of $\Lambda^{i,j}_{g,\mathbb I}$ on elements with support in a nonmeager dense ideal $\SI$ and let
\[
\alpha'_n=r_{I_{2n+j}^i}\alpha_nr_{I_{2n+j}^i}, \,\,\, \alpha'=\sum\alpha'_n.
\]
Then $\alpha'$ is a lift of $\Lambda^{i,j}_{g,\mathbb I}$ on elements of the form $\Phi^{i,j}_{g,\mathbb I}(x)$ for $x\in B(g,\mathbb I,i,j)$ and such that the support of $x$ is in $\SI$.

Moreover if $\alpha$ is skeletal, so is $\alpha'$.
\end{proposition}
\begin{proof} Let $\Phi=\Phi^{i,j}_{g,\mathbb I}$ and $J_n=I_{2n+j}^i$. As before, let 
\[
q_S=r_{\bigcup_{n\in S}J_n}
\] if $S\subseteq\NN$, so that $q_n=r_{J_n}$. Since $|I_n|\geq 3$, $q_nq_m=0$ whenever $n\neq m$, and therefore $\alpha'$ is well defined and asymptotically additive. Notice that with this notation, 
\[
\alpha'_{F}=\sum_{n\in F}\alpha'_n=\sum_{n\in F}q_n\alpha_nq_n.
\]
 Moreover if each $\alpha_n$ is completely determined by finitely many values, so is $\alpha'_n$, hence if $\alpha$ is skeletal, so is $\alpha'$.

Let $\epsilon>0$ and suppose that $x\in B(g,\mathbb I,i,j)$ is such that 
\[
\norm{\pi_B(\alpha(\Phi(x))-\alpha'(\Phi(x)))}>\epsilon.
\]
Then there are finite intervals $F_k\subseteq\NN$ with $\min F_k\to\infty$ such that 
\begin{equation}\label{eqn1}\tag{$\star$}
\norm{\alpha_{F_k}(\Phi(x))-\alpha_{F_k}'(\Phi(x))}>\epsilon
\end{equation}
 for all $k\in\NN$. By passing to a subsequence we can assume that $\max F_k<\min F_{k+1}$
By Lemma~\ref{lemma:cut1} for all $k$ and all $G_k\subseteq F_k$ we have that
\begin{equation}\label{eqn2}\tag{$\star\star$}
\norm{\alpha_{G_k}(\Phi(x))-q_{G_k}\alpha_{G_k}(\Phi(x))q_{G_k}}<\epsilon^2/1000.
\end{equation}
With
\[
y_k=q_{F_k}\alpha_{F_k}(\Phi(x))q_{F_k}-\sum_{n\in F_k}q_n\alpha_{n}(\Phi(x))q_n=q_{F_k}\alpha_{F_k}(\Phi(x))q_{F_k}-\alpha'_{F_k}(\Phi(x)),
\]
we have therefore that $\norm{y_k}>\epsilon/2$ for all $k$, by \eqref{eqn1}, \eqref{eqn2} and the triangle inequality.
 
 Notice that since $q_{F_k}=\sum_{n\in F_k}q_n$ and $\alpha_{F_k}=\sum_{n\in F_k}\alpha_n$, then
\[
y_k=\sum_{n<m\in F_k}q_n\alpha_{F_k}(\Phi(x))q_m+\sum_{n>m\in F_k}q_n\alpha_{F_k}(\Phi(x))q_m+\sum_{n\in F_k}q_n\alpha_{F_k\setminus\{n\}}q_n.
\]
Suppose that 
\[
z_k=\sum_{n, m \in F_k\colon n<m}q_n\alpha_{F_k}(\Phi(x))q_m
\]
 is such that $\norm{z_k}>\epsilon/10$. Since $q_nq_m=0$ for all $n,m\in F_k$ and $k\in\NN$, we have that 
\[
\norm{z_kz_k^*}=\norm{\sum_n q_n\alpha_{F_k}(\Phi(x))(\sum_{n, m \in F_k, \, n<m}q_m^2)\alpha_{F_k}(\Phi(x))q_n}>\epsilon^2/100
\]
Since all the summands are orthogonal to each other, 
\[
\norm{z_kz_k^*}=\max_n\norm{q_n\alpha_{F_k}(\Phi(x))(\sum_{n, m \in F_k, \, n<m}q_m^2)\alpha_{F_k}(\Phi(x))q_n}>\epsilon^2/100
\]
Pick $\bar n$ such that $\norm{q_{\bar n}\alpha_{F_k}(\Phi(x))(\sum_{m>\bar{n}}q_m^2)\alpha_{F_k}(\Phi(x))q_{\bar n}}>\epsilon^2/100$. By~(\ref{eqn1}) we have that 
\[
\norm{\alpha_{\bar n}(\Phi(x))-\alpha_{\bar n}(\Phi(x))q_{\bar n}}<\epsilon^2/1000
\]
 and 
 \[
 \norm{\alpha_{F_k\setminus\bar n}(\Phi(x))-q_{F_k\setminus\bar n}\alpha_{F_k\setminus\bar n}(\Phi(x))}<\epsilon^2/1000.
 \] 
The fact that $\alpha_{F_k}(\Phi(x))=\alpha_{F_k\setminus\bar n}(\Phi(x))+\alpha_{\bar n}(\Phi(x))$ and the triangle inequality lead to a contradiction. The same arguments shows that 
\[
\norm{\sum_{n>m\in F_k}q_n\alpha_{F_k}(\Phi(x))q_m}, \norm{\sum_{n\in F_k}q_n\alpha_{F_k\setminus n}(\Phi(x))q_n}<\epsilon/10,
\]
 giving therefore $\norm{y_k}<\epsilon/2$, a contradiction.
\end{proof}

Suppose that $g\in\NN^\NN$ and $\mathbb I\in\mathbb P$. By Theorem~\ref{thm:lifting} and Propositions~\ref{prop:cut2}, there are maps $\alpha^{i,j}$ and nonmeager dense ideals $\SI_{i,j}$, for $i\in\{e,o\}$ and $j\in\{0,1\}$, such that $\alpha^{i,j}\colon\prod E_{n,g(n)}\to\mathcal M(B)$ is a skeletal lift of $\Lambda$ on elements of $\Phi^{i,j}_{g,\mathbb I}(B(g,\mathbb I,i,j))$ whose support is in $\SI_{i,j}$ where the range of each $\alpha^{i,j}_n$ is included in $r_{I_{2n+j}^i}Br_{I_{2n+j}^i}$. Suppose that $x\in B(g,\mathbb I,e,0)\cap B(g,\mathbb I,e,1)$. Then we can modify $\alpha^{o,0}$ so that 
\[
\alpha^{o,0}(\Phi^{o,0}_{g,\mathbb I}(x))=\alpha^{e,0}(\Phi^{e,0}_{g,\mathbb I}(x)).
\]
Since the finite intersection of finitely many dense nonmeager ideals is still a dense and nonmeager ideal, this can be done pointwise on every $x$, and we would still have that $\alpha^{o,0}$ is a skeletal lift of $\Lambda$ on elements in $\Phi^{o,0}_{g,\mathbb I}(B(g,\mathbb I,o,0))$ whose support is in a dense and nonmeager ideal. By doing so for every pair $i,j$ we can assume that
\begin{equation}\label{eqn22}\tag{$\star\star$}
x\in B(g,\mathbb I,i,j)\cap B(g,\mathbb I,i',j')\Rightarrow \alpha^{i,j}(\Phi^{i,j}_{g,\mathbb I}(x))=\alpha^{i',j'}(\Phi^{i',j'}_{g,\mathbb I}(x))
\end{equation}

Recall that the finite-dimensional Banach spaces $E_{n,m}$ were define before Proposition~\ref{prop:properties2}.
\begin{notation}\label{notation:thespace}
Let $\mathcal X=\{(g,\mathbb I,\bar\alpha)\}$ where 
\begin{itemize}
\item $g\in\NN^{\NN\uparrow}$, $\mathbb I\in\mathbb P$ and, $\bar\alpha=(\alpha^{i,j})_{i\in\{e,o\},j\in\{0,1\}}$,
\item $\alpha^{i,j}\colon \prod E_{n,g(n)}\to\mathcal M(B)$ is a skeletal map,
\item $\alpha^{i,j}$ is a lift of $\Lambda$ on elements of $\Phi^{i,j}_{g,\mathbb I}(B(g,\mathbb I,i,j))$ whose support is a nonmeager dense ideal $\SI_{g,\mathbb I,\alpha,i,j}$, where the range of each $\alpha^{i,j}_n$ is included in $r_{I_{2n+j}^i}Br_{I_{2n+j}^i}$, and
\item $\alpha^{i,j}$ respect~(\ref{eqn22}).
\end{itemize}
\end{notation}

For all $g\in\NN^{\NN\uparrow}$ and $\mathbb I\in\mathbb P$ there is $\bar\alpha=(\alpha^{i,j})_{i,j}$ such that $(g,\mathbb I,\bar\alpha)\in\mathcal X$.  By Proposition~\ref{prop:skeletalpolish}, $\mathcal X$ has a natural Polish topology.
\begin{lemma}\label{lemma:afteracertainpoint}
Let $(g,\mathbb I,\bar \alpha),(g',\mathbb J,\bar\beta)\in \mathcal X$. Then for all $\epsilon>0$ there is $n>0$ such that for all $m>0$ such that if $x\in A_{[n,m]}\cap  B(g,\mathbb I,i,j)\cap B(g',\mathbb J,i',j')$
for some $i,i'\in\{e,o\}$ and $j,j'\in\{0,1\}$ then 
\[
\norm{\alpha^{i,j}(\Phi^{i,j}_{g,\mathbb I}(x))-\beta^{i',j'}(\Phi^{i',j'}_{g',\mathbb J}(x))}<\epsilon.
\]
\end{lemma}
\begin{proof}
Suppose that there are $(g,\mathbb I,\bar\alpha), (g',\mathbb J,\bar\beta)$, $\epsilon>0$, sequences $n_k<m_k$, and contractions $x_k\in A_{[n_k,m_k]}$ leading to a contradiction where $n_k\to\infty$. By going to a subsequence we can assume that $n_k<m_k<n_{k+1}$ for all $k$, and, by nonmeagerness of $\SI_{g,\mathbb I,\alpha,i,j}\cap\SI_{g',\mathbb J,\beta,i,j}$, that $\bigcup [n_k,m_k]\in \SI_{g,\mathbb I,\alpha,i,j}\cap\SI_{g',\mathbb J,\beta,i,j}$. Let $x=\sum x_k$. Since $x_kx_{k'}=0$, $x$ is a contraction and $x\in B(g,\mathbb I,i,j)\cap B(g',\mathbb J,i',j')$, but
\[
\norm{\pi_B(\alpha^{i,j}(\Phi^{i,j}_{g,\mathbb I}(x))-\beta^{i',j'}(\Phi^{i',j'}_{g',\mathbb J}(x)))}\geq\epsilon,
\]
a contradiction to the choice of $\SI_{g,\mathbb I,\alpha,i,j}$ and $\SI_{g',\mathbb J,\beta,i,j}$.
\end{proof}
We define the main colouring to which we will apply $\OCA_{\infty}$.
Let  
\[
[\mathcal X]^2=L_0^n\sqcup L_1^n
\] where 
\[
\{(g,\mathbb I,\bar\alpha),(g',\mathbb J,\bar\beta)\}\in L_0^n
\]
 if and only if there are $m\in\NN$, $i,i'\in\{e,o\}$, $j,j'\in\{0,1\}$ and a positive contraction 
\[
x\in e_mAe_m\cap B(g,\mathbb I,i,j)\cap B(g,\mathbb J,i',j')
\]
with
\[
\norm{\alpha^{i,j}(\Phi^{i,j}_{g,\mathbb I}(x))-\beta^{i',j'}(\Phi^{i',j'}_{g',\mathbb J}(x))}>2^{-n}.
\]
Each $L_0^n$ is open in the Polish topology of $[\mathcal X]^2$ given by Proposition~\ref{prop:skeletalpolish}. Recall that the order $<_1$ on $\mathbb P$ was defined by 
\[
\mathbb I<_1\mathbb J\iff\exists n_0\forall n\geq n_0\exists m (I_n\cup I_{n+1}\subseteq J_m\cup J_{m+1}).
\]
Recall that if $g$ and $g'$ are functions in $\NN^\NN$, we write $g\leq_* g'$ if there is $n_0$ such that $g(n)\leq g'(n)$ whenever $n\geq n_0$.
\begin{lemma}\label{lemma:OCAfirstalternativecomplicated}
If $\mathfrak b>\omega_1$ then for all $n$ there is no uncountable $L_0^n$-homogeneous set.
\end{lemma}
\begin{proof}
We argue by contradiction and suppose that there is $Z\subseteq\mathcal X$ such that $|Z|=\omega_1$ and $Z$ is $L_0^n$-homogeneous for some $n\in\NN$. Let 
\[
Z'=\{g\colon \exists \mathbb I,\bar\alpha ((g,\mathbb I,\bar\alpha)\in Z)\}.
\]
 Set $\epsilon=2^{-n}$. Since $\mathfrak b>\omega_1$ we can find $\hat g$ and $\hat {\mathbb I}$ such that if $(g,\mathbb I,\bar\alpha)\in Z$ then $\mathbb I<_1\hat{\mathbb I}$ and $g\leq_*\hat g$.
By the definition of $\leq_*$ and $\leq_1$, we can refine $Z$ to an uncountable subset of it (which we will call $Z$ to avoid redundant notation) such that there is $n_0$ so that whenever $(g,\mathbb I,\bar\alpha)\in Z$ then if $n\geq n_0$ then $g(n)<\hat g(n)$ and there is $m$ (depending on $\mathbb I)$ such that $I_n\cup I_{n+1}\subseteq \hat I_m\cup \hat I_{m+1}$.

Using Theorem~\ref{thm:lifting} and Propositions~\ref{prop:cut2} and \ref{prop:liftonall}, fix $\hat{\bar\beta}$ such that $(\hat g,\hat {\mathbb I},\hat{\bar\beta})\in\mathcal X$.
By Lemma~\ref{lemma:afteracertainpoint}, for all $(g,\mathbb I,\bar\alpha)\in Z$ there is $n_g\geq n_0$ such that whenever $m>n_g$, $i\in\{e,o\}$, $j\in\{0,1\}$ and a contraction $x\in A_{[n_g,m]}\cap B(g,\mathbb I,i,j)\cap B(g,\hat{\mathbb I},i',j')$ for some $i,i'\in\{e,o\}$ and $j,j'\in\{0,1\}$ then 
\begin{eqnarray}\label{eqo1}
\norm{\alpha^{i,j}(\Phi^{i,j}_{g,\mathbb I}(x))-\hat \beta^{i',j'}(\Phi^{i',j'}_{\hat g,\hat{\mathbb I}}(x))}<\frac{\epsilon}{4}.
\end{eqnarray}
Again refining $Z$ to an uncountable subset of it, we can assume that $n_1=n_g=n_{g'}$ for all $g,g'\in Z'$, and, by a further refinement, that for all $n\leq 4n_1+8$ and $(g,\mathbb I,\bar\alpha),(g',\mathbb J,\bar\beta)\in Z$ we have that $g(n)=g'(n)$ and $I_n=J_n$. Since $I^i_{2n+j}=J^i_{2n+j}$ whenever $n\leq n_1$ and the spaces of all skeletal maps $\gamma\colon\prod_{n\leq 2n_1} E_{n,g(n)}\to A$ is separable, we can assume that 
\begin{eqnarray}\label{eqo2}
\norm{\alpha^{i,j}_n-\beta^{i,j}_n}<\frac{\epsilon}{2}.
\end{eqnarray}
for all $i,j$, whenever $(g,\mathbb I,\bar\alpha),(g',\mathbb J,\bar\beta)\in Z$. This is our final refinement.  

Let $\overline m=\max I_{2n_1+2}$ for some (indeed, any!) $\mathbb I$ such that $(g,\mathbb I,\bar\alpha)\in Z$. Suppose that $(g,\mathbb I,\bar\alpha)$ and $(g',\mathbb J,\bar\beta)$ are in $Z$. Let  $k'\in\NN$, $i,i'\in\{e,o\}$, $,j,j'\in\{0,1\}$ and $x\in e_{k'}Ae_{k'}$ be witnessing that $\{(g,\mathbb I,\bar\alpha),(g',\mathbb J,\bar\beta)\}\in L_0^n$, $x$ is a positive contraction such that
\[
x\in B(g,\mathbb I,i,j)\cap B(g,\mathbb J,i',j')
\]
with
\[
\norm{\alpha^{i,j}(\Phi^{i,j}_{g,\mathbb I}(x))-\beta^{i',j'}(\Phi^{i',j'}_{g',\mathbb J}(x))}>2^{-n}.
\]
Since $\alpha^{i,j}(\Phi^{i,j}_{g,\mathbb I}(x))$ agrees with $\alpha^{i',j'}(\Phi^{i,j}_{g,\mathbb I}(x))$ and similarly $\beta^{i',j'}(\Phi^{i',j'}_{g',\mathbb J}(x))$ agrees with $\beta^{i,j}(\Phi^{i',j'}_{g',\mathbb J}(x))$ by definition of $\mathcal X$ (see Equation~(\ref{eqn22})), we can assume that $i=i'$ and $j=j'$. Fix $k<k'$ with $x\in A_{[k,k']}$. We will get a contradiction by cases:
\begin{itemize}
\item if $k'\leq \overline m$, then by our choice of $\bar m$, we have that $x\in \prod_{n\leq n_1} A_{I^i_{2n+j}}\cap \prod_{n\leq n_1} A_{J^{i}_{2n+j}}$. Since $n_1\leq \bar m=\max I_{2n_1+2}$, we have that $\Phi_{g,\mathbb I}^{i,j}(x)=\Phi_{g',\mathbb J}^{i,j}(x)$ and $\norm{\alpha^{i,j}_{[0,n_1]}-\beta^{i,j}_{[0,n_1]}}<\epsilon/2$, and we have therefore a contradiction to that $x$ is a witness for $\{(g,\mathbb I,\bar\alpha),(g',\mathbb J,\bar\beta)\}\in L_0^n$.
\item if $k\geq n_1$. Notice that by our choice of $n_1$, we have that 
\[
\norm{\alpha^{i,j}(\Phi^{i,j}_{g,\mathbb I}(x))-\hat \beta^{i,j}(\Phi^{i,j}_{\hat g,\hat{\mathbb I}}(x))}<\frac{\epsilon}{4}
\]
and
\[
\norm{\beta^{i,j}(\Phi^{i,j}_{g',\mathbb J}(x))-\hat \beta^{i,j}(\Phi^{i,j}_{\hat g,\hat{\mathbb I}}(x))}<\frac{\epsilon}{4}
\]
by Equation~\eqref{eqo1}. The triangle inequality leads to a contradiction.
\item if $k<n_1< \overline m<k'$ then since $n_1<\max I_{n_1}<\max I_{2n_1}<\overline m$ and $x\in\prod A_{I_{2n+j}^i}$ there is $k''$ such that $k<n_1<k''< \overline m<k'$ and contractions $x_1\in\prod A_{I_{2n+j}^i}\cap A_{[k,k'']}$, $x_2\in \prod A_{I_{2n+j}^i}\cap A_{[k'',k']}$ such that $x=x_1+x_2$. In particular
\[
\norm{\alpha^{i,j}(\Phi^{i,j}_{g,\mathbb I}(x))-\beta^{i,j}(\Phi^{i,j}_{g',\mathbb J}(x))}=\max_{\ell=1,2}\{\norm{\alpha^{i,j}(\Phi^{i,j}_{g,\mathbb I}(x_\ell))-\beta^{i,j}(\Phi^{i,j}_{g',\mathbb J}(x_\ell))}\}.
\]
The first two observations lead to a contradiction.\qedhere
\end{itemize}
\end{proof}

The nonexistence of uncountable $L_n^0$-homogeneous sets in $\mathcal X$ has several consequences: we will use it to obtain all our main results.
\subsection{The proof of Theorem~\ref{thmi:alltrivialnoncomm}}\label{sub:thmalltriv}
In this subsection $\Lambda\colon \mathcal Q(A)\to\mathcal Q(B)$ is assumed to be an isomorphism. First, we need a technical lemma.

\begin{lemma}\label{lemma:nonmeagernoncomm}
Let $I_n\subseteq\NN$ be finite intervals such that $\max I_n+2<\min I_{n+1}$, and let $a\in\mathcal M(A)$ and $b\in\prod A_{I_n}$ be positive contractions. If there is a nonmeager dense ideal $\SI$ such that $\pi_A(ab_S)=\pi_A(b_S)$ for all $S\in\SI$, then $\pi_A(ab)=\pi_A(b)$.
Similarly, if $\pi_A(ab_S)=0$ for all $S\in\SI$, then $\pi_A(ab)=0$.
\end{lemma}
\begin{proof}
Suppose not and let $\epsilon>0$ such that $\norm{\pi_A(ab-b)}>\epsilon$. Let $J_n=[\min I_n-1,\min I_{n+1}-1]$. We can then find intervals $K_n$ such that each $J_i$ is a subset of $K_n$ for some $n$, and there is an infinite set $\{n_i\}$ with $\norm{f_{K_{n_i}}(ab-b)f_{K_{n_i}}}>\epsilon/2$. By going to a subsequence we can assume that $n_i+1<n_{i+1}$. By nonmeagerness of $\SI$ we can find an infinite set $L$ such that $L'=\bigcup_{i\in L}\{k\colon J_k\subseteq K_{n_i}\}\in \SI$. Then, since $\pi_A(ab_{L'}-b_{L'})=0$, we have that for all $i\in L$, $\norm{f_{K_{n_i}}(ab-b)f_{K_{n_i}}}<\epsilon/4$, since 
\[
\lim_{i\in L}\norm{f_{K_{n_i}}(ab-b)f_{K_{n_i}}}\leq \norm{(ab_{\{k\colon J_k\subseteq K_i\}}-b_{\{k\colon J_k\subseteq K_i\}})f_{K_{n_i}}}=0.
\]
This is a contradiction.

For the second statement, note that if $\pi_A(ab_S)=0$, then $\pi_A((1-a)b_S)=\pi_A(b_S)$ for all $S\in\SI$. Hence $\pi_A((1-a)b)=\pi_A(b)$, and therefore $\pi_A(-ab)=0$.
\end{proof}

\begin{proposition}\label{prop:liftonall}
Let $\mathbb I\in\mathbb P$, $g\in\NN^\NN$ and $i\in\{e,o\}$, $j\in\{0,1\}$. Suppose that $\alpha=\sum_n\alpha_n\colon\prod E_{n,g(n)}\to \mathcal M(B)$ is an asymptotically additive lift of $\Lambda^{i,j}_{g,\mathbb I}$ on elements with support in a nonmeager dense ideal $\SI$, and that the range of each $\alpha_n$ is included in $r_{I_{2n+j}^i}Br_{I_{2n+j}^i}$. If $x\in B(g,\mathbb I,i,j)$ is positive then 
\[
\pi_B(\alpha(\Phi^{i,j}_{g,\mathbb I}(x)))=\Lambda(\pi_A(x)).
\]
\end{proposition}
\begin{proof}
To simplify the notation, let $\Phi=\Phi^{i,j}_{g,\mathbb I}$, $\Psi=\Psi^{i,j}_{g,\mathbb I}$ and $J_n=I_{2n+j}^i$. For $S\subseteq\NN$, let $q_S=r_{\bigcup_{n\in S}J_n}$ and note that the map $S\mapsto q_S$ is a continuous map $\mathcal P(\NN)\to\mathcal M(B)$ when $\mathcal P(\NN)$ is endowed with the Cantor topology and $\mathcal M(B)$ has the strict topology. Note moreover that for all $x\in B(g,\mathbb I,i,j)$ we have that $q_S\alpha(\Phi(x))=\alpha(\Phi(x_S))$ and $q_S\alpha(\Phi(x_T))\in A$ if $S\cap T$ is finite.

Let $x\in B(g,\mathbb I,i,j)$ be positive, and 
\[
\SI_x=\{S\subseteq\NN\colon q_S(\alpha(\Phi(x))-F(x))\in B\}.
\]
Since $q_S\in A$ when $S$ is finite, $\Fin\subseteq\SI_x$. 

Let $\SJ$ be the ideal of all $T\subseteq\NN$ such that $T=\bigcup T_n$ where each $T_n$ is an interval and $\bigcup_n[M_{\min J_{\min T_n}-1},M_{\max J_{\max T_n}+1}) \in\SI_{\rho}$. By the characterization of nonmeagerness given in \S\ref{subsec:ST}, $\SJ$ is a nonmeager dense ideal, and so is $\SI_1=\SI\cap\SJ$.

\begin{claim}
If $S\in\SI_1$. Then
\begin{enumerate}
\item $q_S\alpha(\Phi(x))-\alpha(\Phi(x_S))\in B$, and
\item $q_SF(x)-F(x_S)\in B$.
\end{enumerate} 
\end{claim}
\begin{proof}
The first statement is obvious by our choice of $q_S$ and $\alpha$. For the second one, fix $S\in\SI_1$. Note that $\pi_B(q_SF(x))=\pi_B(q_SF(x_S)+q_SF(x_{\NN\setminus S}))$. By our choice of $\SI_1$, we have that $q_SF(x_S)=q_S\alpha(\Phi(x_S))=\alpha(\Phi(x_S))=F(x_S)$, so we are left to prove that $\pi_B(q_SF(x_{\NN\setminus S}))=0$.

Pick $p_S$ such that $\Lambda^{-1}(\pi_B(q_S))=\pi_A(p_S)$. Note that, by our choice of $\SI_1$, whenever $T\subseteq \NN\setminus S$ with $T\in\SI_1$ we have that $\pi_A(p_Sx_T)=0$. By applying Lemma~\ref{lemma:nonmeagernoncomm}, since $\SI_1\cap\mathcal P(\NN\setminus S)$ is nonmeager and dense in $\mathcal P(\NN\setminus S)$, we have that $\pi_A(p_Sx_{\NN\setminus S})=0$, hence $\Lambda(\pi_A(p_Sx_{\NN\setminus S}))=\pi_B(q_SF(x_{\NN\setminus S}))=0$ as required.
\end{proof}
Since $S\in\SI_1$ implies that $\alpha(\Phi(x_S))-F(x_S)\in B$, by the claim we have that $\SI_x\supseteq\SI_1$. Since the latter is nonmeager, so is $\SI_x$. Since $\alpha(\Phi(x))$ and $F(x)$ are fixed, the map $S\mapsto q_S$ is continuous, and $B$ is strictly Borel in $\mathcal M(B)$, we have that $\SI_x$ is Borel. Since the only Borel nonmeager ideal containing all finite sets is $\mathcal P(\NN)$ (\cite{Talagrand.Compacts,Jalali-Naini}), $\SI_x=\mathcal P(\NN)$, in particular $q_{\NN}(\alpha(\Phi(x))-F(x))\in B$. By construction of $q_\NN$, we have that $q_\NN\alpha(\Phi(x))=\alpha(\Phi(x))$. Applying again Lemma~\ref{lemma:nonmeagernoncomm} to $p_\NN$ and $x$, we have that $\pi_A(p_\NN x_S))=\pi_A(x_S))$ for all $S\in\SI_1$, and therefore $\pi_A(p_\NN x)=\pi_A(x)$, meaning $q_\NN F(x)-F(x)\in B$. This concludes the proof.
\end{proof}

By  $\OCA$ and the thesis of Lemma~\ref{lemma:OCAfirstalternativecomplicated} applied to each coloring $L_0^n\sqcup L_1^n$, we can find $L_1^n$-homogeneous sets $\mathcal X_{m,n}$ such that $\mathcal X=\bigcup_m \mathcal X_{m,n}$ for all $n$. As $\leq_*\times\leq_1$ is a $\sigma$-directed partial order\footnote{A partial order $\mathbb Q$ is $\sigma$-directed if every countable set has an upper bound} on $\NN^{\NN\uparrow}\times \mathbb P$, by \cite[Lemma 2.2.2 and 2.4.3]{Farah.AQ}, and the fact that $L_0^n\subseteq L_0^{n+1}$ for all $n$, we can find sets $D_k\subseteq\mathcal Y_k\subseteq\mathcal X$ and an increasing sequence $n_k$ such that
\begin{itemize}
\item $\mathcal Y_{k+1}\subseteq \mathcal Y_k$,
\item $D_k$ is countable and dense in $\mathcal Y_k$, and 
\item $\mathcal Y_k$ is $L_1^k$-homogeneous and the set 
\[
\{(g,\mathbb I)\mid\exists\bar\alpha ((g,\mathbb I,\bar\alpha)\in\mathcal Y_k)\}
\]
 is $\leq_{n_k}\times\leq_1$-cofinal in $\NN^\NN\times\mathbb P$.
\end{itemize}
\begin{lemma}\label{lem:almostthere}
Let $i\in\{e,o\}$, $j\in\{0,1\}$ and $k\in\NN$. Let $x\in\mathcal M(A)$ be a contraction, and suppose that there is a sequence $\langle (g_l,\mathbb I_l,\bar\alpha_l)\rangle\subseteq \mathcal Y_k$, and an increasing sequence of natural numbers $N_l>\max (I_l)_{4l+4}$, where $\mathbb I_l=(I_l)_n$, with the following properties:
\begin{enumerate}
\item\label{Borelc1} $e_{N_l}xe_{N_l}\in B(g_l,\mathbb I_l,i,j)$ and 
\item\label{Borelc2} if $l<l'$ and $\max (I_l)_n\leq N_l$ then $(I_l)_n=(I_{l'})_n$ and $g_l(n)=g_{l'}(n)$.
\end{enumerate}
Then 
\[
\norm{\pi_B(\sum (\alpha^{i,j}_{n})_n(\Phi^{i,j}_{g_{n},\mathbb I_{n}}(y_n))-\Lambda(\pi_A(x))}\leq4\cdot2^{-k}.
\]
where 
\[
y_n=(f_{((I_n)^i_{2n+j})^+}xf_{((I_n)^i_{2n+j})^+}).
\]
\end{lemma}
\begin{proof}
Let $J_n=(I_n)_n$.  Since $(I_n)_n=(I_{n+1})_n$, we have $\min (I_{n+1})_{n+1}=\max (I_n)_n+1$, hence $\langle J_n\colon n\in\NN\rangle\in\mathbb P$. Since $x\in\prod A_{J_{2n+i}^j}$, there is $h\in\NN^\NN$ such that $x\in B(h,\mathbb J,i,j)$. By $\leq_*\times\leq_1$-cofinality of $\mathcal Y_k$ there are $\hat g\geq_* h$, $\hat {\mathbb I}\geq_1\mathbb J$ and $\bar\beta$ such that $(\hat g,\hat {\mathbb I},\bar\beta)\in\mathcal Y_k$. If $h(n)\nleq\hat g(n)$, or there is no $m$ such that $J_n\cup J_{n+1}\subseteq\hat I_m\cup\hat I_{m+1}$, set $y_n=0$. Note that this happens only on finitely many indices. For $i'\in\{e,o\}$ and $j\in\{0,1\}$, let $Z_{i',j'}=\{n\colon y_n\in B(\hat g,\hat{\mathbb I},i',j')\}$. Then $\bigcup_{i',j'}Z_{i',j'}=\NN$. Let $Z'_{e,0}=Z_{e,0}$ and
\[
Z'_{e,1}=Z_{e,1}\setminus Z_{e,0},\,Z'_{o,0}=Z_{o,0}\setminus (Z_{e,0}\cup Z_{e,1}),\,Z'_{o,1}=Z_{o,1}\setminus (Z_{e,0}\cup Z_{e,1}\cup Z_{o,0}).
\]
Let $y_{i',j'}=\sum_{n\in Z'_{i',j'}}y_n$. By definition we have that 
\[
\pi_B(\beta^{i',j'}(\Phi_{\hat g,\hat{\mathbb I}}^{i',j'}(y_{i',j'})))=\Lambda(\pi_A(y_{i',j'}))
\]
for all $i'\in\{e,o\}$, $j'\in\{0,1\}$, and $\pi_A(x)=\pi_A(\sum_{i',j'}y_{i',j'})$.
Fix $n$ and suppose that $n\in Z'_{i',j'}$. Let $m$ such that $y_n\in U_{m,\hat g(m)}\cap A_{J_{2m+j'}^{i'}}$. Since $\{(g_{n},\mathbb I_{n},\bar\alpha_{n}),(\hat g,\hat I,\bar\beta)\}\in \mathcal Y_k$ and $\mathcal Y_k$ is $L_1^k$-homogeneous, we have that 
\[
\norm{(\alpha^{i,j}_{n})_n(\Phi_{g_{n},\mathbb I_{n}}^{i,j}(y_n))-(\beta^{i',j'})_{m}(\Phi_{\hat g,\hat{\mathbb I}}^{i',j'}(y_n))}<2^{-k}
\]
hence 
\begin{eqnarray*}
\norm{\pi_B(\sum_{n\in Z'_{i',j'}}(\alpha^{i,j}_{n})_n(\Phi_{g_{n},\mathbb I_{n}}^{i,j}(y_n)) -\beta^{i',j'}(\Phi_{\hat g,\hat{\mathbb I}}^{i',j'}(y_{i',j'})))}&=&\\\norm{\pi_B(\sum_{n\in Z'_{i',j'}}(\alpha^{i,j}_{n})_n(\Phi_{g_{n},\mathbb I_{n}}^{i,j}(y_n)))-\Lambda(\pi_A(y_{i',j'}))}&\leq& 2^{-k}.
\end{eqnarray*}
Since $\NN=\sqcup Z'_{i',j'}$ and $\pi_A(x)=\pi_A(\sum_{i',j'}y_{i',j'}$, summing over $i\in\{e,o\}$ and $j\in\{0,1\}$ we have the thesis.
\end{proof}
Recall that 
\[
\Gamma_{\Lambda}=\{(a,b)\in\mathcal M(A)\times\mathcal M(B)\colon\norm{a},\norm{b}\leq 1, \,\,\Lambda(\pi_A(a))=\pi_B(b).\}.
\]
is the graph of $\Lambda$.
\begin{theorem}
Let $(x,y)\in\mathcal M(A)\times\mathcal M(B)$ be a pair of contractions. The following are equivalent:
\begin{enumerate}
\item\label{Borel:inthegraph} $(x,y)\in\Gamma_\Lambda$
\item\label{Borel:analytic} For every $k\in\NN$ there are  contractions $x_{i,j}\in\mathcal M(A)$ and $y_{i,j}\in\mathcal M(B)$, for $i\in\{e,o\}$ and $j\in\{0,1\}$, $\sum_{i,j}\pi_A(x_{i,j})=x$ and $\sum_{i,j}\pi_B(y_{i,j})=\pi_B(y)$ and there are sequences $\langle(g_l,\mathbb I_l,\bar\alpha_l)\rangle\subseteq D_k$ and $N^{i,j}_l$ with $N^{i,j}_l\geq\max (I_l)_{4l+4}$ and satisfying 
\begin{enumerate}[label=(\roman*)]
\item\label{a1} $e_{N^{i,j}_l}x_{i,j}e_{N^{i,j}_l}\in B(g_l,\mathbb I_l,i,j)$ 
\item\label{a2} if $l<l'$ and $\max (I_l)_n\leq \max_{i,j}N^{i,j}_l$ then $(I_l)_n=(I_{l'})_n$ and $g_l(n)=g_{l'}(n)$ and
\item\label{a3} 
\[\norm{\sum (\alpha_{l})_l(\Phi^{i,j}_{f_{l}(l)}(f_{(I^i_{2l+j})^+}x_{i,j}f_{(I^i_{2l+j})^+})))-y_{i,j}}<20\cdot 2^{-k}\]
\end{enumerate}
\item\label{Borel:coanalytic} For all contractions $x_{i,j}\in\mathcal M(A)$ and $y_{i,j}\in\mathcal M(B)$, for $i\in\{e,o\}$ and $j\in\{0,1\}$, if $\pi_A(x)=\sum_{i,j}\pi_A(x_{i,j})$ and for every $k\in\NN$ there are sequences $\langle(g_l,\mathbb I_l,\bar \alpha_l)\rangle\subseteq D_k$ and $N^{i,j}_l$ satisfying \ref{a1}, \ref{a2} and \ref{a3}, then $\pi_B(y)=\sum_{i,j}\pi_B(y_{i,j})$.
\end{enumerate}
\end{theorem}
\begin{proof}
We first prove that \ref{Borel:inthegraph} implies \ref{Borel:analytic}. By $\leq_*\times\leq_1$-cofinality of $\mathcal Y_k$ there are $g\in\NN^\NN$, $\mathbb I\in\mathbb P$ and $\bar\alpha$ and $x_{i,j}$, for $i\in\{e,o\}$, $j\in\{0,1\}$ such that $x_{i,j}\in B(g,\mathbb I,i,j)$, $x=\sum x_{i,j}$ and $(g,\mathbb I,\bar\alpha)\in\mathcal Y_k$.

Let $y_{i,j}=\alpha^{i,j}(\Phi_{g,\mathbb I}^{i,j}(x_{i,j}))$. Since $(g,\mathbb I,\bar\alpha)\in\mathcal X$ and $x_{i,j}\in B(g,\mathbb I,i,j)$ we have that $\sum_{i,j}\pi_B(y_{i,j})=\sum_{i,j}\Lambda(\pi_A(x_{i,j}))=\Lambda(\pi_A(x))$.

Let $N^{i,j}_l=\max I_{4l+j}^i$. By density of $D_k$ there are $g_l\in\NN^\NN$, $\mathbb I_l\in\mathbb P$ and $\bar\alpha_l$ such that for all $l$, $g_l(n)=g(n)$ and $(I_l)_n=I_n$ whenever $n\leq \max_{i,j} N_l^{i,j}$, and $(g_l,\mathbb I_l,\bar\alpha_l)\in D_k$. Note that for all $i,j$ we have that conditions \ref{Borelc1} and \ref{Borelc2} of Lemma~\ref{lem:almostthere} are satisfied for $x_{i,j}$, and so are \ref{a1} and \ref{a2}. Applying Lemma~\ref{lem:almostthere}, we have condition \ref{a3}.

Assume now \ref{Borel:analytic}. Then conditions~\ref{a1}--\ref{a3} imply that
$\norm{\Lambda(\pi_A(x_{i,j}))-\pi_B(y_{i,j})}\leq 20\epsilon_k$, therefore
\ref{Borel:inthegraph} follows. For this reason, we also have that
\ref{Borel:inthegraph} implies \ref{Borel:coanalytic}. Similarly pick contractions $x\in
\mathcal M(A)$ and $y\in\mathcal M(B)$, and suppose that \ref{Borel:coanalytic} holds. As in \ref{Borel:inthegraph}$\Rightarrow$\ref{Borel:analytic}, we can find $(g,\mathbb I,\bar\alpha)\in Y_k$ and $x_{i,j}\in B(g,\mathbb I,i,j)$ such that $x=\sum_{i,j} x_{i,j}$. Let $y_{i,j}=\alpha^{i,j}(\Phi_{g,\mathbb I}(x_{i,j}))$. By definition of $\mathcal X$ we have that $\Lambda(\pi_A(x_{i,j}))=\pi_B(y_{i,j})$ for all $i,j$, and by \ref{Borel:coanalytic} we have that $\pi_B(y)=\pi_B(\sum_{i,j}y_{i,j})$, hence $\Lambda(\pi_A(x))=\pi_B(y)$, which is \ref{Borel:inthegraph}.
\end{proof}

\begin{proof}[Proof of Theorem~\ref{thmi:alltrivialnoncomm}]
Condition \ref{Borel:analytic} defines an analytic set. Condition \ref{Borel:coanalytic} defines a coanalytic one, so $\Gamma_\Lambda$ restricted to pairs of contractions is Borel, and therefore $\Lambda$ is topologically trivial.
\end{proof}

\section{Rigidity results}\label{section:rigidity}
In this section we work with a simplified version of $\mathcal X$ (see Notation~\ref{notation:thespace}) to prove a technical rigidity result, Theorem~\ref{thm:almostredprod}, for subalgebras of corona $\Cstar$-algebras that resemble reduced products. This will then lead to the proof of Theorems~\ref{thmi:alltrivial}--\ref{thmi:contimgs}.
We keep all notation from \S\ref{section:general}, especially Notation~\ref{not:autos}.

Let $J_n\subseteq\NN$ be finite intervals such that $\max J_n+6<\min J_{n+1}$, and fix $\mathbb I\in\mathbb P$ such that $I_{2n}^e=J_n$. Let
\[
\mathcal X_1=\{(g,\alpha_0)\colon \exists \alpha_1,\alpha_2,\alpha_3 ((g,\mathbb I,\alpha_0,\alpha_1,\alpha_2,\alpha_3)\in\mathcal X)\},
\]
and consider the colouring $[\mathcal X_{1}]^2= K_0^n\cup K_1^n$ where $\{(g,\alpha),(g,\alpha')\}\in K_0^n$ if and only if there are $m\in\NN$ and a contraction $x\in U_{m,g(m)}\cap U_{m,g'(m)}\cap J_m$ such that 
\[
\norm{\alpha(\Phi_{g,\mathbb I}^{e,0}(x))-\alpha'(\Phi_{g',\mathbb I}^{e,0}(x))}>2^{-n}.
\]
\begin{claim}
If $\mathfrak b>\omega_1$, then for all $n$ there is no $K_0^n$-homogeneous set.
\end{claim}
\begin{proof}
If $Z_1\subseteq\mathcal X_1$, let $Z=\{(g,\mathbb I,\bar\alpha)\colon (g,\alpha^{e,0})\in Z_1\}$. If $Z_1$ is uncountable and $K_0^n$-homogeneous, then $Z$ is uncountable and $L_0^n$-homogeneous, where $L_0^n$ is the colouring defined on $\mathcal X$ after Lemma~\ref{lemma:afteracertainpoint}. This contradicts Lemma~\ref{lemma:OCAfirstalternativecomplicated}.
\end{proof}
For $n\in\NN$, define the partial order $<_{n}$ on $\NN^\NN$ by 
\[
g\leq_n g' \iff \forall m\geq n (g(m)\leq g'(m)).
\]
Since $\leq_*$ is $\sigma$-directed, and $\leq_*$-cofinal sets must be $\leq_n$-cofinal for some $n\in\NN$ (\cite[Lemma 2.4.3]{Farah.AQ}), using the argument reproduced before Lemma~\ref{lem:almostthere}, we can find an increasing sequence $\{n_k\}\subseteq\NN$ and sets $\mathcal Y_k$ with the following properties:
\begin{itemize}
\item $\mathcal Y_0\supseteq \mathcal Y_1\supseteq\cdots$;
\item $\mathcal Y_k$ is $K_1^k$-homogeneous and the set $\{g\colon \exists\alpha_0 ((g,\alpha_0)\in\mathcal Y_k)\}$ is $\leq_{n_k}$-cofinal.
\end{itemize}
For all $n$ with $n_k\leq n <n_{k+1}$, pick a sequence of function $g^{i,n}$ such that there are $\alpha(i,n)$ with $(g^{i,n},\alpha(i,n))\in\mathcal Y_k$ and $g^{i,n}(n)\to\infty$ as $i\to\infty$. For a contraction $x\in A_{J_n}$, define 
\[
m(x)=\min\{m\colon x\in U_{n,m}\}\text{ and }\bar i(x)=\min\{i\colon g^{i,n}(n)\geq m(x)\}.
\]
Define a map 
\[
\gamma_n\colon A_{J_n}\to r_{J_n}Br_{J_n},
\]
by 
\[
\gamma_n(x)=\alpha(\bar i(x),n)\circ\Phi_{g^{\bar i(x),n},\mathbb I}^{e,0}(x)
\]
If $x\in A_{J_n}$ is such that $\norm{x}>1$ let $\gamma_n(x)=\norm{x}\gamma_n(x/\norm{x})$. Define 
\[
\Gamma=\sum\gamma_n\colon \prod A_{J_n}\to\mathcal M(B).
\]
Since $\gamma_n\gamma_m=0$ whenever $n\neq m$ and for each $n$ we have that $\norm{\gamma_n}\leq 2$, $\Gamma$ is well defined.

\begin{lemma}\label{lem:liftingmore}
Let $Z\subseteq \prod A_{J_n}$ be of size $\aleph_1$.  Then there is a nonmeager dense ideal $\SI$ such that 
\[
\pi_B(\Gamma(x_S))=\Lambda(\pi_A(x_S))
\]
for every $x\in Z$ and $S\in\SI$.
\end{lemma}
\begin{proof}
For every $x\in Z$, let 
\[
\SI_x=\{S\subseteq\NN\colon \pi_B(\Gamma(x_S))=\Lambda(\pi_A(x_S))\}.
\]
Since $\mathfrak b>\omega_1$ and by $\leq_*$-cofinality of $\mathcal Y_k$, for all $k$ we can find $g_k$ and $\beta_k$ such that $(g_k,\beta_k)\in \mathcal Y_k$ and $\pi_A(x)\in\pi_A(B(g_k,\mathbb I,e,0))$ for all $x\in Z$. By modifying $x$ on finitely many coordinates we can assume that $x\in B(g_k,\mathbb I,e,0)$. Fix nonmeager dense ideals $\SI_k=\SI_{g_k,\mathbb I,\beta_k,e,0}$ as in the definition of $\mathcal X$ and let $\SI=\bigcap\SI_k$, and note that $\SI_x\supseteq\SI$. Since $|Z|=\aleph_1$ and $\mathfrak b>\omega_1$, by \cite[Corollary 7.9]{Blass.HB}, we have that $\bigcap_Z\SI_x$ is nonmeager and dense. This concludes the proof.
\end{proof}
If $\Lambda$ is an isomorphism, then each $\SI_k=\mathcal P(\NN)$, and therefore $\Gamma$ is a lift of $\Lambda$ on $\prod A_{J_n}$ (see Proposition~\ref{prop:liftonall}).
 
Recall the definition of $\epsilon$-$^*$-homomorphism given in \S\ref{section:prel}. From Lemma~\ref{lem:liftingmore} and Proposition~\ref{prop:liftsandalmost} are have the following.
\begin{theorem}\label{thm:almostredprod}
Assume $\OCA+\MA_{\aleph_1}$. Let $A$ and $B$ be separable $\Cstar$-algebras and let $\Lambda\colon\mathcal Q(A)\to\mathcal Q(B)$ be a $^*$-homomorphism. Let $\{e_n^B\}$ be an approximate identity of positive contractions for $B$ with $e_{n+1}^Be_n^B=e_n^B$ for all $n$ and let $\{e_n\}$ be an approximate identity for $A$ satisfying Notation~\ref{not:autos} for $\Lambda$.
Let $J_n$ be finite intervals with $\max J_n+6<\min J_{n+1}$. Then there is a sequence $\epsilon_n$ with $\epsilon_n\to 0$ as $n\to \infty$ and a sequence of mutually orthogonal maps
\[
\gamma_n\colon A_{J_n}\to (e_{k_n}^B-e_{j_n}^B)B(e_{k_n}^B-e_{j_n}^B)
\]
for some $j_n<k_n\in\NN$ where $j_n\to\infty$ as $n\to\infty$, such that 
\[
\gamma_n\text{ is an }\epsilon_n\text{-}^*\text{-homomorphism}.
\]
 If moreover $\Lambda$ is injective, each $\gamma_n$ is $\epsilon_n$-injective. If $\Lambda$ is an isomorphism, $\Gamma=\sum\gamma_n$ is a lift of $\Lambda$ on $\prod A_{J_n}$.
\end{theorem}

\subsection{The proof of Theorems~\ref{thmi:alltrivial} and \ref{thmi:contimgs}}\label{subsec:abelianproof}
We fix second countable locally compact spaces $X$ and $Y$ and an injective unital $^*$-homomorphism $\Lambda\colon C(X^*)\to C(Y^*)$, $X^*=\beta X\setminus X$ being the \v{C}ech-Stone remainder of $X$. We denote by $\pi_X\colon C_b(X)\to C(X^*)$ the quotient map (similarly we define $\pi_Y$) and we let $\tilde\Lambda\colon Y^*\to X^*$ be the continuous surjection dual to $\Lambda$. Fix $\{e_n^Y\}$ an approximate identity of positive contractions for $C_0(Y)$ satisfying $e_{n+1}^Ye_n^Y=e_n^Y$ for all $n$. Since $\supp (e_n^Y)\subseteq (e_{n+1}^Y)^{-1}(\{1\})$, $\overline{\supp(e_n^Y)}$ is compact for all $n$.

 Let $\{e_n\}\subseteq C_0(X)$ be an approximate identity satisfying Notation~\ref{not:autos} for $\Lambda$. Let 
 \[
 U^e_n=\supp(e_{16n+8}-e_{16n}),\,\text{ and }U^o_n=\supp(e_{16n+17}-e_{16n+7}).
 \]
 Each $U_n^e$ and $U_n^o$ is open, 
 \[
 \overline{U_n^o}\subseteq U_n^o\cup U_n^e\cup U_{n+1}^e\text{ and }\overline{U_n^e}\subseteq U_n^e\cup U_n^o\cup U_{n-1}^o.
 \]
  For $i\in\{e,o\}$, $\prod C_0(U_n^i)$ and $\prod C_0(U_n^i)/\bigoplus C_0(U_n^i)$  are hereditary $\Cstar$-subalgebras of $C_b(X)$ and $C(X^*)$ respectively.\footnote{This is only true in the commutative setting. For example, the canonical copy of $\ell_\infty/c_0$ is not an hereditary $\Cstar$-subalgebra of $\mathcal Q(H)$, being unital.}
 \begin{lemma}
$C_b(X)=\prod C_0(U_n^e)+\prod C_0(U_n^o)$.
\end{lemma}
\begin{proof}
Let $K_n=\overline\supp(e_n)$. Let $h^e_n$ be a positive contraction with 
\[
h^e_n(K_{16n})=h^e_n(X\setminus\intr(K_{16n+8}))=0\text{ and }h^e_n(K_{16n+7}\setminus \intr(K_{16n+1}))=1,
\] 
and $h_n^o$ be a positive contraction with 
\[
h_n^o(K_{16n+1})=f(X\setminus \intr(K_{16n+7}))=0, \,\,\, h_n^o(K_{16(n+1)}\setminus \intr (K_{16n+8}))=1
\]
and
\begin{eqnarray*}
h_n^o(x)=1-h^e_n(x)\,\,\,&\text{ for } &x\in K_{16n+8}\setminus\intr (K_{16n+7}), \\
h_n^o(x)=1-h^e_{n+1}(x)\,\,\, &\text{ for } &x\in K_{16n+17}\setminus\intr (K_{16(n+1)}). 
\end{eqnarray*}
Then for all $h\in C_b(X)$ we have that 
\[
h=\sum h^e_nh+\sum h_n^oh, \,\,\, \sum h^e_nh\in\prod C_0(U_n^e)\text{ and }\sum h_n^oh\in\prod C_0(U_n^o).\qedhere
\]
\end{proof}
By applying Theorem~\ref{thm:almostredprod} to $J^0_n=[16n+1,16n+8]$ and  then to $J^1_n=[16n+8,16n+17]$, we obtain maps 
\[
\Phi^i=\sum\phi_n^i\colon \prod C_0(U_n^i)\to C_b(Y)
\] where 
\[
\phi_n^i\colon C_0(U_n^i)\to C_0(Y).
\]
 If $\Lambda$ is an isomorphism, then
 \[
 \pi_Y(\Phi^i(f))=\Lambda(\pi_X(f))
 \]
for all $f\in \prod C_0(U_n^i)$ and $i\in\{e,o\}$. 

Recall that if $n\neq m$ then $\phi_n^i\phi_m^i=0$. Let $r_n^i$ be the positive contractions such that the range of $\phi_n^i$ is included in $r_n^iC_0(Y)$ as in Notation~\ref{not:autos}. Since $r_n^i\leq (e_{k_n}^Y-e_{j_n}^Y)$ for some $j_n<k_n$, the set $W_n^i=\supp r_n^i$ has compact closure for all $n\in\NN$ and $i\in\{e,o\}$. Since $\phi_n^i$ is an $\epsilon_n$-$^*$-monomorphism for some sequence $\epsilon_n\to 0$, by Theorem~\ref{thm:semrl} there are $^*$-homomorphisms 
\[
\psi_n^i\colon C_o(U_n^i)\to r_n^iC_0(Y)\cong C_0(W_n^i)
\] such that $\norm{\psi_n^i-\phi_n^i}\to 0$ as $n\to\infty$. For a large enough $n$ we have that each $\psi_n^i$ is injective, being close to a $\epsilon_n$-monomorphism whenever $\epsilon_n<\frac{1}{4}$. In case $\Lambda$ is an isomorphism, as $\norm{\psi_n^i-\phi_n^i}\to 0$ if $n\to\infty$, the $^*$-homomorphism
\[
\Psi^i=\sum\psi_n^i
\]
is such that 
 \[
 \pi_Y(\Psi^i(f))=\Lambda(\pi_X(f))
 \]
for all $f\in \prod C_0(U_n^i)$. 

Let $V_n^i=\bigcup_{f\in U_n^i}\supp(\psi_n^i(f))$ and $B_n^i=C_0(V_n^i)$. Then 
\[\Psi^i\circ\pi_X\colon \prod C_0(U_n^i)/\bigoplus C_0(U_n^i)\to\prod B_n^i/\bigoplus B_n^i
\]
 is an embedding, and moreover each $\psi_n^i$ sends an approximate identity for $C_0(U_n^i)$ to an approximate identity for $B_n^i$, and so there exist continuous proper surjections $\gamma_n^i\colon V_n^i\to U_n^i$ dual to $\psi_n^i$.  $\gamma_n^e$ and $\gamma_n^o$ are defined for all but finitely $n$.
 
 
\begin{lemma}\label{lemman0}
There is $n_0\in\NN$ such that if $n,m\geq n_0$ then $\gamma_n^e$ and $\gamma_m^o$ agree on common domains.
\end{lemma}
\begin{proof}
First, note that the $V_n^e$ and $V_m^o$ intersect if and only if $m=n$ or $m=n+1$. This is because $C_0(U_n^e)C_0(U_m^o)\neq 0$ if and only if $m=n$ or $m=n+1$. Let $J_n=V_n^e\cap V_n^o$, and suppose that for infinitely many $n$ there is $y_n\in J_n$ such that 
\[
x_n=\gamma_n^e(y_n)\neq \gamma_n^o(y_n)=z_n.
\]

Since for every compact set $K\subseteq X$ there is $n_0$ such that $K\cap U_n^i=\emptyset$ whenever $n\geq n_0$, the sequences $\{x_n\}$ and $\{z_n\}$ have no accumulation point. Find open sets $Z_{n,1},Z_{n,2}\subseteq U_n^e\cap U_n^o$ such that $x_n\in Z_{n,1}$, $z_n\in Z_{n,2}$ and $\overline{Z_{n,1}}\cap\overline {Z_{n,2}}=\emptyset$. Since $Z_{n,1},Z_{n,2}\subseteq U_{n}^e$, we have that 
\[
\overline{\bigcup_n Z_{n,1}}\cap\overline{\bigcup_nZ_{n,2}}=\bigcup_n\overline{Z_{n,1}}\cap \bigcup_n\overline{Z_{n,2}}=\emptyset.
\]
 Let $f_n,g_n$ be positive contractions such that 
 \[
 \supp(f_n),\supp(g_n)\subseteq U_n^e\cap U_n^o,\,\,f_ng_n=0\text{ and }f_n\restriction Z_{n,1}=g_n\restriction Z_{n,2}=1.
 \] Let  $f=\sum f_n$ and $g=\sum g_n$. Note that $fg=0$. Moreover since $f_ng_m=0$ whenever $n\neq m$, we have that $(\sum_{n\in S}f_n)(\sum_{n\in S}g_n)=0$ for all $S\subseteq\NN$.
 
Let \[
f'_n=\psi_n^e(f_n)\text{ and }g'_n=\psi_n^o(g_n).
\]
 With 
\[
f'=\sum f'_n\text{ and }g'=\sum g_n'
\]
 we have that $f'g'(y_n)=f'_ng'_n=1$, and therefore for all $S\subseteq\NN$ we have that 
\[
\norm{(\sum_{n\in S}f_n')(\sum_{n\in S}g'_n)}=\norm{\sum_{n\in S}f'_ng'_n}=1.
\]
By Lemma~\ref{lem:liftingmore} there is $S\subseteq\NN$ such that 
\[
\pi_Y(\sum_{n\in S}f_n')=\Lambda(\pi_X(\sum_{n\in S}f_n))\text{ and }
\pi_Y(\sum_{n\in S}g_n')=\Lambda(\pi_X(\sum_{n\in S}g_n)).
\]
In particular we have 
\[
0=\norm{\Lambda(\pi_X((\sum_{n\in S}f_n)(\sum_{n\in S}g_n)))}=\norm{\pi_Y((\sum_{n\in S}f_n')(\sum_{n\in S}g_n'))}=1,
\]
a contradiction
\end{proof}

We are ready to prove our main results. Fix $n_0$ as provided by Lemma~\ref{lemman0}.
Then 
\[
\Gamma=\bigcup_{n\geq n_0}(\gamma_n^e)\cup\bigcup_{n\geq n_0}(\gamma_n^o)
\]
is a continuous surjection from a subset of $Y$ to a cocompact subset of $X$. Note that since for all $n$ we have that 
\[
\overline {U_n^e}\subseteq U_{n-1}^o\cup U_n^e\cup U_n^o\text{ and }\overline {U_n^o}\subseteq U_{n}^e\cup U_n^o\cup U_{n+1}^e,
\] we have that for $n\geq n_0+1$, 
\[
\overline {V_n^e}\subseteq V_{n-1}^o\cup V_n^e\cup V_n^o\text{ and }\overline {V_n^o}\subseteq V_{n}^e\cup V_n^o\cup V_{n+1}^e.
\]
In particular 
\[
Z=\overline{(\bigcup V_n^e)\cup(\bigcup V_n^e)}=\bigcup \overline{V_n^e}\cup\bigcup\overline {V_n^o}
\]
 is equal, modulo compact, to $\bigcup V_n^e\cup\bigcup V_n^o$. Hence $Z$ is clopen modulo compact.
\begin{proof}[Proof of Theorem~\ref{thmi:contimgs}]
Note that since $Z$ is clopen modulo compact, $Z^*$ is clopen in $Y^*$. (In case $Y^*$ is connected, we have that $Z^*=Y^*$). The map $\Gamma$ is then a continuous surjection from a clopen modulo compact subspace of $Y$ to a cocompact subspace of $X$, as required by the thesis.
\end{proof}
\begin{remark}\label{rem:notall}
In \cite[Example 3.2.1]{Farah.AQ} Farah provided an example, in $\ZFC$, of an injective unital $^*$-homomorphism $\ell_\infty/c_0\to\ell_\infty/c_0$ that cannot be lifted to an asymptotically additive unital map. Such an obstruction can be modified to provide higher dimensional examples. For this reason, we cannot ask in the thesis of Theorem~\ref{thmi:contimgs} for the set $Z$ to be equal (modulo compact) to $Y$, unless $Y^*$ is connected. 

The study of similar phenomena in the noncommutative setting can be applied to the study of endomorphisms of $\mathcal Q(H)$. Vaccaro proved (\cite{Vaccaro.Endo}) that under $\OCA$ all endomorphisms of $\mathcal Q(H)$ are of the form $\Ad(u)\circ \iota_n$ where $\iota_n\colon\mathcal Q(H)\to M_n(\mathcal Q(H))$ is the endomorphism of $\mathcal Q(H)$ mapping $a$ to the diagonal matrix $(a,a,\ldots,a)$. (Note that Forcing Axioms are needed, as \cite[Corollary 3.7]{FHV.Calkin} can be applied to construct under $\CH$ an endomorphism of $\mathcal Q(H)$ sending an operator of Fredholm index one to an operator of zero Fredholm index). The reason Example 3.2.1 in \cite{Farah.AQ} cannot be brought to such setting is that in $\mathcal Q(H)$ one can perform the `partial isometry' trick, see \cite[\S17.5]{Farah.Book}.\end{remark}
For the proof of Theorem~\ref{thmi:alltrivial} we need a couple of lemmas exploiting additional properties of $\Gamma$ in case $\Lambda$ is an isomorphism.

\begin{lemma}\label{lemma:hered}
Suppose that $\Lambda$ is an isomorphism. Then for almost all $n\in\NN$ we have that the image of $\psi_n^i$ is equal to $B_n^i$.
\end{lemma}
\begin{proof}
Since $V_n^i=\bigcup_{f\in C_0(U_n^i)}\supp(f)$, it is enough to show that the image of $\psi_n^i$ is hereditary\footnote{For $\Cstar$-algebras $A\subseteq B$, $A$ is hereditary in $B$ if for all positive $a\in A$, if $b\leq a$ and $b$ is positive, then $b\in A$.}. We argue by contradiction. Let $C_n^i=\psi_n^i(C_0(U_n^i))$. Fix $i$ and assume that $C_n^i$ is not hereditary for infinitely many $n$.  Then there are positive contractions $\tilde f_n$ and $\tilde g_n\in C^i_n$ with $\tilde f_n\leq \tilde g_n$ and $\inf_{\tilde g\in C^i_n}\norm{\tilde f_n-\tilde g}=1$.

Since $C^i_nC^i_m=0$ when $n\neq m$, $\tilde f=\sum \tilde f_n$ and $\tilde g=\sum \tilde g_n$ are well defined. On the other hand $\pi_Y(\tilde f)\leq\pi_Y(\tilde g)$, but $\pi_Y(\tilde f)\notin\prod C^i_n/\bigoplus C^i_n$, a contradiction to $\prod C_n^i/\bigoplus C_n^i$ being hereditary (as it is the isomorphic image, through $\Lambda$, of the hereditary $\Cstar$-algebra $\prod C_0(U_n^i)/\bigoplus C_0(U_n^i)$).
\end{proof}

\begin{lemma}
Suppose that $\Lambda$ is an isomorphism. Then $Y\setminus Z$ is compact.
\end{lemma}
\begin{proof}
Let $\{y_n\}$ be a sequence of points such that $\{y_n\}\to\infty$ as $n\to\infty$ with $y_n\notin Z$. Since $Z$ is closed, for almost all $n$ there is an open set $Z_n$ with $y_n\in Z_n$ and $Z_n\cap Z=\emptyset$. Since $\{y_n\}$ has no accumulation points in $Y$, we can assume that $\overline {Z_n}\cap \overline{Z_m}=\emptyset$. Pick a contraction $g_n\in C_0(Z_n)$, and let $g=\sum g_n\in C_b(Y)$. Since the support of $g$ is disjoint from $Z$ we have that for all $i\in\{e,o\}$, $\pi_Y(g)\Lambda(f)=0$ whenever $f\in\prod C_0(U_n^i)/\bigoplus C_0(U_n^i)$, and therefore $\pi_Y(g)\Lambda(f)=0$ whenever 
\[
f\in \prod C_0(U_n^e)/\bigoplus C_0(U_n^e)+\prod C_0(U_n^o)/\bigoplus C_0(U_n^o)=C(X^*).
\]
 This is a contradiction.
\end{proof}
\begin{proof}[Proof of Theorem~\ref{thmi:alltrivial}]
Both domain and  range of $\Gamma$ are cocompact and $\Gamma^*=\Lambda^{-1}$. This is the thesis.
\end{proof}

\subsection{The proof of Theorem~\ref{thmi:embedding}}\label{subsection:nonemb}
Here we prove Theorem~\ref{thmi:embedding}.
\begin{definition}
Let $P$ be a property of $\Cstar$-algebras. 
\begin{itemize}
\item A nonunital $\Cstar$-algebra $A$ has property $P$ at infinity if there is $\{e_n\}$, an approximate identity of positive contractions for $A$ with $e_{n+1}e_n=e_n$ for all $n$ such that there is $n_0$ with the property that $(e_m-e_n)A(e_m-e_n)$ has property $P$ whenever $n_0\leq n<m$. 
\item The property $P$ is preserved under almost (unital) inclusions if there is $\epsilon>0$ such that whenever $A$ has property $P$ and $\phi\colon A\to B$ is a (unital) $\epsilon$-monomorphism, then $B$ has property $P$.
\end{itemize}
\end{definition}
A projection $p\in A$ is infinite if there is $v\in A$ such that $vv^*=p$ and $v^*v<p$. A $\Cstar$-algebra $A$ is finite if it has no infinite projection, and stably finite if $A\otimes M_n(\mathbb C)$ is finite for all $n$.
The proof of the following is given by weak stability of the relation `$p$ is an infinite projection' (see \cite[\S3.2]{bourbaki}).
\begin{proposition}
The properties `being infinite-dimensional' and `being infinite' are preserved under almost inclusions.\qed
\end{proposition}
If $A$ is unital  then $A\otimes\mathcal K$ and $c_0(A)$ are infinite-dimensional at infinity if and only if $A$ is infinite-dimensional. $A\otimes\mathcal K$ is stably-finite at infinity if and only if $A$ is stably finite.

\begin{theorem}\label{thm:propertyP}
Assume $\OCA+\MA_{\aleph_1}$ and let $A$ and $B$ be separable nonunital $\Cstar$-algebras. Let $P$ be a property whose negation is preserved under almost inclusions. If $\mathcal Q(A)$ embeds into $\mathcal Q(B)$ and $B$ has property $P$ at infinity, so does $A$.
\end{theorem}
\begin{proof}
By contradiction, fix an embedding $\Lambda\colon \mathcal Q(A)\to\mathcal Q(B)$. Let $\{e_n^B\}$, an approximate identity for $B$ with the property that there is $n_0$ such that 
\[
(e_m^B-e_n^B)B(e_m^B-e_n^B)
\]
 has property $P$ whenever $n_0\leq n<m$.

Let $\{e_n\}$ be an increasing approximate identity for $A$ and suppose that there are finite nonempty intervals $I_n\subseteq\NN$ with $\max I_n+6<\min I_{n+1}$ such that
\[
A_{I_n}=(e_{\max I_n}-e_{\min I_n})A(e_{\max I_n}-e_{\min I_n})
\]
 does not have property $P$. Thanks to the assumption of Forcing Axioms and applying Theorem~\ref{thm:almostredprod}, we can go to a subsequence and assume that there are disjoint finite intervals $J_n\subseteq\NN$ and maps $\gamma_n\colon A_{I_n}\to B_{J_n}$ such that each $\gamma_n$ is a $2^{-n}$-monomorphism. Since $\neg P$ is preserved under almost inclusions, there is $n_0$ such that $B_{J_n}$ has property $\neg P$ for $n\geq n_0$. Since $\min J_n\to\infty$ as $n\to\infty$ we have a contradiction.
\end{proof}

\begin{proof}[Proof of Theorem~\ref{thmi:embedding}]
Assume $\CH$. Then \cite[Corollary 3.7]{FHV.Calkin} asserts that if $A$ and $B$ are unital and separable, $\mathcal Q(A\otimes\mathcal K(H))$ embeds unitally into $\mathcal Q(B\otimes\mathcal K(H))$. More than that: the algebra $\mathcal Q(A\otimes\mathcal K)(H)$ is $2^{\aleph_0}$-universal, and all coronas of separable $\Cstar$-algebras embed into it. This proves (i).

We now show (ii). Note that $\mathcal K(H)$ is finite-dimensional at infinity, that being infinite-dimensional is preserved under almost inclusions and that if $A$ is unital, $A\otimes\mathcal K(H)$ and $c_0(A)$ are finite-dimensional at infinity only if $A$ is finite-dimensional. By Theorem~\ref{thm:propertyP}, when $\OCA$ and $\MA_{\aleph_1}$ are assumed, $\mathcal Q(A\otimes\mathcal K(H))$ and $\ell_\infty(A)/c_0(A)$ embed into $\mathcal Q(H)$ only when $A\otimes\mathcal K(H)$ and $c_0(A)$ are finite-dimensional at infinity.

Lastly, note that ``being infinite" is preserved under almost inclusions, and that $A\otimes\mathcal K(H)$ is finite at infinity only if $A$ is stably finite. $\OCA$ and $\MA_{\aleph_1}$, via Theorem~\ref{thm:propertyP}, imply that, if $A$ and $B$ are unital and separable and $B$ is stably finite, $\mathcal Q(A\otimes\mathcal K(H))$ embeds into $\mathcal Q(B\otimes \mathcal K(H))$ only if $A$ is stably finite. 
\end{proof}
The same argument of the proof of Theorem~\ref{thmi:embedding} shows that, under $\OCA_{\infty}+\MA_{\aleph_1}$, if $A_n$ are unital $\Cstar$-algebras, then $\prod A_n/\bigoplus A_n$ embeds into $\mathcal Q(H)$, or into $\prod M_{n_k}/\bigoplus M_{n_k}$ whenever $n_k\in\NN$, only if $A_n$ is finite-dimensional for almost all $n\in\NN$.
\section{Algebraically trivial isomorphisms}\label{section:algtriv}

In this section we discuss algebraically trivial isomorphisms and algebraically isomorphic coronas (Definition~\ref{def:algtriv}). We start by making two conjectures, motivated by the situation in the abelian setting, the case of $\mathcal Q(H)$, and isomorphisms of reduced products of finite-dimensional $\Cstar$-algebras (see~\S\ref{subsec:cstartrivial}). 
\begin{conjecture}\label{conj:algtriv}
Assume $\OCA+\MA_{\aleph_1}$. Then all isomorphisms between all coronas of separable $\Cstar$-algebras are algebraically trivial.
\end{conjecture}
\begin{conjecture}\label{conj:algtriviso}
Assume $\OCA+\MA_{\aleph_1}$. Let $A$ and $B$ be separable $\Cstar$-algebras. Then $\mathcal Q(A)$ are $\mathcal Q(B)$ are isomorphic if and only if they are algebraically isomorphic.
\end{conjecture}
Confirming Conjecture~\ref{conj:algtriv} implies confirming Conjecture~\ref{conj:algtriviso}, but the converse is open. For example, Conjecture~\ref{conj:algtriviso} holds when $A$ and $B$ are of the form $A=\bigoplus A_n$ and $B=\bigoplus B_n$, for separable unital UHF algebras $A_n$ and $B_n$ (this is the main result of \cite{McKenney.UHF}), but Conjecture~\ref{conj:algtriv} is open in this case. 

Conjecture~\ref{conj:algtriviso} was confirmed if $A$ and $B$ are of the form $\bigoplus A_n$ and $\bigoplus B_n$, where $A_n$ and $B_n$ are either unital separable simple AF algebras or unital Kirchberg algebras\footnote{A Kirchberg algebra is a separable, purely infinite, simple and nuclear $\Cstar$-algebra.}  satisfying the Universal Coefficient Theorem (\cite[Corollary 6.1]{MKAV.FA}). 

We study the conjectures for reduced products in \S\ref{subsec:redprod}, and for coronas of algebras of the form $A\otimes\mathcal K(H)$ in \S\ref{subsec:stab}.
\subsection{Reduced products}\label{subsec:redprod}
Fix unital $\Cstar$-algebras $A_n$, for $n\in\NN$, and let $A=\bigoplus A_n$. In this case, 
\[
\mathcal M(A)=\prod A_n, \,\,\, \mathcal Q(A)=\prod A_n/\bigoplus A_n.
\]
The latter is called the \emph{reduced product} of the algebras $A_n$.

Let us start to analyse what algebraically trivial isomorphisms are in this setting. If $a=(a_n)\in \mathcal M(A)$ is positive element such that $1-a\in A$,, then we have that there is $n_0$ such that for all $n\geq n_0$ we have that $\norm{1_{A_n}-a_n}<\frac{1}{4}$, and therefore $a_n$ is invertible. This implies that $\overline{a_nA_na_n}=A_n$, hence we have that 
\[
\overline{aAa}=C_a\oplus\bigoplus_{k\geq n_0} A_k.
\]
where $C_a=\overline{c(\bigoplus_{k<n_0}A_k)c}$ for some positive $c\in\bigoplus_{k<n_0}A_k$.

A projection $p\in A$ is central if it commutes with $A$. A projection $p\in A$ is either central or there is $\pi\colon A\to\mathcal B(H)$, an irreducible representation of $A$, such that $\pi(p)\notin\{0,1\}$. This is equivalent to the existence of a contraction $a\in A$ with $\norm{pa-ap}\geq\frac{1}{2}$. (This is a consequence of Kadison's Transitivity Theorem. For a proof, see e.g., \cite[Proposition 6.2]{MKAV.FA}.) Define, for $S\subseteq\NN$, the projection $p^A_S\in\prod A_n$ by
\[
(p^A_S)_n=\begin{cases}1_{A_n}&n\in S\\
0&\text{ else}
\end{cases}.
\]
Note that $p_S^A\in A$ if and only if $S$ is finite. 
Suppose that each $A_n$ does not have central projections. Then all central projections in $A$ and $\mathcal M(A)$ are of the form $p_S^A$, for some $S\subseteq\NN$. More than that, if $q\in\mathcal Q(A)$ is a central projection, then $q=\pi_A(p_S^A)$ for some $S\subseteq\NN$. In this case, if $a$ and $a'$ are positive elements in $\mathcal M(A)$ such that $1-a$ and $1-a'$ are in $A$, an isomorphism $\phi\colon \overline{aAa}\to\overline{a'Aa'}$ that maps strictly convergent sequences to strictly convergent sequences induces an almost permutation of $\NN$. The following summarises the observations we just made, and characterises algebraically trivial isomorphisms of reduced products. Recall the definition of algebraically trivial isomorphism given in Definition~\ref{def:algtriv}.
\begin{proposition}\label{prop:redprodalgtriv}
Let $A_n$ and $B_n$ be unital $\Cstar$-algebras with no central projections. Let $A=\bigoplus A_n$ and $B=\bigoplus B_n$, and denote by $\pi_A$ and $\pi_B$ the quotient maps from $\mathcal M(A)$ to $\mathcal Q(A)$ and from $\mathcal M(B)$ to $\mathcal Q(B)$ respectively.

Let $\Lambda\colon\mathcal Q(A)\to\mathcal Q(B)$ be an isomorphism. Then $\Lambda$ is algebraically trivial if and only if there are 
\begin{itemize}
\item $f$, an almost permutation of $\NN$,
\item $n_0\in\NN$, and
\item isomorphisms $\phi_n\colon A_n\to B_{f(n)}$, for $n\geq n_0$,
\end{itemize} 
such that the map between $\mathcal Q(A)$ and $\mathcal Q(B)$ induced by $\sum\phi_n$ equals $\Lambda$, that is, for all $a=(a_n)\in\mathcal M(A)$ we have that 
\[
\Lambda(\pi_A(a))=\pi_B(\sum\phi_n(a_n)).
\]
\end{proposition}
\begin{proof}
We first prove the converse direction. Let $f$, $n_0$ and $\phi_n\colon A_n\to B_{f(n)}$ be inducing $\Lambda$ in the above fashion. Notice that $\sum\phi_n$ induces a map between $\mathcal M(A)=\prod A_n$ and $\mathcal M(B)=\prod B_n$ which sends strictly continuous sequences to strictly continuous sequences, as the strict topology on $\prod A_n$ is given by the product topology of the norm topologies on the $A_n$, and the map $\sum\phi_n$ is the product map of the norm-continuous maps $\phi_n$.

Let $a=\sum_{n\geq n_0}$ and $b=(\sum\phi_n)(a)$. Then $1-a\in A$ and $1-b\in B$. Moreover $\sum\phi_n$ induces an isomorphism between 
\[
aAa=\overline{aAa}=\bigoplus_{n\geq n_0}A_n\text{ and }bBb=\overline{bBb}=\bigoplus_{n\geq n_0}B_{f(n)}.
\]
Since $\Lambda$ is the isomorphism induced by $\sum\phi_n$, $\Lambda$ is algebraically trivial.

Conversely, pick $a\in\mathcal M(A)$, $b\in\mathcal M(B)$ with $1-a\in A$ and $1-b\in B$ and an isomorphism $\phi\colon \overline{aAa}\to\overline{bBb}$ such that $\tilde\phi=\Lambda$ and $\phi$ maps strictly convergent sequences to strictly convergent sequences. Let $n_a$ and $n_b$ be naturals and $c_a\in \bigoplus_{n<n_a} A_n$ and $c_b\in \bigoplus_{n<n_b}B_n$ be positive elements such that
\[
\overline{aAa}=\overline{c_a(\bigoplus_{n<n_a} A_n)c_a}\oplus\bigoplus_{k\geq n_a} A_k
\]
and
\[
\overline{bBb}=\overline{c_b(\bigoplus_{n<n_b} B_n)c_b}\oplus\bigoplus_{k\geq n_b} B_k.
\]
Let $a'=(c_a+\sum_{n\geq n_a}1_{A_n})$ and $b'=c_b+\sum_{n\geq n_b}1_{B_n}$. As $a-a'\in A$ and $b-b'$, we can see $\phi$ as an isomorphism $\phi\colon \overline{(a')A(a')}\to\overline{(b')B(b')}$. Notice that all central projections in $\overline{(a')A(a')}=\bigoplus_{n\geq n_a}$ are of the form $\sum_{n\in S}1_{A_n}$ where $S\subseteq\NN\setminus n_a$, and minimal central projections are of the form $1_{A_n}$ for $n\geq n_a$. Similarly, central projections in $\overline{(b')B(b')}=\bigoplus_{n\geq n_b}B_n$ are of the form $\sum_{n\in T}1_{B_n}$ for some $T\subseteq\NN\setminus n_b$, and minimal ones have the form $1_{B_n}$. Since an isomorphism sends minimal central projections to minimal central projection, for all $n\geq n_a$ there is $f(n)\geq n_b$ such that $\phi(1_{A_n})=1_{B_{f(n)}}$. The maps $\phi_n\colon A_n\to B_{f(n)}$, for $n\geq n_a$, given by $\phi_n(a)=\phi(1_{A_n}a1_{A_n})$ are the isomorphisms giving the thesis.
\end{proof}

With in mind the notion of $\epsilon$-$^*$-isomorphism from \S\ref{subsec:approxmaps}, combining the techniques to remove the metric approximation property from the hypotheses of \cite[Theorem~E]{MKAV.FA}, we can reproduce verbatim its proof to show the following.
\begin{theorem}\label{thm:redprod}
Assume $\OCA+\MA_{\aleph_1}$. Let $A_n$ and $B_n$ be unital separable $\Cstar$-algebras with no central projections. Let $A=\bigoplus A_n$ and $B=\bigoplus B_n$, and denote by $\pi_A$ and $\pi_B$ the quotient maps from $\mathcal M(A)$ to $\mathcal Q(A)$ and from $\mathcal M(B)$ to $\mathcal Q(B)$ respectively. Suppose that $\Lambda\colon \mathcal Q(A)\to\mathcal Q(B)$ is an isomorphism. Then there are finite sets $F_1,F_2\subseteq\NN$, a bijection $f\colon\NN\setminus F_1\to\NN\setminus F_2$, maps $\phi_n\colon A_n\to B_{f(n)}$ and a sequence $\epsilon_n\to0$ such that $\phi_n$ is an $\epsilon_n$-$^*$-isomorphism and
\[
\widetilde{\sum\phi_n}=\Lambda,
\]
that is, for all $a=(a_n)\in\prod A_n$ we have that 
\[
\Lambda(\pi_A(a))=\pi_B(\sum\phi_n(a_n)).
\]
\end{theorem}
The two  results above asserts that when $A_n$ and $B_n$ are unital separable $\Cstar$-algebras with no central projections, then modulo almost permutations, that is, assuming that the function $f$ in Theorem~\ref{thm:redprod} is the identity (this can be done if we reindex the $B_n$'s),
\begin{itemize}
\item if we have an algebraically trivial isomorphism $\Lambda$ between $\prod A_n/\bigoplus A_n$ and $\prod B_n/\bigoplus B_n$ then we can find $n_0\in\NN$ and, for $n\in\NN$, isomorphisms $\phi_n\colon A_n\to B_n$ such that $\sum\phi_n$ lifts $\Lambda$
\item under Forcing Axioms, all isomorphisms between $\prod A_n/\bigoplus A_n$ and $\prod B_n/\bigoplus B_n$ are induced by sequences of approximate $^*$-isomorphisms $A_n\to B_n$.
\end{itemize}
Therefore, confirming Conjecture~\ref{conj:algtriv} for reduced products amounts to the study of when approximate isomorphisms are uniformly close to isomorphisms, that is, when we could move the approximate maps of Theorem~\ref{thm:redprod} into on-the-nors isomorphisms to induce the same map at the level of the reduced products. This is formalised by the notion of Ulam stability. The following was given as \cite[Definition 2.6]{MKAV.UC} with $\mathcal C=\mathcal D$.
\begin{definition}\label{defin:ulam}
Let $\mathcal C$ be a  class of $\Cstar$-algebras. We say that $\mathcal C$ is \emph{Ulam stable} if for every $\epsilon>0$ there is $\delta>0$ such that for all $A$ and $B\in\mathcal C$ and a $\delta$-$^*$-isomorphism $\phi\colon A\to B$ there is a $^*$-isomorphism $\psi\colon A\to B$ with $\norm{\phi-\psi}<\epsilon$.
\end{definition}
We have the following:
\begin{theorem}\label{thm:ulam}
Let $\mathcal C$ be a class of unital separable $\Cstar$-algebras with no central projections. The following statements are equivalent:
\begin{enumerate}
\item the class $\mathcal C$ is Ulam stable;
\item all topologically trivial isomorphisms between reduced products of algebras in $\mathcal C$ are algebraically trivial;
\end{enumerate}
\end{theorem}
\begin{proof}
$1\Rightarrow 2$. Suppose that $\mathcal C$ is Ulam stable. Let $A_n$ and $B_n$ be elements of $\mathcal C$ and let $\Lambda\colon\prod A_n/\bigoplus A_n\to\prod B_n/\bigoplus B_n$ be a topologically trivial isomorphism. By passing to a forcing extension, we may assume $\OCA_{\infty}$ and $\MA_{\aleph_1}$ are assumed without changing $\Lambda$. In fact, since $\Lambda$ is topologically trivial, it is coded by a $\Pi_2^1$ formula, and, by Shoenfield's absoluteness theorem (see Theorem~\ref{thm:absolute}), the code for $\Lambda$ induces a topologically trivial isomorphism, which we call $\Lambda$ for brevity. By Theorem~\ref{thm:redprod} there are an almost bijection $f\colon\NN\to\NN$, $n_0\in\NN$, a sequence $\epsilon_n\to 0$ and $\epsilon_n$-$^*$-isomorphisms $\phi_n\colon A_n\to B_{f(n)}$ for $n\geq n_0$ such that 
\[
\Lambda=\widetilde{\sum\phi_n}.
\] 
Since $\mathcal C$ is Ulam stable, we can find isomorphisms $\psi_n\colon A_n\to B_{f(n)}$ such that $\norm{\phi_n-\psi_n}\to 0$ s $n\to\infty$, hence 
\[
\Lambda=\widetilde{\sum\psi_n}
\]
and $\Lambda$ is algebraically trivial. By Theorem~\ref{thm:absolute}, $\Lambda$ is then algebraically trivial in the original model, as once again the code for $\Lambda$ it is coded by a $\Pi_2^1$ formula, and this will not be changed by forcing.

$2\Rightarrow 1$. We work by contradiction. Suppose that there is $\epsilon>0$ such that for all $n$ there are $A_n,B_n\in\mathcal C$ and $\frac{1}{n}$-$^*$-isomorphisms $\phi_n\colon A_n\to B_n$ such that for all $n$ there is no $\psi_n\colon A_n\to B_n$ for which $\norm{\phi_n-\psi_n}<\epsilon$. By Proposition~\ref{prop:borelredprod}, the map $\sum\phi_n$ induces a topologically trivial isomorphism 
\[
\widetilde{\sum\phi_n}\colon\prod A_n/\bigoplus A_n\to\prod B_n/\bigoplus B_n.
\]
Since by assumption $\tilde\phi$ is algebraically trivial, by Proposition~\ref{prop:redprodalgtriv} there are an almost bijection $f\colon\NN\to\NN$ and isomorphisms $\psi_n\colon A_n\to B_{f(n)}$ such that 
\[
\widetilde{\sum\psi_n}=\widetilde{\sum\phi_n}.
\]
Since $\widetilde{\sum\psi_n}$ and $\widetilde{\sum\phi_n}$ induce the same automorphism of the canonical copy of $\ell_\infty/c_0$ given by central projections (namely, the identity), we have that eventually $f(n)=n$, so $\psi_n\colon A_n\to B_n$. By hypothesis, for all $n$ there is a contraction $a_n\in A_n$ such that $\norm{\phi_n(a_n)-\psi_n(a_n)}>\epsilon$. With $a=(a_n)$ we have that
\[
\epsilon<\sup\norm{\phi_n(a_n)-\psi_n(a_n)}=\norm{\pi_B(\prod\phi_n(a)-\prod\psi_n(a))}\leq\norm{\widetilde{\sum\psi_n}-\widetilde{\sum\phi_n}}=0,
\]
a contradiction.
\end{proof}
\subsubsection*{Canonically isomorphic reduced products}
We now analyze Conjecture~\ref{conj:algtriviso}.
\begin{definition}
Two $\Cstar$-algebras $A$ and $B$ are said $\epsilon$-isomorphic if there an $\epsilon$-$^*$-isomorphism $\phi\colon A\to B$.
\end{definition}
Although this is not an equivalence relation, we have that
\begin{itemize}
\item  if $A$ and $B$ are $\epsilon$-isomorphic, then $B$ and $A$ are $2\epsilon$-isomorphic and,
\item  if $A$ and $B$ are $\epsilon$-isomorphic and $B$ and $C$ are $\delta$-isomorphic, then $A$ and $C$ are $\epsilon+\delta$-isomorphic.
\end{itemize}

\begin{question}\label{question:approx}
 Let $\mathcal C$ be a class of unital separable $\Cstar$-algebras. Is there $\epsilon>0$ such that for all $A,B\in\mathcal C$, if $A$ and $B$ are $\epsilon$-isomorphic then $A$ and $B$ are isomorphic?
\end{question}
Answers to Question~\ref{question:approx} are connected to the existence of isomorphic reduced products that are not algebraically isomorphic and to  a fundamental conjecture of Kadison and Kastler. We make these connections explicit.
 
\begin{theorem}\label{thm:ulam2}
Let $\mathcal C$ be a class of separable unital $\Cstar$-algebras with no central projections. The following are equivalent:
\begin{enumerate}[label=(\roman*)]
\item Question~\ref{question:approx} has a positive answer for the class $\mathcal C$;
\item if $A_n,B_n\in\mathcal C$ and $\prod A_n/\bigoplus A_n$ is isomorphic to $\prod B_n/\bigoplus B_n$ via a topologically trivial isomorphism, then they are isomorphic via an algebraically trivial isomorphism.
\end{enumerate}
\end{theorem}
\begin{proof}

(i)$\Rightarrow$(ii): Fix $\epsilon>0$ giving a positive answer to Question~\ref{question:approx} for the class $\mathcal C$. Let $A_n,B_n\in\mathcal C$ be as in (ii), and suppose that $\Lambda\colon\prod A_n/\bigoplus A_n\to\prod B_n/\bigoplus B_n$ is a topologically trivial isomorphism. By passing to a forcing extension where $\OCA_{\infty}$ and $\MA_{\aleph_1}$ hold, $\Lambda$ remains topologically trivial by Theorem~\ref{thm:absolute}. In such forcing extension, by Theorem~\ref{thm:redprod}, there are a sequence $\epsilon_n\to 0$, an almost permutation $f\colon\NN^\NN$ and $\epsilon_n$-$^*$-isomorphisms $\phi_n\colon A_n\to B_{f(n)}$. Without loss of generality we may assume that $\epsilon_n<\epsilon$, so the algebras $A_n$ and $B_{f(n)}$ are isomorphic (for all large enough $n$). Fix an isomorphism $\psi_n\colon A_n\to B_{f(n)}$. Then $\sum\psi_n$ induces an algebraically trivial isomorphism $\prod A_n/\bigoplus A_n\to\prod B_n/\bigoplus B_n$. Such an isomorphism is algebraically trivial in the forcing extension, and again by Theorem~\ref{thm:absolute}, in the ground model.

(ii)$\Rightarrow$(i): Let $\mathcal C$ be a class as in the hypothesis, assume (ii) and the negation of (i). We proceed by cases. 

\underline{CASE 1:}
There are nonisomorphic $A$ and $B$ in $\mathcal C$ such that $A$ is $\frac{1}{n}$-isomorphic to $B$ for all $n\in\NN$. Let $A_n=A$ and $B_n=B$ for all $n$, and let $\phi_n\colon A_n\to B_n$ be a $1/n$-$^*$-isomorphism. By Proposition~\ref{prop:borelredprod}, $\sum\phi_n$ induces a topologically trivial isomorphism between $\ell_\infty(A)/c_0(A)$ and $\ell_\infty(B)/c_0(B)$. By (ii), there is then an algebraically trivial isomorphism between $\ell_\infty(A)/c_0(A)$ and $\ell_\infty(B)/c_0(B)$. By Proposition~\ref{prop:redprodalgtriv}, $A\cong B$, which is a contradiction.

If Case 1 does not apply, the strategy is to produce from the negation of (i) a sequence of pairwise nonisomorphic algebras $C_n\in \mathcal C$ with the property that $C_{2n}$ is $\frac{1}{n}$-isomorphic to $C_{2n+1}$ for all $n\in\NN$. By applying Proposition~\ref{prop:borelredprod}, $\prod C_{2n}/\bigoplus C_{2n}$ and $\prod C_{2n+1}/\bigoplus C_{2n+1}$ are then isomorphic via a topologically trivial isomorphism. The existence of an algebraically trivial isomorphism between $\prod C_{2n}/\bigoplus C_{2n}$ and $\prod C_{2n+1}/\bigoplus C_{2n+1}$, together Proposition~\ref{prop:redprodalgtriv}, contradicts the fact that the $C_n$'s are pairwise nonisomorphic. 

\underline{CASE 2:} There are $A\in \mathcal C$ and pairwise nonisomorphic $B_n\in \mathcal C$ such that $A$ is $\frac{1}{n}$-isomorphic to $B_n$. The algebras $B_{2n}$ and $B_{2n+1}$ are then $(\frac{1}{2n}+\frac{2}{2n+1})$-isomorphic. Setting $C_{2n}=B_{4n}$ and $C_{2n+1}=B_{4n+1}$ gives the required sequence.

\underline{CASE 3:} There are nonisomorphic $A_n,B_n\in\mathcal C$ such that $A_n$ is $\frac{1}{n}$-isomorphic to $B_n$ but $A_n$ and $B_n$ are not isomorphic for all $n$. The algebras $\prod A_n/\bigoplus A_n$ and $\prod B_n/\bigoplus B_n$ are then isomorphic via a topologically trivial isomorphism by Proposition~\ref{prop:borelredprod}, and therefore via an  algebraically isomorphism. By Proposition~\ref{prop:redprodalgtriv} for all (large enough) $n$ there is $f(n)$ such that $A_n$ and $B_{f(n)}$ are isomorphic, where $f\in\NN^\NN$ is an almost permutation. In particular $B_n$ and $B_{f(n)}$ are $\frac{1}{n}$-isomorphic. Note that $n\neq f(n)$ for all $n$, and we can assume that $f(n)\neq f(n')$ whenever $n\neq n'$.  Let $n_k$ be a subsequence such that $n_{k+1}, f(n_{k+1})> \max\{n_k,f(n_k)\}$ for all $k\in\NN$. The algebras $C_{2k}=B_{n_k}$ and $C_{2k+1}=B_{f(n_k)}$ give the required sequence.
\end{proof}
If $A,B\subseteq\mathcal B(H)$ one defines the Kadison-Kastler distance between $A$ and $B$ as
\[
d_{KK}(A,B)=\max\{\sup_{a\in A, \norm{a}\leq 1}\inf_{b\in B}\norm{a-b},\sup_{b\in B, \norm{b}\leq 1}\inf_{a\in A}\norm{a-b}\}.
\]
This is the Hausdorff distance between the unit balls of $A$ and $B$ as subsets of the unit ball of $\mathcal B(H)$. The following two conjectures are implicit in \cite{KadKast}:
\begin{conjecture}\label{conj:KK}
\begin{enumerate}
\item There is $\epsilon>0$ such that if $A,B\subseteq\mathcal B(H)$ are separable $\Cstar$-algebras such that $d_{KK}(A,B)<\epsilon$ then $A\cong B$.
\item Let $A\subseteq\mathcal B(H)$ be a separable $\Cstar$-algebra. Then there is $\epsilon>0$ such that if $B\subseteq\mathcal B(H)$ is a $\Cstar$-algebra such that $d_{KK}(A,B)<\epsilon$ then $A\cong B$. 
\end{enumerate}
\end{conjecture}
Conjecture~\ref{conj:KK}(1) is a uniform version of (2). In \cite{CSSWW} 1. was confirmed in case $A$ (or $B$) is nuclear. The class of nuclear algebras is the largest one for which  2. was confirmed. 
\begin{theorem}\label{thm:KK}
Let $\mathcal C$ be a class of separable $\Cstar$-algebras. Then a positive answer to Question~\ref{question:approx} for the class $\mathcal C$ implies that Kadison and Kastler conjecture holds for algebras in the class $\mathcal C$.
\end{theorem}
\begin{proof}
Let $\epsilon>0$ be giving a positive answer to Question~\ref{question:approx} for the class $\mathcal C$. If $A,B\subseteq\mathcal B(H)$ are elements of $\mathcal C$ such that $d_{KK}(A,B)<\epsilon/2$, the Axiom of Choice gives an $\epsilon$-$^*$-isomorphism $A\to B$. By hypothesis, $A\cong B$.
\end{proof}
\begin{remark}\label{remark:comments}
\begin{itemize}
\item For certain classes of $\Cstar$-algebras Question~\ref{question:approx} has a positive solution. This happens, for example, for (separable) AF algebras (\cite[Corollary 3.2]{MKAV.UC}), abelian algebras (see Theorem~\ref{thm:semrl}) and Kircherg UCT algebras (by showing that if $A$ and $B$ are $\epsilon$-$^*$-isomorphic then they have the same $K$-theory). The most interesting class for which we do not know whether Question~\ref{question:approx} has a positive solution is the class of nuclear separable $\Cstar$-algebras. Evidences in the Kadison-Kastler perturbation setting (\cite{CSSWW}) suggest that a positive solution should be expected.
 For a more detailed treatment of what properties are preserved by  approximate maps, and of how it relates to isomorphisms of reduced products, see \cite[\S3 and \S6]{MKAV.FA}.
 \item By the Axiom of Choice, if $A$ and $B$ have faithful representations $\rho_A$ and $\rho_B$ on the same Hilbert space $H$ with the property that $d_{KK}(\rho_A[A],\rho_b[B])\leq\epsilon$, they are $2\epsilon$-isomorphic. The converse, whether there is $\epsilon>0$ such that two $\epsilon$-isomorphic $\Cstar$-algebras can be represented faithfully on a Hilbert space in such a way the images of the representations are Kadison-Kastler close to each other, is yet open even for natural classes. (This is true, in the nuclear setting, if the $\epsilon$-isomorphism is linear, by a result of Johnson, \cite{Johnson.AMNM}).
\item Topological triviality is needed as an assumption in Theorem~\ref{thm:ulam} and~\ref{thm:ulam2}. In fact,  if $A_n$ are separable unital $\Cstar$-algebras, countable saturation of reduced products (\cite{Farah-Shelah.RCQ}) implies that under $\CH$ there are $2^{2^{\aleph_0}}$ automorphisms of $\prod A_n/\bigoplus A_n$, and therefore automorphisms that are not topologically trivial exist. More than that, a Feferman-Vaught-like result of Ghasemi (\cite{Ghasemi.FFV}) proves the following statement: $\CH$ implies that whenever $A_n$ are unital separable $\Cstar$-algebra there is a strictly increasing sequence $\{n_k\}$ such that if $S,T\subseteq\NN$ are infinite then 
\[
\prod_{k\in S}A_{n_k}/\bigoplus_{k\in S}A_{n_k}\cong \prod_{k\in T}A_{n_k}/\bigoplus_{k\in T}A_{n_k}.
\]
If the algebras $A_n$ are pairwise nonisomorphic and have no central projections, all these isomorphisms cannot be (topologically or algebraically) trivial.
\item The assumption that the algebras involved do not have central projections is merely used to make the statements in \S\ref{subsec:redprod} less technical. Since a projection that is close to a central projection must be equal to it, one can rephrase, by exponentiating the technical statements and with proper care, equivalents of the results obtained in this subsection for algebras with central projections. We decide not to do so and leave this to the interested reader.
\end{itemize}
\end{remark}

\subsection{Coronas of stabilizations of unital algebras}\label{subsec:stab}
In this section we study Conjecture~\ref{conj:algtriviso} for algebras of the form $A\otimes\mathcal K$, where $A$ is a unital separable $\Cstar$-algebra and $\mathcal K$ denotes the algebra of compact operators on a separable Hilbert space. Fix an increasing approximate identity of finite-rank projections $p_n\in\mathcal K$, with the property that each $p_{n+1}-p_n$ is a projection of rank 1. Then $\{1\otimes p_n\}_n$ is an increasing approximate identity for $A\otimes\mathcal K$. For $S\subseteq \NN$, let 
\begin{equation}\label{eqnxx}
q^A_S=\sum_{n\in S}(1_A\otimes(p_{n+1}-p_n))\in\mathcal M(A\otimes\mathcal K).
\end{equation}
 If an ideal $I\subseteq A\otimes\mathcal K$ contains $q^A_{\{n\}}$, for some $n$, then it contains $q^A_{\{n\}}$ for all $n$, and so $I=A\otimes\mathcal K$. For this reason, if $a\in\mathcal M(A\otimes\mathcal K)$ is such that $1-a\in A\otimes\mathcal K$, then the hereditary algebra $\overline{a(A\otimes\mathcal K)a}$ is full\footnote{A $\Cstar$-algebra $B\subseteq A$ is full if it generates $A$ as an ideal} in $A\otimes \mathcal K$. Brown's Theorem (\cite[Theorem 2.8]{Brown:stabher}) gives the following:
 \begin{proposition}
 Let $A$ and $B$ be unital separable $\Cstar$-algebras. Then $\mathcal Q(A\otimes\mathcal K)$ and $\mathcal Q(B\otimes\mathcal K)$ are algebraically isomorphic if and only if $A\otimes\mathcal K\cong B\otimes\mathcal K$.
 \end{proposition}
The following is \cite[Conjecture 9.1]{MKAV.FA}.
\begin{conjecture}\label{conj:canonicstable}
Assume $\OCA+\MA_{\aleph_1}$. Let $A$ and $B$ be unital separable $\Cstar$-algebras. Then $\mathcal Q(A\otimes \mathcal K)\cong \mathcal Q(B\otimes\mathcal K)$ if and only if $A\otimes\mathcal K\cong B\otimes\mathcal K$.
\end{conjecture}

\begin{remark}
Some instances of the conjecture are true with no set theoretical assumption when $A$ and $B$ are chosen in a specific class of $\Cstar$-algebras. Elliott showed in \cite[Theorem 1]{Elliott.Der2} that if $A$ and $B$ are separable unital and UHF then the conjecture is true. In general, one has that $\mathcal Q(A\otimes \mathcal K)\cong \mathcal Q(B\otimes\mathcal K)$ if and only if $A\otimes\mathcal K\cong B\otimes\mathcal K$ in $\ZFC$ when $A$ and $B$ are chosen in a class of simple algebras which is classifiable by $K$-theory\footnote{A class of algebras is classifiable by $K$-theory is $A\cong B$ if and only if $K_0(A)\cong K_0(B)$ and $K_1(A)\cong K_1(B)$ for all $A,B\in\mathcal C$. For an introduction of the $K_0$ and $K_1$ groups of a $\Cstar$-algebras, see~\cite[\S V.1]{Blackadar.OA}}. This is because the $K$-theory of $\mathcal Q(A\otimes\mathcal K)$ computes the $K$-theory of $A$ (see \cite[Proposition 12.2.1]{Blackadar.KT}). Classes of unital $\Cstar$-algebras that are classifiable by $K$-theory are the one of unital separable simple AF algebras (\cite{Elliott.ClassAF}) and the one of unital Kirchberg algebras satisfying the Universal Coefficient Theorem (\cite{KirchPhil}).
\end{remark}

The rest of the section is dedicated to show the following:
\begin{theorem}\label{thm:stabofab}
Assume $\OCA+\MA_{\aleph_1}$. Let $X$ and $Y$ be compact connected metrizable  topological spaces. Then $\mathcal Q(C(X)\otimes\mathcal K)\cong\mathcal Q(C(X)\otimes\mathcal K)$ if and only if $C(X)\cong C(Y)$.
\end{theorem}
Before moving to the proof of Theorem~\ref{thm:stabofab}, we need some preparatory work. The following is Lemma 2.6 in \cite{MKAV.FA}.
\begin{lemma}\label{lemma:MKAV}
 Let $A$ be a $\Cstar$-algebra with an increasing countable approximate identity of positive contractions $\{e_n\}$ with, for all $n$, $e_n e_{n+1} = e_n$. Given an interval $I\subseteq\NN$, write $e_I = e_{\max(I)} - e_{\min(I)}$. Let $t\in\mathcal M(A)$. Then there are finite intervals $I_n^i\subseteq \NN$, for each $n\in\NN$ and $i = 0,1$, and $t_0$ and $t_1$ in $\mathcal(A)$, such that for each $i\in\{0,1\}$,
 \begin{enumerate}
 \item the intervals $I_n^i$, for $n\in\NN$, are pairwise disjoint and consecutive,
 \item $t_i$ commutes with $e_{I_n^i}$ for each $n\in\NN$, and
 \item $t - (t_0 + t_1)\in A$.
 \end{enumerate}
\end{lemma}

 An important consequence of Johnson-Parrott's theorem is that the canonical copy of $\ell_\infty/c_0$ in the Calkin algebra $\mathcal Q(H)$ is a maximal abelian subalgebra. We generalise this fact.

\begin{lemma}\label{lemm:commutingwithskeleton}
Let $A$ be a $\Cstar$-algebra with an increasing approximate identity of projections $p_n$, and let $q_S=\sum_{n\in S}(p_n-p_{n-1})\in\mathcal M(A)$ for $S\subseteq\NN$. If a positive $p\in\mathcal Q(A)$ commutes with $\pi_A(q_S)$ for all $ S\subseteq\NN$ then there is a positive $q$ with $q=\sum_n q_{\{n\}}qq_{\{n\}}$ such that $\pi_{A}(q)=p$. 
\end{lemma}
\begin{proof}
Let $q'$ be such that $\pi_A(q')=p$.
\begin{claim}\label{claim:comm}
For every $\epsilon>0$ there is $n_0\in\NN$ such that if $F\subseteq\NN$ is finite and $\min F>n_0$ then $\norm{q'q_F-q_Fq'}<\epsilon$.
\end{claim}
\begin{proof}
By contradiction fix $\epsilon>0$ and a sequence $F_n$ with $\max F_n<\min F_{n+1}$ such that $\norm{q'q_{F_n}-q_{F_n}q'}\geq\epsilon$ for all $n$. Since $q'q_{F_n}-q_{F_n}q'\in A$, there is a finite $J_n\subseteq\NN$ with $F_n\subseteq J_n$ and $\norm{q_{J_n}q'q_{F_n}-q'q_{F_n}}, \norm{q_{F_n}q'-q_{F_n}q'q_{J_n}}<2^{-n}$. Note that $J_n$ can be chosen so that $\min J_n\to\infty$ as $n\to\infty$, and that 
\[
\norm{q_{J_n}q'q_{F_n}q_{J_n}-q_{J_n}q_{F_n}q'q_{J_n}}>\epsilon/2. 
\]
By passing to a subsequence we can assume that $\max J_n+1<\min J_{n+1}$. Let $G=\bigcup F_n$. Then 
\[
\pi_A(q_Gq'-q'q_G)=\pi_A(\sum_n (q_{J_n}q_{F_n}q'q_{J_n}-q_{J_n}q'q_{F_n}q_{J_n}))
\]
Since when $n\neq m$ we have $J_n\cap J_m=\emptyset$, then 
\[
\norm{\pi_A(q_Gq'-q'q_G)}=\limsup_n\norm{q_{J_n}q_{F_n}q'q_{J_n}-q_{J_n}q'q_{F_n}q_{J_n}}>\epsilon/2, 
\]
a contradiction.
\end{proof}

By Lemma~\ref{lemma:MKAV}, we can find $\mathbb J=\langle J_n\colon n\in\NN\rangle$, a partition of $\NN$ into finite intervals, such that $q'-q_0'-q_1'\in A$ where
\[
q_0'=\sum_n (q_{J_{2n}}q'q_{J_{2n}}+q_{J_{2n+1}}q'q_{J_{2n}}+q_{J_{2n}}q'q_{J_{2n+1}})
\]
and 
\[
q_1'=\sum (q_{J_{2n+1}}q'q_{J_{2n+1}}+q_{J_{2n+2}}q'q_{J_{2n+1}}+q_{J_{2n+1}}q'q_{J_{2n+2}}).
\]
\begin{claim}
$q'-\sum_n q_{J_{n}}q'q_{J_{n}}\in A$.
\end{claim}
\begin{proof}
Let $s_n=q_{J_{2n+1}}q'q_{J_{2n}}+q_{J_{2n}}q'q_{J_{2n+1}}$. Then $s_ns_m=0$ when $n\neq m$, so $\norm{\pi_A(\sum s_n)}=\limsup\norm{s_n}$. As $\min J_n\to\infty$ when $n\to\infty$, we have that $\norm{s_n}\to 0$ as $n\to\infty$ by Claim~\ref{claim:comm}. This shows that $\norm{\pi_A(\sum_n (q_{J_{2n+1}}q'q_{J_{2n}}+q_{J_{2n}}q'q_{J_{2n+1}}))}=0$. A similar argument shows that $\norm{\pi_A(q_{J_{2n+2}}q'q_{J_{2n+1}}+q_{J_{2n+1}}q'q_{J_{2n+2}})}=0$, and therefore $q_0'-\sum_n q_{J_{2n}}q'q_{J_{2n}}\in A$, and $q_1'-\sum_n q_{J_{2n+1}}q'q_{J_{2n+1}}\in A$.
\end{proof}
We want to show that
\[
\sum_{n}q_{J_n}q'q_{J_n}-\sum_n\sum_{i\in J_n}q_{\{i\}}q'q_{\{i\}}\in A.
\]
Note that 
\[
\norm{\pi_A(\sum_{n}q_{J_n}q'q_{J_n}-\sum_n\sum_{i\in J_n}q_{\{i\}}q'q_{\{i\}})}=\limsup_n\norm{q_{J_n}q'q_{J_n}-\sum_{i\in J_n}q_{\{i\}}q'q_{\{i\}}}
\]
since $q_{J_n}q_{J_m}=0$ whenever $n\neq m$.
Suppose then that for infinitely many $n$ we have 
\[
\norm{q_{J_n}q'q_{J_n}-\sum_{i\in J_n}q_{\{i\}}q'q_{\{i\}}}>\epsilon/2.
\]
Since $q_{J_n}q'q_{J_n}-\sum_{i\in J_n}q_{\{i\}}q'q_{\{i\}}=\sum_{i,j\in J_n, i\neq j}q_{\{i\}}q'q_{\{j\}}$, we can find a nonempty $G_n\subseteq J_n$ such that $\norm{q_{G_n}q'q_{J_n\setminus G_n}}>\epsilon/8$. Let $n_0$ be as given by Claim~\ref{claim:comm} for $\epsilon/32$. Then 
\[
\norm{q_{G_n}q'q_{J_n\setminus G_n}}\leq \epsilon/32+\norm{q'q_{G_n}q_{J_n\setminus G_n}}=\epsilon/32,
\]
a contradiction. This conclude the proof of the lemma.
\end{proof}

A projection $p$ in a $\Cstar$-algebra $A$ is abelian if $pAp$ is abelian. By $M_n$ we denote the algebra of complex $n\times n$ matrices. If $A$ is a $\Cstar$-algebra, $M_n(A)$ is the algebra of $A$-valued $n\times n$ matrices. This is canonically isomorphic to $M_n\otimes A$.
\begin{proposition}
Let $X$ be a connected compact space and let $n\in\NN$. Let $p\in M_n(C(X))$ be a projection. The following are equivalent
\begin{enumerate}
\item\label{s1} $p$ is abelian,
\item\label{s2}  $pM_n(C(X))p\cong C(X)$,
\item\label{s3} the only projections in $pM_n(C(X))p$ are $p$ and $0$.
\end{enumerate}
\end{proposition}
\begin{proof}
By identifying $M_n(C(X))$ with the algebra of continuous functions $X\to M_n$, we can see that a projection $p$ is abelian if and only if $p(x)$ is rank $1$ for all $x\in X$. This in turn happens if and only if $p(x)M_np(x)\cong \mathbb C$ for every $x\in X$, and therefore if and only if the only projections in $pM_n(C(X))p$ are $p$ and $0$, as $X$ is connected. So (\ref{s1}) and (\ref{s3}) are equivalent. 

For (\ref{s1}) implies (\ref{s2}), note that the function $C(X)\to pM_n(C(X))p$ given by $f\mapsto (f(x)p(x))_{x\in X}$ provides the necessary isomorphism. Since  (\ref{s2}) implies (\ref{s1}) is obvious, the proof is complete.
\end{proof}

\begin{proof}[Proof of Theorem~\ref{thm:stabofab}]
We want to prove, under the appropriate set theoretical assumptions, that for $X$ and $Y$ connected compact metrizable spaces, if $\mathcal Q(C(X)\otimes\mathcal K)$ and $\mathcal Q(C(Y)\otimes \mathcal K)$ are isomorphic, then $C(X)$ and $C(Y)$ must be isomorphic.

Fix compact connected metrizable spaces $X$ and $Y$. Let $A=C(X)\otimes\mathcal K$ and $B=C(Y)\otimes\mathcal K$.  Let $\pi_A$ (similarly $\pi_B$) be the quotient map $\mathcal M(A)\to\mathcal Q(A)$. For $S\subseteq\NN$, let $q_S^X$ and $q_S^Y$ be chosen as in Equation~\eqref{eqnxx} for $C(X)$ and $C(Y)$, that is
\[
q^X_S=\sum_{n\in S}(1_{C(X)}\otimes(p_{n+1}-p_n))\in\mathcal M(A)
\]
and 
\[
q^Y_S=\sum_{n\in S}(1_{C(Y)}\otimes(p_{n+1}-p_n))\in\mathcal M(B)
\]
where each $p_n$ is a rank $1$ projection in $\mathcal K$. 

Fix an isomorphism 
\[
\Lambda\colon\mathcal Q(A)\to\mathcal Q(B).
\]
Let $J_n=[12n,12n+3)$. Notice that $\max J_n+6<\min J_{n+1}$, hence we can apply Theorem~\ref{thm:almostredprod} and find orthogonal maps $\phi_n$ and finite sets $I_n\subseteq \NN$ with $\min I_n\to\infty$ where
\[
\phi_n\colon q_{[12n,12n+3)}^XAq_{[12n,12n+3)}^X\to q_{I_n}^YBq_{I_n}^Y\cong M_{|I_n|}(C(Y)),
\]
and each $\phi_n$ is an $\epsilon_n$-monomorphism where $\epsilon_n\to 0$ as $n\to\infty$. Since each $I_n$ is finite and $\min I_n\to\infty$, there is an infinite set $D\subseteq\NN$ such that if $n<n'\in D$ then $\max I_n<\min I_{n'}$ and for all $n\in D$, $\epsilon_n<2^{-n}$. Notice that the choice of $\phi_n$ was made, again by Theorem~\ref{thm:almostredprod}, so that $\sum_{n\in D}\phi_n$ is a lift of $\Lambda$ on $\prod_{n\in D} A_{[12n,12n+3)}$.
 
As $\phi_n$ is an $\epsilon_n$-monomorphism, we have that $\phi_n(q_{\{12n\}}^X)$ is a $2\varepsilon_n$-projection, that is, for each $n\in\mathbb D$ we have that 
\[
\norm{\phi_n(q_{\{12n\}}^X)^2-\phi_n(q_{\{12n\}}^X)}, \norm{\phi_n(q_{\{12n\}}^X)^*-\phi_n(q_{\{12n\}}^X)}<2\epsilon_n.
\] 
By a standard functional calculus argument, if $\varepsilon<\frac{1}{10}$ and $p$ is an $\varepsilon$-projection, one can find a projection $q$ with $\norm{p-q}<3\varepsilon$ lying in the same algebra as $p$ does. Therefore, we can find, for $n\in D$, projections $r_n\in q_{I_n}^YBq_{I_n}^Y$ such that $\norm{r_n-\phi_n(q_{\{12n\}}^X)}<2^{-n+1}$. As $\sum_{n\in D}\phi_n$ is a lift for $\Lambda$ on $\prod_{n\in D} A_{[12n,12n+3)}$ we have that 
\[
\pi_B(\sum_{n\in D} r_n)=\Lambda(\sum_{n\in D} q_{\{12n\}}^X).
\]
\begin{claim}
$r_n$ is an abelian projection for almost all $n\in D$.
\end{claim}
\begin{proof}
Suppose not. Let $\{n_k\}\subseteq D$ be an infinite set such that $r_{n_k}Br_{n_k}$ has a nonzero projection $q_k<r_{n_k}$ for each $k$. Since the projections $q_{n_k}$ are orthogonal, as $I_{n_k}$ and $I_{n_{k'}}$ are disjoint, the element $\sum_k q_k$ is a well-defined element of $\mathcal M(B)$. Notice that 
\[
\Lambda^{-1}(\pi_B(\sum_k q_k))\neq \pi_A(q_S^X)
\]
 for all $S\subseteq\NN$. On the other hand, $\Lambda^{-1}(\pi_B(\sum q_k))$ commutes with $\{\pi_A(q_S^X)\colon S\subseteq\NN\}$ since $q_k$ commutes with $r_{n_k}$. 
 
Let $p'\in\mathcal M(A)$ such that $\pi_A(p')=\Lambda^{-1}(\pi_B(\sum_k q_k))$. Since $\pi_A(p')$ commutes with $\{\pi_A(q_S^X)\colon S\subseteq\NN\}$, by Lemma~\ref{lemm:commutingwithskeleton}, we can assume that $p'=\sum_n q_{\{n\}}^Xp'q_{\{n\}}^X$. Since $\pi_A(p')$ is a projection, $q_{\{n\}}^Xp'q_{\{n\}}^X$ is a $\varepsilon_n$-projection for a sequence $\varepsilon_n\to 0$. Moreover, since $\pi_A(p')\neq \pi_A(q_S^X)$ for all $S\subseteq\NN$, we have that for infinitely many $n\in\NN$, 
 \[
 \norm{q_{\{n\}}^Xp'q_{\{n\}}^X},\norm{q_{\{n\}}^Xp'q_{\{n\}}^X-q_{\{n\}}^X}>\frac{1}{2}.
 \]
 In particular, we can find  projections $s_n\in q_{\{n\}}^XAq_{\{n\}}^X$ such that 
 \[
\norm{ s_n-q_{\{n\}}^Xp'q_{\{n\}}^X}\to 0
 \]
 as $n\to\infty$. Since all the $s_n$'s (and all the $q_{\{n\}}^Xp'q_{\{n\}}^X$) are pairwise orthogonal, then 
 \[
 \norm{\pi_A(\sum s_n-\sum q_{\{n\}}^Xp'q_{\{n\}}^X)}=\limsup_n\norm{s_n-q_{\{n\}}^Xp'q_{\{n\}}^X}=0.
 \]
Also, as $\norm{\pi_A(\sum s_n)}=1$ we have that 
\[
1=\norm{\pi_A(\sum s_n)}=\limsup_m\norm{\sum_{n\geq m}s_n}=\limsup_n\norm{s_n}.
\]
(The last equality follows from that all the $s_n$'s are orthogonal). In particular there is arbitrarily large $n$ such that $s_n$ is a nontrivial projection in $q_{\{n\}}^XAq_{\{n\}}^X$. Hence $s_n$ is a nontrivial projection which is differently from $q_{\{n\}}^X$ and belongs to $q_{\{n\}}^XAq_{\{n\}}\cong C(X)$, a contradiction to that $X$ is connected.
\end{proof}
Fix $n\in D$. The map $\phi'_n\colon q_{\{12n\}}^XAq_{\{12n\}}^X\to r_nBr_n$ given by $\phi'_n(f)=r_n\phi_n(f)r_n$ is a $2^{-n+2}$-$^*$-monomorphism. Since the domain of $\phi'_n$ is isomorphic to $C(X)$, and the codomain is isomorphic to $C(Y)$, by Theorem~\ref{thm:semrl} we can find a (necessarily unital) injective $^*$-homomorphism $\tilde\phi_n\colon q_{\{12n\}}^XAq_{\{12n\}}^X\to r_nBr_n$ such that $\norm{\phi'_n-\tilde\phi'_n}\to 0$ as $n\to \infty$. As 
\[
\norm{\phi'_n-\phi_n\restriction q_{\{12n\}}^XAq_{\{12n\}}^X}\to 0\text{ and }\norm{\phi'_n-\tilde\phi'_n}\to 0,\] as $n\to\infty$, we have that $\tilde\Phi=\sum\tilde\phi_n$ is a lift of $\Lambda$ on $\prod q_{\{12n\}}^XAq_{\{12n\}}^X$.
\begin{claim}
There is $n_0$ such that $\tilde\phi_n$ is surjective for all $n\geq n_0$.
\end{claim}
\begin{proof}
If not, for infinitely many $n\in D$ there are positive contractions $c_n\in r_nBr_n$ such that the distance between $c_n$ and the image of $\tilde\phi_n(q_{\{12n\}}^XAq_{\{12n\}}^X)$ is $1$. Let $c=\sum c_n$. The element $\Lambda^{-1}(c)$ then is commuting with $\pi_{A}(q_S^X)$ for all $S\subseteq\NN$, but by Lemma~\ref{lemm:commutingwithskeleton}, it cannot be the $\pi_A$-image of $d$ where $d=\sum q_{\{n\}}^Xdq_{\{n\}}^X$, a contradiction.
\end{proof}
Fix $n\geq n_0$ and $n\in D$ where $n_0$ is given by the claim. Then $\tilde\phi_{n}$ is an isomorphism between $q_{\{n\}}^XAq_{\{n\}}^X\cong C(X)$ and $r_nBr_n\cong C(Y)$. This is the thesis.
\end{proof}

\begin{remark}
\begin{itemize}
\item We believe that an analysis of the proof of Theorem~\ref{thm:stabofab}, specifically of the application of Theorem~\ref{thm:semrl}, combined with the proof that all automorphisms of $\mathcal Q(H)$ are inner, can show that under Forcing Axioms all automorphisms of $\mathcal Q(C(X)\otimes\mathcal K)$ are algebraically trivial whenever $X$ is compact and metrizable. We predict that this proof would be slightly technical, and that it would not add content to this paper\footnote{Unless the statement is false!}. For this reason, we do not to include it.
\item Differently than in the reduced product's case, corona of stabilizations of unital $\Cstar$-algebras are rarely countably saturated (that $\mathcal Q(H)$ is not countably saturated was shown in \cite{Farah-Hart.CS}). For this reason, it is very difficult to produce from $\CH$ and a Cantor back-forth argument an isomorphism between coronas of different algebras. We do not know whether it is consistent with $\ZFC$ that there are two unital algebras $A$ and $B$ such that $\mathcal Q(A\otimes\mathcal K)$ and $\mathcal Q(B\otimes\mathcal K)$ are isomorphic but not via a topologically trivial isomorphism.
\item Similarly to \S\ref{subsec:redprod} (an in particular Theorems~\ref{thm:ulam} and ~\ref{thm:ulam2}), we can state and prove equivalent formulations of Conjectures~\ref{conj:algtriv} and~\ref{conj:algtriviso} for algebras of the form $\mathcal Q(A\otimes\mathcal K)$ in terms of statements regarding Ulam stability and the study of approximate maps.
\item We can state the correspondent of Question~\ref{question:approx} for, and generalise Definition~\ref{defin:ulam} to, embeddings instead that for isomorphisms. We leave the relevant reformulations to the reader. 
\end{itemize}
\end{remark}

\bibliographystyle{plain}
\bibliography{library}
\end{document}